\documentclass[a4paper,10pt]{article}

\usepackage{versions}
\excludeversion{Ex}
\includeversion{I}
\usepackage{authblk}
\usepackage[version=4]{mhchem}
\usepackage[ngerman,english]{babel}
\usepackage[latin1]{inputenc}
\usepackage{csquotes}

\usepackage{amsfonts,amsmath,amsthm}
\usepackage{empheq}

\usepackage[titletoc,title]{appendix}
\usepackage[backend=bibtex8,doi=false,eprint=false,giveninits=true,isbn=false,style=numeric-comp,maxnames=99]{biblatex}
\makeatletter
\def\blx@maxline{77}
\makeatother

\usepackage{cases}
\usepackage{mathabx}
\usepackage{bbm}
\usepackage{xfrac}

\bibliography{bibl.bib}
\AtEveryBibitem{\clearlist{language}}
\usepackage{enumitem}
\setlist[enumerate,1]{label=(\arabic*), ref =(\arabic*)} \setlist[enumerate,2]{label=(\roman*), ref =(\roman*)}

\usepackage{fancyhdr}

\usepackage{color}
\usepackage[colorlinks]{hyperref}
\definecolor{blue75}{rgb}{0,0,.75}
\definecolor{green75}{rgb}{0,.75,0}
\hypersetup{colorlinks=true, urlcolor=blue75,linkcolor=blue75,citecolor=green75,pdfstartview=FitB,bookmarksopen=true,bookmarksopenlevel=1}

\usepackage{graphicx}

\usepackage[a4paper, left=2.5cm, right=2.5cm, top=2.5cm,bottom=2cm]{geometry}

\graphicspath{{images/} }
\usepackage{wrapfig}
\usepackage{figbib}
\allowdisplaybreaks
\usepackage[capitalise]{cleveref}

\crefdefaultlabelformat{{\it #2#1#3}}
\crefname{equation}{}{}
\crefname{enumi}{}{}
\creflabelformat{enumi}{#2{\it#1}#3}
\crefname{section}{{\it Section}}{{\it Sections}}
\crefname{subsection}{{\it Subsection}}{{\it Subsections}}
\crefname{subsubsection}{{\it Paragraph}}{{\it Paragraphs}}
\numberwithin{equation}{section}

\newtheorem{Theorem}{Theorem}[section]
\crefname{Theorem}{{\it Theorem}}{{\it Theorems}}
\newtheorem{Definition}[Theorem]{Definition}
\crefname{Definition}{{\it Definition}}{{\it Definitions}}
\newtheorem{Lemma}[Theorem]{Lemma}
\crefname{Lemma}{{\it Lemma}}{{\it Lemmas}}

\crefname{Proposition}{{\it Proposition}}{{\it Propositions}}

\newtheorem{Notation}[Theorem]{Notation}
\crefname{Notation}{{\it Notation}}{{\it Notations}}
\theoremstyle{definition}
\newtheorem{Remark}[Theorem]{Remark}
\crefname{Remark}{{\it Remark}}{{\it Remarks}}
\newtheorem{Example}[Theorem]{Example}
\crefname{Example}{{\it Example}}{{\it Examples}}

\renewbibmacro*{doi+eprint+url}{%
    \printfield{doi}%
    \newunit\newblock%
    \iftoggle{bbx:eprint}{%
        \usebibmacro{eprint}%
    }{}%
    \newunit\newblock%
    \iffieldundef{doi}{%
        \usebibmacro{eprint}}%
        {}%
    }
    
\newtheorem{Step}{Step}
\makeatletter
\@addtoreset{Step}{Theorem}
\makeatother

\usepackage{constants}
\newcommand{\parenthezises}[1]{\arabic{#1}}
\newconstantfamily{C}{
symbol=C,
format=\parenthezises,
}
\newconstantfamily{EE}{
symbol=\E,
format=\parenthezises,
}
\newconstantfamily{R}{
symbol=R,
format=\parenthezises
}
\newconstantfamily{SS}{
symbol=\SM,
format=\parenthezises,
}
\newconstantfamily{T}{
symbol=T,
format=\parenthezises,
}
\newconstantfamily{U}{
symbol=U,
format=\parenthezises,
}
\newconstantfamily{UU}{
symbol=\U,
format=\parenthezises,
}
\newconstantfamily{W}{
symbol=W,
format=\parenthezises,
}
\newconstantfamily{WW}{
symbol=\W,
format=\parenthezises,
}
\newconstantfamily{Rho}{
symbol=\rho,
format=\parenthezises,
}

\newcommand{\cmg}[1]{}

\newcommand{\eps}{\varepsilon}
\newcommand{\ve}{\varepsilon}
\newcommand{\one}{\mathbf{1}}

\newcommand{\wu}{\widehat{u}}
\newcommand{\uem}{u_{\eps m}}
\newcommand{\wuem}{\widehat u_{\eps m}}
\newcommand{\ww}{\widehat{w}}

\newcommand{\Br}{B_{\rho}}
\newcommand{\Bro}{B_{\rho_0}}
\newcommand{\oB}{\overline{\Br}}
\newcommand{\oBo}{\overline{\Bro}}

\newcommand{\E}{\mathcal E}
\newcommand{\G}{{\mathcal G}} %
\newcommand{\HH}{{\mathcal H}} %
\newcommand{\LL}{{\mathcal L}} %
\newcommand{\K}{{\mathcal K}} %
\newcommand{\SM}{{\mathcal S}} %
\newcommand{\U}{{\mathcal U}} %
\newcommand{\V}{{\mathcal V}} %
\newcommand{\W}{{\mathcal W}} %
\newcommand{\X}{{\mathcal X}} %
\newcommand{\Y}{{\mathcal Y}} %

\def\diam{\operatorname{diam}}

\def\id{\operatorname{id}}

\def\sign{\operatorname{sign}}
\def\supp{\operatorname{supp}}
 
\newcommand{\mres}{\mathbin{\vrule height 1.6ex depth 0pt width
0.13ex\vrule height 0.13ex depth 0pt width 1.3ex}}

\newcommand{\R}{\mathbb{R}}
\newcommand{\N}{\mathbb{N}}
\newcommand{\Z}{\mathbb{Z}}

\newcommand{\Rpositive}{\R^+}
\newcommand{\Rnonnegative}{\R^+_0}

\let \Rpositive \OZI
\let \Rnonnegative \CZI

\newcommand{\EO}{(0,T)\times O}
\newcommand{\iRn}{\int_{\R^d}}
\newcommand{\Om}{{\Omega}}
\newcommand{\oOm}{\overline{\Omega}}
\newcommand{\Bm}{B_m}

\newcommand{\Ker}{\text{Ker}}

\newcommand{\MB}{{\mathcal M}(\oB)}
\newcommand{\MO}{{\mathcal M}(\oOm)}
\newcommand{\MPB}{{\mathcal M}^+(\oB)}
\newcommand{\MPO}{{\mathcal M}^+(\oOm)}
\newcommand{\MGO}{{\mathcal M}_G(\oOm)}
\newcommand{\MGPO}{{\mathcal M}_G^+(\oOm)}
\newcommand{\MKB}{{\mathcal M}_{KR}(\oB)}
\newcommand{\MKO}{{\mathcal M}_{KR}(\oOm)}

\newcommand{\Mn}{{\mathcal M}(\R^d)}
\newcommand{\MPn}{{\mathcal M}^+(\R^d)}

\newcommand{\LOnen}{L^{1}(\R^d)}
\newcommand{\LOneB}{L^{1}(\Br)}
\newcommand{\LOneo}{L^{1}(O)}

\newcommand{\LTwon}{L^{2}(\R^d)}

\newcommand{\LInfB}{L^{\infty}(B_{\rho})}
\newcommand{\LInfn}{L^{\infty}(\R^d)}
\newcommand{\LInfo}{L^{\infty}(O)}

\newcommand{\WInfn}{W^{1,\infty}(\R^d)}
\newcommand{\WInfo}{W^{1,\infty}(O)}
\newcommand{\WInfB}{W^{1,\infty}(\Br)}
\newcommand{\WInfBo}{W^{1,\infty}(B_{\rho_0})}
\newcommand{\HOnen}{H^{1}(\R^d)}

\newcommand{\HTwon}{H^{2}(\R^d)}
\newcommand{\HmOnen}{H^{-1}(\R^d)}

\renewcommand{\aa}{a}
\newcommand{\bb}{b}

\makeatletter
\def\@fnsymbol#1{\ensuremath{\ifcase#1\or *\or **\else\@ctrerr\fi}}
\makeatother

\begin{document}

\title{{Local well-posedness for a novel nonlocal model for cell-cell adhesion via receptor  binding}}
\author{
Mabel Lizzy Rajendran\thanks{School of Mathematics, 
Watson Building, University of Birmingham, Edgbaston, Birmingham
B15 2TT, UK, \href{mailto:m.l.rajendran@bham.ac.uk}{m.l.rajendran@bham.ac.uk}} \ and Anna Zhigun\thanks{School of Mathematics and Physics, Queen's University Belfast, University Road, Belfast BT7 1NN, Northern Ireland, UK, \href{mailto:A.Zhigun@qub.ac.uk}{A.Zhigun@qub.ac.uk}} 
}
\date{}

\maketitle
\begin{abstract}
 Local well-posedness is established for a highly nonlocal nonlinear diffusion-adhesion system for bounded initial values with small support. Macroscopic systems of this kind were previously obtained by the authors through upscaling in \cite{ZRModelling} and can account for the effect of microscopic receptor binding dynamics in cell-cell adhesion. The system  analysed here couples an integro-PDE featuring degenerate diffusion of the  porous media type and nonlocal adhesion with a novel nonlinear integral equation. The approach is based on decoupling the system and using Banach's fixed point theorem to solve each of the two equations individually and subsequently the entire system. The main challenge of the implementation lies in selecting a suitable framework. One of the key results is the local well-posedness for the integral equation with a Radon measure as a parameter. The analysis of this equation utilizes the Kantorovich-Rubinstein norm, marking the first application of this norm in handling a nonlinear integral equation. 
\end{abstract}

\section{Introduction}
In this paper we study the coupling of the Cauchy problem for the integro-PDE 
\begin{subequations}\label{M_PLnewIBVP}
\begin{alignat}{3}
&\partial_tu=\nabla\cdot\left(2u\nabla u- v\chi(u)(\nabla H\star u)\right)&&\quad\text{in }(0,T]\times\R^d,\label{M_PLnew}\\
&u(0,\cdot)=u_0&&\quad\text{in }\R^d,\label{M_u0}
\end{alignat}
\end{subequations}
with the integral equation 
\begin{align}
&v=\frac{\G^{+}(u-v)}{\G^{+}(u-v)+\G^{-}v}u
\qquad\text{in }(0,T]\times\R^d,
\label{M_LambdaEq}
\end{align}
where  
\begin{align}
 (u,v):[0,T]\times\R^d\to\{(z_1,z_2)\in\R^2:\quad 0\leq z_2\leq z_1\}\nonumber
\end{align}
 is a pair of unknowns,  $\star$ denotes convolution in the spatial variable $x\in\R^d$, $d\in\N$, $T>0$, and 
\begin{subequations}\label{ParFct}
\begin{align}
&\chi\in C^1_b(\Rnonnegative), \qquad \chi(0)=0,\label{Asschi}\\
 &H(x):=-
\int_0^{|x|}\one_{[0,1]}F(s)\,ds,\label{H}\\
&F\in C^1(\Rnonnegative),\label{AssF}\\
 &\G^{\pm}f(t,x):=
\int_{\R^d}G^{\pm}(t,x,y) f(y)\,dy,\label{DefG}\\
&G^{\pm}(t,x,y):=\varphi^\pm(|x-y|) K^{\pm}\left(t,\frac{x+y}{2}\right),\label{DefG2}\\
 &\varphi^{\pm}(s):=\begin{cases}\left(1-s^{b^\pm}\right)^{\pm a^\pm}&\text{for }s\in[0,1),\\
 0&\text{for }s\in[1,\infty),\end{cases}
 \label{AssPhi}\\
 &K^{\pm}\in C^2([0,T]\times \R^d;(0,\infty)),\label{Kpm}\\
&a^{\pm}>0,\  b^{\pm}\geq 2.\label{aabb} 
\end{align}
\end{subequations}
System \cref{M_PLnewIBVP}-\cref{M_LambdaEq} is a modification of a novel system 
that the authors recently  derived using a multiscale approach, see (4.35) in \cite{ZRModelling}. Equation \cref{M_PLnew} models the  evolution of density $u=u(t,x)$, $t$ being time and $x$ position in space, of a cell population dispersing through diffusion and nonlocal cell-cell adhesion. {The advection direction due to adhesion is given by the nonlocal term $\nabla H\star u$, referred to as 'adhesion velocity' \cite{Armstrong2006}. The 'adhesion potential' \cite{ZRModelling}, $H$, is determined by a function $F$ that models the distance-dependent component of the local adhesion force that acts at a distance not exceeding the 'sensing radius' \cite{Armstrong2006}, here scaled to one. The additional dependent variable $v$ describes the amount of formed bonds and, together with the 'sensitivity coefficient' $\chi$ that can account for local heterogeneity of response to adhesion, modulates the} 'adhesion sensitivity'. More precisely, 
\begin{align}
 w:=\frac{v}{u}\in[0,1]
 \nonumber%
\end{align}
is the average fraction of bounded receptors. Variables $u$ and $v$ are further connected through a nonlinear nonlocal equation of a new type, \cref{M_LambdaEq} that  describes the equilibration of the receptor binding dynamics{. Here, $G^+$ and $G^-$ are the binding and unbinding rates, respectively. These rates have a distance-dependent component, $\varphi^{\pm}$, but may also depend on the concentration of an external factor, described by $K^{\pm}$. We refer to \cite{ZRModelling} for further details regarding this effect and its microscopic modelling that led to the  macroscopic equation \cref{M_LambdaEq}.}

Unlike the original model in \cite{ZRModelling}, 
system \cref{M_PLnewIBVP}-\cref{M_LambdaEq} includes the 
porous media-type (PM) diffusion
\begin{align}
 \Delta(u^2)\label{PM}
\end{align}
rather than the linear myopic diffusion
\begin{align}
 \nabla\nabla^T:({\mathbb D}u), \qquad {\mathbb D}={\mathbb D}(x)\in \R^{d\times d}.\label{myopic}
\end{align}
This modification is mainly motivated by the  analytical challenges that  are discussed in \cref{SecCh} below. Still, from the modelling perspective, such a term is justifiable and, in fact, often arises in derivations of models for cell motion due to a combination of short- and long-range interactions between population members. For example, in \cite{MorCapOel2005} a nonlocal equation was derived that corresponds to \cref{M_PLnew} for constant $\chi$ and $w$. Yet another distinction between \cref{M_PLnew} and the original equation is that the adhesion sensitivity coefficient $\chi$ is no longer an arbitrary function of $t$ and $x$. Instead, it is assumed to be a function of $u$.
This choice enables the incorporation of the influence of local density on adhesion sensitivity. {The assumption $\chi(0)=0$ is of a technical nature and is needed for our analysis. It is not unrealistic from the modelling perspective, ensuring a proper balance between a degenerate  diffusion and advection at low densities, see, e.g. similar examples in models with local taxis in \cite{BBTW}. }

\bigskip 
Since  $0\leq v\leq u$, it is  convenient to work with the pair 
\begin{align}
 (u,w):[0,T]\times\R^d\to[0,\infty)\times[0,1]\nonumber
\end{align}
rather than $(u,v)$, thus keeping the codomains of the unknowns independent. 
Restated for these variables, equations   \cref{M_PLnewIBVP}-\cref{M_LambdaEq} take the form
\begin{subequations}\label{PLnewIBVP}
\begin{alignat}{3}
&\partial_tu=\nabla\cdot\left(2u\nabla u- wu\chi(u)(\nabla H\star u)\right)&&\quad\text{in }(0,T]\times\R^d,\label{PLnew}\\
&u(0,\cdot)=u_0&&\quad\text{in }\R^d,\label{u0}
\end{alignat}
\end{subequations}
and
\begin{align}
&w=\Y(u,w)\qquad\text{in }\{u>0\},
\label{LambdaEq}
\end{align}
where
\begin{align}
 &\Y(u,w):=\frac{\G^{+}((1-w)u)}{\G^{+}((1-w)u)+\G^{-}(wu)}.\label{Y}
\end{align}
We mostly deal with this reformulation of the original system. In this work we make the first step in its study. Here we prove a local well-posedness result for the special case where $u$ remains compactly supported in a sufficiently small ball. In particular, the support diameter needs to be smaller than $1$, the radius of the sampling domain in the integral operators 
$\G^\pm$. 
The smallness of the support of the cell density can be interpreted as an initial stage of population dispersal, e.g. when a small tumour just starts invading the surrounding tissue. 

The result is local in the following sense. The solution exists  for finite times provided $u_0$ is close to a $\mu_0$ such that for some $w_0$ the pair $(\mu_0,w_0)$ solves \cref{LambdaEq} and satisfies a non-degeneracy condition. It is unique as long as it stays close to this pair.

The rest of the paper is organised as follows. In \cref{SecCh}, we present the main challenges encountered in the analysis of system \cref{PLnewIBVP}-\cref{LambdaEq} and explain how we handle them. After collecting some notation and useful facts in  \cref{SubsecRM}, we fully set up our problem and state the main results in \cref{prset}. The analysis of the individual equations \cref{LambdaEq} and \cref{PLnewIBVP} is presented in \cref{Sec:wEqn,ExPLnew}, respectively. In the intermediate \cref{Sec:FixedAdv} we study an auxiliary parabolic PDE including PM diffusion and a fixed advection direction. The  results for this equation feed into the analysis in \cref{ExPLnew}. The final \cref{SecMainProof} is devoted to the proof of a well-posedness result for the whole system \cref{PLnewIBVP}-\cref{LambdaEq}.

\section{Analytical challenges of \texorpdfstring{\cref{PLnewIBVP}-\cref{LambdaEq}}{} and their handling}\label{SecCh}
System \cref{PLnewIBVP}-\cref{LambdaEq} combines two equations of a very different nature. In this situation, a standard approach, which we also adopt in this paper, is first to  decouple the system and establish solvability of the two equations separately, assuming that $w$ and $u$ are fixed in \cref{PLnewIBVP,LambdaEq}, respectively,  and then to obtain a joint solution by means of some fixed-point argument. The non-standard nonlinear integral equation \cref{LambdaEq} puts multiple obstacles in the way.
\begin{enumerate}
[label=(IE\arabic*),ref=(IE\arabic*)]
 \item\label{ChI1} As with every new equation, it would be desirable to have some explicit solutions of \cref{LambdaEq} at hand, at least for some simple time-independent $u$. Not only would it help to understand what to expect from its solutions in general, but it could also allow to prove existence of other solutions that are close to already known ones. This is even more important in the case of system \cref{PLnewIBVP}-\cref{LambdaEq} where $u$ solves an initial value problem, so that one needs a $w_0$ corresponding to a given initial value $u_0$, to start with.  
 Yet already for constant functions $u$  it is not evident how, if at all, a matching $w$ could be determined. 
 If, however, we allow $u$ to be a singular measure, accordingly reinterpreting the involved integral operators $\G^{\pm}$ as integrals of kernels $G^{\pm}$ against such a measure (see \cref{SecY}), then the equation can be resolved, e.g. in the special case of $u$ a point mass (see \cref{Exdelta}).  
 \item\label{ChI2} Based on the observations made in the previous item, one is led to try to establish solvability of \cref{LambdaEq} for such $u$ that are close to a point mass but are functions rather than singular measures. This would be desirable for its own sake but also because equation  \cref{PLnew} includes nonlinear functions of $u$, hence does not allow for measure-valued solutions. As is well-known, singular measures cannot be approximated by Lebesgue integrable functions if the distance between them is measured by the metric induced by the variation norm, with which the whole space of signed Radon measures is equipped. To circumvent this,  weaker norms restricted to suitable subspaces are used instead. A popular choice is the Kantorovich-Rubinstein (KR) norm \cite{KantorovichAkilov} (see \cref{SecM}) that we also utilise in this work. 
The KR metric generated by the KR norm is a special case of the Wasserstein metric. This metric has proven to be a valuable tool for measuring distances between probability measures in various contexts, including optimal transportation problems \cite{Villani, AmbrosioGigliSavare}, analysis of evolution PDEs with a gradient flow structure \cite{AmbrosioGigliSavare}, upscaling of mean field equations \cite{Golse}, and more recently in statistical machine learning \cite{Kolourietal}. To the best of our knowledge, the KR metric has so far not been used in analysis of nonlinear integral equations. As it turns out, it offers a convenient framework for such equations as \cref{LambdaEq}. In particular, balls in this metric centred at a point mass contain  bounded compactly supported functions of any shape (see \cref{ExSolvPointM}). This provides an ample source of suitable initial values $u_0$.
\item\label{ChI3} The trade-off associated with using the KR norm in the context of the  linear integral operators such as $\G^\pm$ 
is that it requires their kernels to be sufficiently smooth in order to be able to derive the bounds we need. In the case of $G^-$ we have to avoid integration in the neighbourhoods of its  singularities. For each $x$, the corresponding singularity set is a circle of radius $1$ centred at $x$. To maintain a positive distance from these sets, we restrict our study of \cref{LambdaEq} to measures $u$ with support diameter smaller than $1$.
\item\label{ChI4} The nonlinear operator $\Y$ on the right-hand side of \cref{LambdaEq} is of the form
\begin{align}
 \Y=\frac{a}{a+b},\label{Yab}
\end{align}
where the applications of the integral operators $\G^+$ and $\G^-$ that are inserted into $a$ and $b$, respectively, are both zero at points that are {at} a distance farther than $1$ from the support of $u$. The function on the right-hand side of \cref{Yab} cannot be extended to $(0,0)$ in a continuous fashion. Hence, the denominator of $\Y$ is a  potential source of singularities of this operator.  Since we aim at solving \cref{LambdaEq} for $u(t,\cdot)$, $t\in[0,T]$ for some $T>0$, we strengthen the condition from the previous item, requiring the supports of all members of this family 
to be  confined to a single ball of diameter less than $1$. Then, for each $t$, \cref{LambdaEq} only needs to be solved in this ball, and there $\Y$ has good properties (see \cref{SecYball}). 
\item\label{ChI5}  When solving an equation with a parameter, some version of the implicit mapping theorem is often the method of choice. For instance, the Banach space adaptation,  \cite[Chapter 4, \S4.7, Theorem 4.B]{ZeidlerFP},  provides solvability in a small ball in the parameter space centred at a parameter value for which the existence of a solution is known.  Having chosen a setting for \cref{LambdaEq} based on the discussion in \cref{ChI1,ChI2,ChI3,ChI4}, we cannot apply that theorem directly because the set of admissible parameters does not contain a  ball. Still, since it is convex%
, it is possible to construct a similar argument and thereby derive some useful estimates that ensure a locally stable dependence of $w$ upon $u$ (see \cref{Thm:wEqn-FP} and its proof given in \cref{SecExEq2}). 
\end{enumerate}
Having decided on a suitable framework for \cref{LambdaEq}, we turn to  \cref{PLnewIBVP}.
\begin{enumerate}
[label=(PDE\arabic*),ref=(PDE\arabic*)]
\item\label{PDE1}  As already mentioned in the Introduction, the original derivation in  \cite{ZRModelling} yielded a diffusion-advection equation with the linear myopic diffusion \cref{myopic}. 
 The few available works on analysis of such equations specifically focused on local taxis \cite{WinSur2017,Heihoff,EckSur}. This type of advection occurs along the gradient of an external stimulus and  not as a result of    nonlocal self-adhesion.

Assume that ${\mathbb D}$ is a smooth positive semi-definite matrix. 
If it is also uniformly positive definite, then  \cref{myopic} does not pose considerable additional analytical difficulties compared to the basic constant diffusion case  covered in \cite{HillenPainterWinkler} for diffusion-adhesion models.
However, a nondegenerate diffusion, whether myopic or not, immediately destroys  compactness of support, hence is unsuitable for the present setting, see \cref{ChI4}.
To evade this unwanted  effect, diffusion has to be  degenerate.
In this work, we selected the PM diffusion  \cref{PM}. This choice is primarily motivated by the need to guarantee that for small times the support of $u$ is confined to a small ball, as advocated in \cref{ChI4}. At the same time, it allows us to                                                                                                                                                                                                                                                                                                                                                                                                                                                                                                                                                                                         benefit from a range of analytical tools already available for this 
prototypical degenerate diffusion  \cite{Vazquez},  which is much better understood than a myopic one.  

On the whole, perhaps  somewhat paradoxically, degeneracy  {\it enables} the analysis of \cref{PLnewIBVP}-\cref{LambdaEq}.    
\item\label{ChPM} Analysis of equations that combine   quasilinear degenerate diffusion
 such as \cref{PM} 
and nonlocal advection with velocity
\begin{align}
  \nabla H\star u=\nabla(H\star u),\nonumber%
 \end{align}
 corresponding to \begin{align*}
\chi\equiv&1,\qquad\\  w\equiv&1,                                                 \end{align*}
has attracted much attention. For smooth kernels $H$, well-posedness, which is what is  relevant for us here,  initially in entropy and then  also in weak senses was established in  
 \cite{BurgerCapMor} and \cite{BertozziSlepcev}, respectively. Subsequent works have primarily focused on kernels singular only at the origin where they tend to infinity, see \cite{CarrCrYao} and references therein. 
 The situation is different for the self-adhesion kernel \cref{H}. Its gradient is bounded and supported in the unit ball centred at the origin, yet generally has   discontinuities in the centre and on the boundary of the ball. As a result, $\Delta H$ is, in most cases, a singular Radon measure with singular parts supported on the ball's boundary and, in dimension one, also in the origin (see \cref{LemH}). Still, because this measure is finite,   the standardly  required  estimates in $L^p$ norms for its convolution with $u$ can be obtained (see in the proof of \cref{LemV1}).
 
 Two recent studies \cite{DaneriRadiciRuna,FagioliRadici} incorporated  a non-constant coefficient function $\chi$. In a one-dimensional setting, weak solutions were obtained there as limits of first-order deterministic many-particle systems. In these works, $\chi$ was assumed to be nonnegative and decreasing, as well as satisfy certain other conditions.
 As to the kernel, it was not mandated to be compactly supported but   allowed to have a Lipschitz singularity solely at the origin  \cite{DaneriRadiciRuna}. 
 
 Unlike \cite{DaneriRadiciRuna,FagioliRadici}, we impose  \cref{Asschi}, thus ensuring  that $\chi$ is subordinate to the diffusion coefficient at $u=0$. 
 This assumption is sufficient for uniqueness of weak solutions $u$ to \cref{PLnewIBVP} for fixed $w$ and is convenient for several other steps in our analysis that includes well-posedness as well as other properties such as support control and stable dependence on $w$ (see \cref{SubSecuEqWP}).
\end{enumerate}
We emphasise that our objective here is not to have the least restrictive assumptions on the parameters in \cref{PLnew}. We have therefore chosen conditions that ensure solvability of the entire system \cref{PLnewIBVP}-\cref{LambdaEq}, while at the same time being reasonable from the modelling perspective.

\section{Preliminaries}\label {SubsecRM}
\subsection{Miscellaneous notation}

For convenience, in the few cases of a numerical fraction that has a finite  numerator $a$ and a potentially zero denominator the interpretation is  
\begin{align*}
 \frac{a}{0}:=\infty.
\end{align*}

\medskip

The spacial domains that we mostly consider are  either the whole space $\R^d$, $d\in\N$, or the ball $B_{\rho}$ in $\R^d$ that is  centred at the origin and has  radius $\rho>0$, its boundary being the sphere  $S_{\rho}$. More generally,  $B_{\rho}(x)$ stands for the $\rho$-ball centred at $x$. The $(d-1)$-dimensional unit ball centred at the origin is denoted by $B_1^{d-1}$.

For $T>0$, we define the time-space cylinder  $E_T:=(0,T)\times\R^d$. 

\medskip

Depending on the context, $|\cdot|$ can stand for: the absolute value of a number, the Euclidean norm of a vector in $\R^d$, the $d$-dimensional Lebesgue measure of a subset in $\R^d$, and the total variation of a Radon measure. 

\medskip

For any set $A$, we denote by $\one_A$ its characteristic function. Another standard function we use is
\begin{align}
 \sign(x):=\begin{cases}
            \frac{x}{|x|}&\text{for }x\in\R^k\backslash\{0\},\\
            0&\text{for }x=0,
           \end{cases}
\qquad k\in\{1,d\}.\nonumber
\end{align}

\medskip

For a real variable $z$, e.g. such as $t$ or $x_i$, we write $\partial_z$ for the usual partial derivative with respect to $z$. Further, $\nabla$, $\nabla\cdot$, and $\Delta$ stand for the spatial gradient, divergence, and Laplace operators, respectively. 

When referring to the Gateaux partial differentials in an infinite-dimensional setting, we always mean the right-side Gateaux partial semi-differentials. For example, an operator $l=l(\mu,w)$ in a Banach space possesses the Gateaux partial differential at $(\mu_0,w_0)$ with respect to $w$ in direction $h$ if the following limit exists:              
\begin{align*}
 &\partial_{w}l(\mu_0,w_0)h:=\underset{s\to0+}{\lim}\frac{1}{s}\left(l(\mu_0,w_0+sh)-l(\mu_0,w_0)\right).
\end{align*}
The same notation is used for the Fr\'echet partial derivatives. Which of the two notations is used depends on the context and is explicitly mentioned.
\medskip

For the (continuous) dual of a normed space $X$ we use the notation $X^*$. Symbols $\rightharpoonup$ and $\overset{*}{\rightharpoonup}$ represent the weak and weak-$*$ convergence, respectively. 
By $<\cdot,\cdot>$ we denote the duality paring. 

For $X$, $Y$, and $Z$ normed spaces, we write $L(X;Z)$ for the space of bounded linear maps from $X$ into $Z$ and $B(X\times Y;Z)$ for the space of bounded bilinear maps from $X\times Y\to Z$.

If $X,Y\subset Z$, we set 
\begin{align*}
 \|\cdot\|_{X\cap Y}:=\max\{\|\cdot\|_{X},\|\cdot\|_{Y}\}.
\end{align*}

\medskip

Let $k\in \N_0\cup \{\infty\}$ and $E\subset \R^N$ for some $N\in\N$. As usual, $C^k(E)$ ($C(E):=C^0(E)$)  is the notation for the space of $k$ times continuously differentiable real-valued functions in $E$.
The subspaces of bounded and compactly supported functions in this class are denoted by $C_b^k(E)$ and $C_0^k(E)$, respectively. When dealing with vector- or Banach-valued function, but, also, when we want to specify the codomain that is a particular subset of $\R$, we use, e.g. the notation $C^k(E;F)$ for a $C^k$-mapping from $E$ into $F$. For $F$ a Banach space,  $C_w(E;F)$ denotes the space of weakly continuous maps between these sets.

We use standard Lebesgue spaces $L^p(\Omega)$, Sobolev spaces $W^{s,p}(\Omega)$ and  $H^s(\Omega):=W^{s,2}(\Omega)$ for various $p\in[1,\infty]$ and $s\in\Z$, as well as Bochner spaces of functions taking values in such spaces.
In particular, we denote 
\begin{align*}
 L^{p,q}(E_T):=L^{p}(0,T;L^q(\R^d)),\qquad L^{p}(E_T):=L^{p,p}(E_T)\qquad\text{for }p,q\in[1,\infty].
\end{align*}

We assume the reader to be familiar with the standard properties of the above mentioned functional spaces. 
We also need to deal with subsets of the space of finite measures. The related notation and required facts are provided in \cref{SecM} below.

\medskip

Throughout the paper, $C_i$, $T_i$, and $\rho_i$ are either positive constants or positive-valued mappings, whereas  $U_i$ and $W_i$ are sets or families of sets. In many instances we do not track the dependence of these entities on some or all of parameters $d$, $\chi$, $F$, $a^\pm$, $b^\pm$, $K^\pm$, $u_0$, as well as $w_0$, an initial state of $w$. When a change occurs, e.g. because we stop tracking certain dependences, a remark on notation is provided at the beginning of the corresponding section/subsection.
\subsection{Radon measures and the Kantorovich-Rubinstein norm}\label{SecM}
Let $\Omega\subset\R^d$  be a bounded domain.  
 We denote by $\MO$ ($\MPO$) the space of finite signed (positive) Radon measures in $\oOm$ and by $|\mu|$ the total variation of $\mu\in\MO$. As is well-known, $(\MO,|\cdot|)$ is a non-separable Banach space. In particular, singular measures, such as point masses, cannot be obtained as limits of Lebesgue integrable functions in this norm. This issue is relevant in our case, see the discussion in \cref{SecCh}\cref{ChI2}. To counteract it, we also use the Kantorovich-Rubinstein (KR)  norm  associated with the Euclidean metric on $\R^d$. Since $\oOm$ is compact, this norm is well-defined on the subspace  
 $${\mathcal M}_{KR}(\oOm):=\{\mu\in\MO:\ \mu(\oOm) = 0\}.$$ By the  KR duality theorem  %
\cite[Chapter VIII, \S 4.5, Theorem 1]{KantorovichAkilov}, this norm is given by
\begin{align}
 \|\mu\|_{\MKO}:=\sup\left\{\int_{\overline{\Omega}}\varphi\,d\mu(y):\quad \varphi \in W^{1,\infty}(\Omega)\text{ and } \|\nabla \varphi\|_{L^{\infty}(\Omega)}\leq1\right\}\quad\text{for all } \mu\in \MKO.\label{dualKR}
\end{align}
It is strictly weaker than the total variation and satisfies
\begin{align}
 \|\mu\|_{\MKO}\leq \frac{\diam(\Omega)}{2}|\mu|\qquad\text{for all }\mu\in \MKO.\label{NormComp}
\end{align}
This norm gives rise to the  Wasserstein $1$-distance on the set of probability measures \cite{Villani} or, more generally, on the set of positive measures that have the same total variation. Moreover, when restricted to an arbitrary closed ball in the variational norm, the convergence with respect to the Wasserstein metric is equivalent to the weak-$*$ convergence \cite[Chapter VIII, \S 4.6, Theorem 3]{KantorovichAkilov}. As a result, this metric turns such a ball into a metric space which is compact and, hence, separable. 
We refer, e.g. to \cite[Chapter VIII, \S 4]{KantorovichAkilov} for further details.

\medskip
For $w:\oOm\to\R$ Borel and bounded and $\mu\in \MPO$ we use the standard  definition for their product:
\begin{align*}
 w\mu(B):=\int_{B}w\,d\mu\quad\text{for any Borel set }B\subset \oOm.
\end{align*}
It is well-known that  $w\mu\in\MO$, and for any bounded Borel $f:\oOm\to\R$ 
it holds that
\begin{align}
 \int_{\oOm}f\,d(w\mu)=\int_{\oOm}fw\,d\mu.\nonumber%
\end{align}
\medskip

The notation $\mu\mres A$ means the restriction of a measure $\mu$ to a set $A$.

\medskip

By $\LL^d$ and ${\cal H}^{d-1}$ we denote the $d$-dimensional Lebesgue and the $(d-1)$-dimensional Hausdorff measures, respectively. We identify $u$ and $u\LL^d$ for $u\in L^1$. 

Unless stated otherwise, a.e. means 'almost everywhere' with respect to the Lebesgue measure with dimension that is standardly adopted for a set in question.
\section{Problem setting and main results}\label{prset}

We define solutions to individual equations  \cref{M_PLnewIBVP,PLnewIBVP,M_LambdaEq,LambdaEq} and the corresponding systems as follows.
\begin{Definition}[Solutions to \cref{M_PLnewIBVP,PLnewIBVP}]\label[Definition]{DefIPDE}
Let $T>0$,  \cref{H,Asschi,AssF}, and  
\begin{align}
 0\leq u_0\in \LInfn\cap\LOnen\label{Assumpu01Om}
\end{align}
be satisfied. 
\begin{enumerate}
 \item We call a pair of functions $(u,v):[0,T]\times\R^d\to[0,\infty)\times[0,\infty)$ a  solution to \cref{M_PLnewIBVP} in $[0,T]\times\R^d$ if:
\begin{enumerate}
\item $u\in C_w\left([0,T];\LInfn\cap\LOnen\right)\cap L^{\infty}(0,T;\LInfn)$;
\item%
  $u^{\frac{3}{2}}\in L^2\left(0,T;\HOnen\right)$;
 \item%
  $\partial_tu\in L^2\left(0,T;\HmOnen\right)$, $\partial_tu \in L^{\infty}\left(0,T;\left(W^{2,p}\left(\R^d\right)\right)'\right)$ for all $p\in[1,\infty]$;
  \item $v\in L^{\infty}(0,T;\LInfn)$;
 \item $(u,v)$ satisfies \cref{M_PLnewIBVP} in  a weak, 
 \begin{align}
  \left<\partial_tu,\varphi\right>=-\int_{\R^d}\left(\nabla u^2-v\chi(u)(\nabla H\star u)\right)\cdot\nabla\varphi\,dx\qquad\text{a.e. in }\left(0,T\right)\text{ for all }\varphi\in\HOnen,
  \nonumber%
 \end{align}
 and a very weak,  
 \begin{align}
  \left<\partial_tu,\varphi\right>=\int_{\R^d} u^2\Delta \varphi+v\chi(u)(\nabla H\star u)\cdot\nabla\varphi\,dx\quad\text{a.e. in }\left(0,T\right)\text{ for all }\varphi\in \bigcup\limits_{p=1}^{\infty} W^{2,p}(\R^d),\nonumber%
 \end{align}
 senses%
;
\item%
$u\left(0,\cdot\right)=u_0$ in $\LInfn\cap\LOnen$.
\end{enumerate}
\item We call a pair of functions $(u,w):[0,T]\times\R^d\to[0,\infty)\times[0,\infty)$ a  solution to \cref{PLnewIBVP} in $[0,T]\times\R^d$ if $(u,wu)$ is a solution to \cref{M_PLnewIBVP} in $[0,T]\times\R^d$ in the above sense.
\end{enumerate}

\end{Definition}

\begin{Definition}[Time-independent solutions to \cref{M_LambdaEq,LambdaEq}]\label[Definition]{DefEqY}
Let assumptions  \cref{DefG2,AssPhi,Kpm,aabb} be satisfied. 
\begin{enumerate}
 \item We call a pair of measures $(\mu,\nu)$ a solution to \cref{M_LambdaEq} if:
 \begin{enumerate}
 \item%
 $\mu\in\MPn$;
  \item%
  $\G^-\mu<\infty$;
  \item $\nu\in \MPn$;
  \item $\nu\leq \mu$ in $\R^d$; 
  \item $(\mu,\nu)$ satisfies \cref{M_LambdaEq} in the following sense:
  \begin{align}
&\nu=\frac{\G^{+}(\mu-\nu)}{\G^{+}(\mu-\nu)+\G^{-}\nu}\mu\qquad\text{in }\MPn,
\label{M_LambdaEq_m}
\end{align}
where $0/0$ in the fraction is interpreted as $1$;
 \end{enumerate}
\item We call a pair of a measure and a function $(\mu,w)$ a solution to \cref{LambdaEq} if:
\begin{enumerate}
\item $w:\R^d\to[0,1]$ is Borel;
\item $(\mu,w\mu)$ is a solution to \cref{M_LambdaEq} in the above sense.\end{enumerate}

\end{enumerate}
\end{Definition}
\begin{Remark}%
 \begin{enumerate}
\item 
In consequence of  \cref{LemYwd} below, the fraction on the right-hand side of \cref{M_LambdaEq_m} is well-defined and Borel. Therefore, \cref{DefEqY} makes sense. 
 \item If $(\mu,w)$ has the regularity required by \cref{DefEqY} and $\nu=w\mu$, then \cref{M_LambdaEq_m} is obviously  equivalent to 
 \begin{align}
  w=\Y(\mu,w)\qquad\mu-\text{a.e.}\nonumber
 \end{align}   
 Thus, through \cref{M_LambdaEq_m} we have a rigorous interpretation of \cref{LambdaEq}.
 \end{enumerate}

\end{Remark}

\begin{Definition}[Solutions to \cref{M_PLnewIBVP}-\cref{M_LambdaEq}  and \cref{PLnewIBVP}-\cref{LambdaEq}]\label[Definition]{DefSol}
Let $T>0$ and   \cref{H,Asschi,AssF}, \cref{DefG2,AssPhi,Kpm,aabb}, and \cref{Assumpu01Om} be satisfied. We call a pair of functions $(u,v)$ ($(u,w)$) a solution to system \cref{M_PLnewIBVP}-\cref{M_LambdaEq}  (\cref{PLnewIBVP}-\cref{LambdaEq}) in $[0,T]\times\R^d$ if:
\begin{enumerate}
\item $(u,v)$ ($(u,w)$) solves \cref{M_PLnewIBVP} (\cref{PLnewIBVP}) in $[0,T]\times\R^d$ in the sense of \cref{DefIPDE};
\item for all $t\in[0,T]$, $(u(t,\cdot),v(t,\cdot))$ ($(u(t,\cdot),w(t,\cdot))$) solves \cref{M_PLnewIBVP} (\cref{PLnewIBVP}) in the sense of \cref{DefEqY}.
\end{enumerate}

\end{Definition}

The main results of this paper are the following Theorems on local well-posedness of equation \cref{LambdaEq} and system \cref{PLnewIBVP}-\cref{LambdaEq}.
\begin{Theorem}[Local well-posedness of \cref{LambdaEq}]\label{Thm:wEqn-FP}
Let  $$\rho\in\left(0,\frac{1}{2}\right)$$ and $(\mu_0,w_0)$ satisfy 
\begin{subequations}\label{Assmu0w0}
 \begin{align}
  &\mu_0\in (\MPB)\backslash\{0\},\\
  &w_0\in W^{1,\infty}(B_{\rho};(0,1)),\\
  &w_0=\Y(\mu_0,w_0)\qquad  \text{in }\oB,\\
  &\X:=id-\partial_w\Y(\mu_0,w_0)\text{ is invertible in }L(\WInfB),
 \end{align}
\end{subequations}
where the partial derivative $\partial_w$ is taken in the Fr\'echet sense. 
Define sets
\begin{subequations}\label{setsWM}
\begin{alignat}{3}
\Cl[U]{SU2}:=&\Big\{\mu\in \MPB:&&\quad |\mu|=|\mu_0|\ \text{ and }\ \|\mu-\mu_0\|_{\MKB}\leq \Cl{R2}\Big\},\\
\Cl[W]{W3}:=&\Big\{w\in W^{1,\infty}(B_{\rho};[0,1]):&&\quad \|w-w_0\|_{\WInfB}\leq \Cl{R1}\Big\}.\label{SetW} 
\end{alignat}
\end{subequations}
Then, there exist constants $\Cr{R2},\Cr{R1},\Cl{F-contin-mu}>0$ that depend only on the parameters from \cref{DefG2,AssPhi,Kpm,aabb}  
as well as 
\begin{align*}
|\mu_0|,\ \rho,\
\|\nabla w_0\|_{L^{\infty}\left(B_{\rho}\right)},\  \min_{\overline{B_{\rho}}}w_0,\  \max_{\overline{B_{\rho}}}w_0, \
 \left\|\X^{-1}\right\|_{{L\left(W^{1,\infty}\left(B_{\rho}\right)\right)}},                                                                                                                                                                                                                  \end{align*}
such that for all $\mu\in \Cr{SU2}$ there exists a unique $w\in \Cr{W3}$ for which 
\begin{align}
 &w=\Y(\mu,w)\qquad  \text{in }\oB,\label{LambdaEqmu}
\end{align}
 and the solution map 
\begin{align}
    \Cl[WW]{OpW1}: \Cr{SU2}\to \Cr{W3}\cap C^2(\oB;(0,1)),\qquad \mu\mapsto w,  \label{solmapW}                  
\end{align}
 is well-defined and Lipschitz continuous in the following sense:
\begin{align}
\|\Cr{OpW1}(\mu_1)-\Cr{OpW1}(\mu_2)\|_{\WInfB} \leq \Cr{F-contin-mu} \|\mu_1-\mu_2\|_{\MKB}\qquad\text{for all }\mu_1,\mu_2\in \Cr{SU2}.\label{W-stability}
\end{align}

\end{Theorem}
\begin{Remark}%
The Fr\'echet partial differentiability of $\Y$ with respect to $w$ in the case of a spatial domain being a ball of radius smaller then $1/2$ is established in \cref{Lem:contdwY}.%
\end{Remark}
\begin{Remark}
 By \cref{Lemma:reg_w}, every solution to \cref{LambdaEqmu}  satisfies $w\in C^2(\oB;(0,1))$. Hence, the requirement $w\in \Cr{W3}\cap C^2(\oB;(0,1))$ in \cref{solmapW} only restricts the distance to $w_0$.
\end{Remark}

The proof of \cref{Thm:wEqn-FP} is given in \cref{SecExEq2}. %
\begin{Theorem}[Local well-posedness of \cref{PLnewIBVP}-\cref{LambdaEq}]
\label{thmmain}
Let $d\in\N$, \cref{H,Asschi,AssF} and  \cref{DefG2,AssPhi,Kpm,aabb} hold, and be given: numbers $m,m_{\infty},\rho>0$ such that
\begin{align}
 &\rho<\min\left\{\frac{1}{2},\frac{d+2}{m\|\chi'\|_{L^{\infty}(\Rpositive)}\|F\|_{L^{\infty}(0,1)}}\right\},\label{rho0}
\end{align}
 and a pair $(\mu_0,w_0)$ that satisfies assumption \cref{Assmu0w0} of \cref{Thm:wEqn-FP} 
 and
\begin{align}
 |\mu_0|=m.\nonumber
\end{align}
Define 
\begin{alignat*}{3}
 \Cl[U]{U01}:=&\Big\{0\leq u_0\in \LInfn:&&\quad\supp(u_0)\subset \overline{B_{\frac{1}{4}\rho}},\\
 &&&\quad \|u_0\|_{\LInfn}\leq m_{\infty},\\
 &&&\quad\|u_0\|_{\LOnen}= m,\\
 &&&\quad\|u_0-\mu_0\|_{\MKB}\leq \frac12\Cr{R2}\Big\}.
\end{alignat*}
Then, there exists a number $\Cl[T]{Tstar}$ that depends only on the parameters from \cref{ParFct} as well as  
\begin{align}
&m,\ m_{\infty},\ \rho,\ 
\|\nabla w_0\|_{\LInfB},\  \min_{\oB}w_0,\  \max_{\oB}w_0, \
 \left\|\X^{-1}\right\|_{{L\left(\WInfB\right)}},\label{thmmainPar}                                                                                                                                                                                                                  \end{align}
such that for any $u_0\in\Cr{U01}$ there exists a pair  of functions $(u,w)$ that solves \cref{PLnewIBVP}-\cref{LambdaEq} in $[0,\Cr{Tstar}]\times\R^d$ in the sense of \cref{DefSol} and  satisfies
\begin{subequations}\label{Condu}
\begin{align}
 &\supp(u(t,\cdot))\subset \Br\qquad\text{for all }t\in[0,\Cr{Tstar}],\label{SRsuppu}
 \\
&\|u-u_0\|_{C([0,\Cr{Tstar}],\MKB)}\leq \Cr{R2},%
\end{align}
\end{subequations}
and 
\begin{subequations}\label{Condw}
\begin{align}
 &w(t,\cdot)\in C^2(\oB;(0,1))\qquad\text{for all }t\in[0,\Cr{Tstar}],\\
 &w\in C([0,\Cr{Tstar}];\WInfn),\label{wReg1}
 \\
 &\|w-w_0\|_{C([0,\Cr{Tstar}];\WInfB)}\leq \Cr{R1}.\label{wBall}
\end{align}
\end{subequations}
In the above, constants $\Cr{R2}$ and $\Cr{R1}$ are from \cref{Thm:wEqn-FP}.

The solution is locally unique in the following sense: if for some $T\in(0,\Cr{Tstar}]$ another solution $(\wu,\ww)$ in $[0,T]\times\R^d$ satisfies
\begin{subequations}\label{AssUniq1}
\begin{align}
&\ww\in C([0,T];\WInfn),\label{wReg1_}
 \\
  &\|\ww-w_0\|_{C([0,T];\WInfB)}\leq \Cr{R1},\label{wBall_}
\end{align}
\end{subequations}
then for all $t\in [0,T]$ it holds that
\begin{subequations}\label{uniqmean}
\begin{alignat}{3}
 &u(t,\cdot)=\wu(t,\cdot)&&\qquad\text{ a.e. in }\R^d,\\
 &w(t,\cdot)=\ww(t,\cdot)&&\qquad\text{ in }\oB.
\end{alignat}
\end{subequations}

Finally, a Lipschitz property holds: for every $u_0,\wu_0\in\Cr{U01}$ the corresponding solutions $(u,w)$ and $(\wu,\ww)$ in $[0,\Cr{Tstar}]\times\R^d$  satisfy 
\begin{align}
 \max_{[0,\Cr{Tstar}]}\left\|u-\wu\right\|_{\LOneB}+\left\|w-\ww\right\|_{C([0,\Cr{Tstar}];\WInfB)}\leq\Cl{CLipS}\left\|u_0-\wu_0\right\|_{L^1\left(B_{\frac{\rho}{4}}\right)},\label{mainLip}
\end{align}
where constant $\Cr{CLipS}$ only depends on parameters from \cref{ParFct,thmmainPar}.
\end{Theorem}
This Theorem is proved in \cref{SecMainProof}.%

Solutions to \cref{M_LambdaEq} and \cref{M_PLnewIBVP}-\cref{M_LambdaEq} can be recovered from the corresponding solutions to \cref{LambdaEq} and \cref{PLnewIBVP}-\cref{LambdaEq}, respectively. Moreover, there are no other solutions, as the next two theorems imply.
\begin{Theorem}[Time-independent solutions to \cref{M_LambdaEq} vs. \cref{LambdaEq}]\label{CompThmLambdaEq}
 If $(\mu,\nu)$ is a time-independent solution to \cref{M_LambdaEq}, then $\left(\mu,\frac{d\nu}{d\mu}\right)$ is a time-independent solution to \cref{LambdaEq}.
\end{Theorem}

\begin{Theorem}[Solutions to \cref{M_PLnewIBVP}-\cref{M_LambdaEq} vs. \cref{PLnewIBVP}-\cref{LambdaEq}]\label{M_thmmain} 
 If $(u,v)$ is a solution to \cref{M_PLnewIBVP}-\cref{M_LambdaEq} in $[0,T]\times\R^d$, then $\left(u,w\right)$ is a solution to \cref{PLnewIBVP}-\cref{LambdaEq} for 
 some $w$ that for all $t\in[0,T]$ satisfies
 \begin{align}
 &w(t,\cdot)\text{ is Borel},\nonumber\\
 & w(t,\cdot)={\frac{v}{u}}(t,\cdot)\qquad\text{a.e. in }\{u>0\}.\nonumber
 \end{align}

\end{Theorem}
These Theorems are proved in \cref{SecExEq2,SecMainProof}, respectively.

\section{{Well-posedness of  {\texorpdfstring{\cref{LambdaEq}}{}} for fixed \texorpdfstring{$u(=:\mu)$}{u}}} \label{Sec:wEqn}
The main goal of this Section is to prove \cref{Thm:wEqn-FP}, a result on local well-posedness of equation \cref{LambdaEqmu} for given $\mu$ from a suitable subset of a space of measures and sufficiently close to some $\mu_0$ for which a solution is assumed to exist. 
Throughout the Section we assume   %
\cref{DefG,DefG2,AssPhi,Kpm,aabb} to hold and also that all involved functions/measures are time-independent. To highlight the fact that the first argument of $\Y$ can be a measure, we denote it by $\mu$ rather than $u$.

We begin by properly defining operator $\Y$ and studying its  properties, first in general domains in \cref{SecY} and then in a small ball in \cref{SecYball}. Some of the presented results are not needed in \cref{SecExEq2}, where we prove \cref{Thm:wEqn-FP}. Still, they provide  useful insights into the nature of this new operator.

\subsection{{Operator \texorpdfstring{$\Y$}{Y}}}\label{SecY}
{Throughout this Subsection we assume 
\begin{align}
 \Omega\subset\R^d\text{ a domain}.\nonumber
\end{align}
Here we study operator $\Y$ from \cref{Y} for measures in $\oOm$ as its first argument. We start with operators $\G^\pm$. 
Set
\begin{align}
 &\G^{\pm}\mu(x):=
 \int_{\oOm}G^{\pm}(x,y) \,d\mu(y),\nonumber
\end{align}
whenever the integral exists, and
\begin{subequations}\label{DefGs}
\begin{align}
   &\MGO:=\left\{\mu\in \MO:\ \G^-|\mu|<\infty \ \text{in }\oOm\right\},\\
   &\MGPO:= \MGO\cap\MPO.
\end{align}
\end{subequations}
Our first Lemma ensures that set $\MGO$ and operators $\G^{\pm}$ applied to its elements are well-defined.}
\begin{Lemma}\label[Lemma]{LemGwd}
 \begin{enumerate}[label=(\arabic*), ref=(\arabic*)]
\item\label{GpmWD1} 
Let $\mu\in \MPO$. Then $\G^{\pm}\mu:\oOm\to [0,\infty]$ are well-defined Borel functions.   
\item $\MGO$ and $\MGPO$ are well-defined linear subspace and convex subset, respectively, of $\MO$.
\item\label{GpmWD3} Let $\mu\in\MGO$. Then $\G^{\pm}\mu:\oOm\to [0,\infty)$ are well-defined Borel functions.
\end{enumerate}
\end{Lemma}
\begin{proof}
 Being superpositions of continuous functions and the characteristic function of an interval, {hence of Borel functions,} $$G^\pm:\oOm\times\oOm\to \Rnonnegative$$ are 
Borel. 
 By Fubini's theorem,
 $$\G^{\pm}\mu_{\pm}:\oOm\to [0,\infty]$$ are well-defined Borel functions for every $\mu\in\MO$.
 
 Further, due to continuity of $K^+$ and $\varphi^+$ and compactness of the support of the latter, $G^+(x,\cdot)$ is bounded for all $x\in\oOm$. Hence, it is integrable against every $\mu\in\MO$.
 
 The above observations combined readily imply that \cref{GpmWD1,GpmWD3} hold and sets \cref{DefGs} are well-defined.

 Finally, it is obvious that $\MGO$ is a linear subspace of $\MO$ and $\MGPO$ is its convex subset.
\end{proof}

Next, we {define operator $\Y$ for} measures. 
Let 
\begin{align}
 \psi:\Rnonnegative\times\Rnonnegative\to[0,1],\quad \psi(a,b):=\begin{cases}\frac{a}{a+b}&\text{for }(a,b)\neq(0,0),\\
1&\text{for }(a,b)=(0,0).\end{cases}\label{psi}
\end{align}
The following Lemma can be easily verified by induction.
\begin{Lemma}\label[Lemma]{Lempsi} Let $\psi$ be as defined in \cref{psi}. Then $\psi\in C^{\infty}((\Rnonnegative\times\Rnonnegative)\backslash\{(0,0)\})$, and for all $0\leq k_1\leq k$ it holds that
\begin{align*}
 \partial^{k}_{a^{k_1}b^{k-k_1}}\psi(a,b)=\frac{C_{k_1,k,1}a+C_{k_1,k,2}b}{(a+b)^{k+1}}\quad\text{for all }(a,b)\in(\Rnonnegative\times\Rnonnegative)\backslash\{(0,0)\}
\end{align*}
and 
\begin{align}
 \left|\partial^{k}_{a^{k_1}b^{k-k_1}}\psi(a,b)\right|\leq C_k(a+b)^{-k}\quad\text{for all }(a,b)\in(\Rnonnegative\times\Rnonnegative)\backslash\{(0,0)\}\label{estdirpsi}
\end{align}
for some constants
$C_{k_1,k,1},C_{k_1,k,2}\in\R$ and  $C_k>0$.
\end{Lemma}
For $\mu\in \MGPO$ and a Borel function $w:\oOm\to [0,1]${,} we set
\begin{align}
 \Y(\mu,w):=\psi(\G^{+}((1-w)\mu),\G^{-}(w\mu)).\nonumber%
\end{align}
The subsequent Lemma infers that $\Y$ is well-defined and measurable.
\begin{Lemma}\label[Lemma]{LemYwd}
 Let $\mu\in\MGPO$, $\nu\in\MPO$, and $\nu\leq\mu$. Then $$\psi(\G^{+}(\mu-\nu),\G^{-}\nu):\oOm\to [0,1]$$ is a well-defined Borel function.
\end{Lemma}
\begin{proof}
 Due to the assumptions on $\mu$ and $\nu$, we have $\mu-\nu,\nu\in\MGPO$, so \cref{LemGwd} applies and yields that  $\G^{+}(\mu-\nu),\G^{-}\nu:\oOm\to \Rnonnegative$ are well-defined Borel functions. Further, $\psi$ is continuous at every point of its domain of definition apart from $(0,0)$. Hence it is Borel.
 
 Altogether, we have $\psi(\G^{+}(\mu-\nu),\G^{-}\nu):\oOm\to [0,1]$ is well-defined and, as a composition of Borel functions, it is also Borel.
\end{proof}

{To study directional derivatives of $\Y$, we introduce} sets of 'admissible' directions
\begin{alignat}{7}
 &{\K_{\mu}}:=\{{k_{\mu}}\in \MKO:&&\quad\mu+s_0{k_{\mu}}\in\MPO&&\ \text{ for some }s_0>0\}&&\qquad\text{for }\mu\in\MPO,\nonumber%
\\
 &\HH_w:=\{h:\oOm\to[-1,1]\text{ Borel}:&&\quad 0\leq w+s_0h\leq 1\text{ in }\oOm&&\ \text{ for some }s_0>0\}&&\qquad \text{for }w:\oOm\to[0,1]\text{ Borel}.\nonumber%
\end{alignat}
{Our next Lemma establishes Gateaux differentiability at a spatial point $x$.}
\begin{Lemma}\label[Lemma]{LemDD}
Assume  $\mu_0\in\MGPO$, $w_0:\oOm\to[0,1]$ Borel, and $x\in\R^d${,} such that
\begin{align}
 \mu_0(\oOm\cap B_1(x))>0.\label{PosB}
\end{align}
{Then, 
\begin{align}
\G^+((1-{w_0})\mu_0)(x)+\G^-({w_0}\mu_0)(x)>0, \label{PosBx}                     
\end{align}
and} functional 
$$(\mu,w)\mapsto\Y(\mu, w)(x)$$
possesses Gateaux partial differentials of any order with respect to $\mu$ and $w$ at $(\mu_0,w_0)$ along the directions from ${\K_{\mu_0}}\cap\MGO$ and $\HH_{w_0}$, respectively. In particular,
we have the following formulas:
\begin{subequations}\label{DerYx}
 \begin{align}
  (\partial_w\Y(\mu_0, {w_0})(x))h_1=& 
  D\psi(\G^{+}((1-{{w_0}})\mu_0),\G^{-}({{w_0}}\mu_0))\cdot(-\G^{+}(h_1\mu_0),\G^{-}(h_1\mu_0))(x), \label{der1}\\
  (\partial_{\mu}\Y(\mu_0, {w_0})(x)){k_{\mu_0}}=& 
  D\psi(\G^{+}((1-{{w_0}})\mu_0),\G^{-}({{w_0}}\mu_0))\cdot(\G^{+}((1-{{w_0}}){k_{\mu_0}}),\G^{-}({w_0}{k_{\mu_0}}))(x), \label{der2}
  \end{align}
  and
  \begin{align}
  (\partial_w(\partial_w\Y(\mu_0, {w_0})(x))h_1)h_2
  =& D^2\psi(\G^{+}((1-{{w_0}})\mu_0),\G^{-}({{w_0}}\mu_0))(-\G^{+}(h_1\mu_0),\G^{-}(h_1\mu_0))^T\nonumber\\
  &\cdot (-\G^{+}(h_2\mu_0),\G^{-}(h_2\mu_0))^T(x),\label{der3}\\
  (\partial_{\mu}(\partial_w\Y(\mu_0, {w_0})(x))h_1){k_{\mu_0}}
  =& D^2\psi(\G^{+}((1-{{w_0}})\mu_0),\G^{-}({{w_0}}\mu_0))(-\G^{+}(h_1\mu_0),\G^{-}(h_1\mu_0))^T\nonumber\\
  &\cdot (\G^{+}((1-{{w_0}}){k_{\mu_0}}),\G^{-}({w_0}{k_{\mu_0}}))^T(x)\nonumber\\
  &+D\psi(\G^{+}((1-{{w_0}})\mu_0),\G^{-}({{w_0}}\mu_0))\cdot(-\G^{+}(h_1{k_{\mu_0}}),\G^{-}(h_1{k_{\mu_0}}))(x)\label{der4}
 \end{align}
 \end{subequations}
 for all ${k_{\mu_0}}\in {\K_{\mu_0}}\cap\MGO$ and $h_1,h_2\in \HH_{w_0}$.
\end{Lemma}
\begin{proof}
Let ${k_{\mu_0}}\in{\K_{\mu_0}}\cap\MGO$ and let $s_0>0$ be a number such that $\mu_0+s_0{k_{\mu_0}}\in\MGPO$. Since $0\in{\K_{\mu_0}}$ and $\MGPO$ is convex, we have $\mu_0+s{k_{\mu_0}}\in\MGPO$ for all  $s\in[0,s_0]$.
Similarly, for $h\in\HH_{w_0}$ 
we obtain 
$0\leq w_0+sh\leq 1$ 
for all  $s\in[0,{s_0}]$. Therefore, taking the limits necessary to compute partial directional derivatives along such directions is possible. 

{Next, we observe that \cref{PosB} implies \cref{PosBx}.}
This is because{, firstly,} $\varphi^\pm$ and $K^\pm$ are strictly positive in $B_1(x)$ ({compare} \cref{Kpm,AssPhi}), so that the kernels $G^\pm(x,\cdot)$ are strictly positive there{, and secondly,} $1-w_0,w_0\geq 0$, and $(1-w_0)\mu_0+w_0\mu_0=\mu_0$.  Consequently, $\G^+((1-{w_0})\mu_0)(x)\geq0$ and $\G^-({w_0}\mu_0)(x)\geq0$ cannot be zero at the same time.

Thus,  $\psi$ is infinitely many times differentiable at $(\G^{+}((1-{w_0})\mu_0),\G^{-}({w_0}\mu_0))(x)$. This, the fact that $(\mu,w)\mapsto\G^\pm(w\mu)$ is bilinear, and the chain rule together imply the infinite directional differentiability of $(\mu, w)\to\Y(\mu, w)(x)$ at $(\mu_0,w_0)$ and formulas  \cref{DerYx}.

\end{proof}

\subsection{{Properties of \texorpdfstring{$\Y$}{Y} for \texorpdfstring{$\Om$}{} a small ball}}\label{SecYball}
In this Subsection we assume 
\begin{align}
 \Omega=\Br\qquad\text{for some }\rho\in\left(0,\frac{1}{2}\right).\label{OmBall}
\end{align}
\begin{Notation}
 To simplify the notation, we do not explicitly mention the dependence on $\rho$ for constants, mappings, and sets that we introduce in this and subsequent Subsections.
\end{Notation}
The properties of operator $\Y$ that we derive below for this special case allow to establish local existence of \cref{LambdaEqmu} in the subsequent \cref{SecExEq2}.

As previously observed in \cref{SecCh}\cref{ChI3}, the smallness of the diameter of domain $\Omega$ allows to avoid singularities in $G^-$, as well as zeros in the denominator of $\Y$, ensuring, as we see below in \cref{LemMG} that property \cref{DefG2} holds in the whole of $\oOm$. We formulate a Lemma that addresses the properties of $G^\pm$ in this case.
\begin{Lemma}\label[Lemma]{RemG2}
 Let $\rho\in\left(0,{1}/{2}\right)$ hold. Then,
 \begin{subequations}\label{PropGsBall}
\begin{align}
&G^{\pm}\in C^2(\oB\times\oB),\label{regC2}\\
 &\|G^{\pm}\|_{C^2(\oB\times \oB)}\leq\Cl{Cro},\label{estGG}\\
 &\Cl{CGlb}:=\min\left\{G^{\pm}(x,y):\ (x,y)\in \oB\times \oB\right\}>0.\label{Glb}
 \end{align}
\end{subequations}
\end{Lemma}

\begin{proof}
Since $\rho\in\left(0,{1}/{2}\right)$, we have with \cref{DefG2,AssPhi,aabb} that
\begin{align}
 |x-y|<1\qquad\text{for all }(x,y)\in \oB\times\oB.\label{xy}
\end{align}
Combined with assumptions $0<K^{\pm}\in C^2$ and  $0<\varphi^\pm\in C^2[0,1)$, \cref{xy} yields properties \cref{PropGsBall}.

\end{proof}

Our next Lemma collects several consequences of \cref{OmBall} for operators $\G^\pm$ and $\Y$.

\begin{Lemma}\label[Lemma]{LemMG}
Let  $\rho\in\left(0,{1}/{2}\right)$. Then:
 \begin{enumerate} 
\item\label{itemM1} $\MB={\mathcal M}_G(\oB)$;
\item\label{itemM2} if $\mu\in{\mathcal M}^+(\oB)$, then 
\begin{align}
 \mu(\oB\cap B_1(\cdot))=|\mu|\qquad\text{in }\oB\label{estmu}
\end{align}  
and for all {$w:\oB\to[0,1]$ Borel}
\begin{align} 
 \G^{+}((1-w)\mu)+\G^{-}(w\mu)
 \geq& \Cl{B3}|\mu|\qquad\text{in }\oB;\label{lowerbnd}
\end{align}
\item\label{item3} $\G^{\pm}\in L(\MB; C^2(\oB))$ and 
\begin{align}
\|\G^{\pm}\|_{L(\MB; C^2(\oB))}\leq \Cr{Cro};\label{estnorm1}
\end{align}
\item\label{item4} $\G^{\pm}((\cdot)(\cdot))\in B(\WInfB\times\MKB;C^1(\oB))$ and 
\begin{align}
 \|\G^{\pm}((\cdot)(\cdot))\|_{B(\WInfB\times\MKB;C^1(\oB))}\leq\Cl{CroG3};\label{estnorm2}
\end{align}
\item\label{item5} $\Y:((\MPB)\backslash\{0\})\times \{w:\oB\to[0,1]\ \text{Borel}\}\to C^2(\oB;[0,1])$.
\end{enumerate}

\end{Lemma}
\begin{proof}
 Thanks to \cref{regC2}, function $G^-(x,\cdot)$ is continuous on $\oB$. Hence,  $\G^-|\mu|<\infty$ for any $\mu\in\MB$, so \cref{itemM1} holds.
If $\mu\in\MPB$, {then  assumption $\rho\in(0,1/2)$ implies that
\begin{align*}
 \oB\subset B_1(x)\qquad\text{for all }x\in\oB,
\end{align*}
yielding \cref{estmu}. 
As to \cref{lowerbnd}, it obviously  follows with \cref{Glb,estmu,DefG}.}

{Next, we observe each $\mu\in\MB$ produces functions $\G^{\pm}\mu$ that inherit the smoothness that  $G^{\pm}$ has and that differentiation of $\G^{\pm}\mu$ is interchangeable with integration in the definitions of $G^\pm$.} This is a direct consequence of \cite[Chapter 6, Theorems 6.27-6.28]{Klenke} applied to each coordinate $x_k$ separately. In particular, \cref{regC2} implies that $G^{\pm}\mu\in C^2(\oB)$. Further, since $h\mu\in\MB$ for $h\in\WInfB$ and $\mu\in\MB$, we have that $\G^{\pm}(h\mu)\in C^2(\oB)$. Also, since $\G^{\pm}$ is linear, $\G^{\pm}((\cdot)(\cdot))$ is bilinear. Hence the linear and bilinear operators in \cref{item3,item4} are well-defined between the required spaces.  Checking  \cref{estnorm1} and thus boundedness of $\G^{\pm}$ is straightforward, given that integration and differentiation can be interchanged.
 Combining  \cref{estGG,dualKR} with the product rule, we can estimate as follows:
\begin{align}
 \underset{x\in\oB}{\max}\,|\G^{\pm}(h\mu)(x)|\leq&\underset{x\in\oB}{\max}\,\|\nabla_y (G^{+}(x,\cdot)h)\|_{(\LInfB)^n}\|\mu\|_{\MKB}\nonumber\\
 \leq&\|G^{\pm}\|_{C^1(\oB\times \oB)}
 \|h\|_{\WInfB}\|\mu\|_{\MKB}\nonumber\\\leq&\Cr{Cro}\|h\|_{\WInfB}\|\mu\|_{\MKB}\nonumber
\end{align}
and
\begin{align}
 \underset{x\in\oB}{\max}\,|\nabla_x\G^{\pm}(h\mu)(x)|\leq&\underset{x\in\oB}{\max}\,\left\|\nabla_y (\nabla_x^T G^{+}(x,\cdot)h)\right\|_{(\LInfB)^n}\|\mu\|_{\MKB}\nonumber\\
 \leq&\|G^{\pm}\|_{C^2(\oB\times \oB)}
 \|h\|_{\WInfB}\|\mu\|_{\MKB}\nonumber\\\leq&\Cr{Cro}\|h\|_{\WInfB}\|\mu\|_{\MKB}\nonumber
\end{align}
for all $h\in \WInfB$ and $\mu\in\MKB$, yielding \cref{estnorm2}. 

By \cref{item3}, $\G^+((1-w)\mu),\G^-(w\mu)\in C^2(\oB)$ for $\mu\in\MPB$ and $w:\oB\to[0,1]$ Borel. Thanks to  \cref{lowerbnd} the sum of these functions is bounded below by a positive number {if $\mu\neq0$}. Together with the chain rule and the smoothness of $\psi$ away from $(0,0)$, this implies \cref{item5} {and} completes the proof.
\end{proof}
Our next Lemma shows that under the assumption   $\rho\in\left(0,{1}/{2}\right)$ any solution to \cref{LambdaEqmu} cannot attain values $0$ and $1$ anywhere in $\oB$.
\begin{Lemma}\label[Lemma]{Lemma:reg_w}
Let  $\rho\in\left(0,{1}/{2}\right)$. Let a pair of $\mu\in (\MPB)\backslash\{0\}$ and a Borel $w:\oB\to[0,1]$ be a solution to \cref{LambdaEqmu} {at every point $x\in\oB$}. Then,
\begin{align*}
    &w\in C^2(\oB),\\
    &0<\min_{\oB}{w}\leq\max_{\oB}w<1.
\end{align*}

\end{Lemma}
\begin{proof}
By \cref{LemMG}\cref{item5}, $w=\Y(\mu,w)\in C^2(\oB)$. In particular, $w$ is continuous, so the required inequalities are satisfied provided $w\in(0,1)$ in $\oB$. 
Since 
\begin{align*}
\{\psi=0\}=\{(a,b)\in\Rnonnegative\times\Rnonnegative:\quad a=0\},                                                                                         \end{align*}
we have
\begin{align}
 \{w=0\}=&\{\Y(\mu,w)=0\}\nonumber\\
 =&\{\G^{+}((1-w)\mu)=0\}.\label{Yzero}
\end{align}
Since $(1-w)\mu\in\MPB$ and kernel $G^+$ is  strictly positive in $\oB\times\oB$ (see \cref{RemG2}), \cref{Yzero} yields that either $\{w=0\}$ is empty or $w\equiv 0$ in $\oB$ and $(1-w)\mu=0$. The latter combination is  impossible since $\mu\neq0$, so $w$ is nowhere zero.                                                                                                                                                            
In the very same way one verifies that $w$ is nowhere equal one. 

\end{proof}

Next, we combine \cref{LemDD,LemMG} in order to obtain the following result {on Gateaux differentiability at every spatial point $x$}.
\begin{Lemma}\label[Lemma]{LemDerbnds}
Let  $\rho\in\left(0,{1}/{2}\right)$. Let 
 $\mu_0\in (\MPB)\backslash\{0\}$ and $w_0\in W^{1,\infty}(\Br;(0,1))$.
 Then, functional 
$(\mu,w)\mapsto\Y(\mu, w)(x)$
possesses Gateaux partial differentials of any order with respect to $\mu$ and $w$ at $(\mu_0,w_0)$ along the directions from ${\K_{\mu_0}}$ and $\WInfB$, respectively, for every $x\in\oB$. As functions of $x$, {these} Gateaux differentials belong to $C^2(\oB)$ and formulas \cref{DerYx} are satisfied together with inequalities
\begin{subequations}\label{estGatods}
 \begin{align}
 \|(\partial_w\Y(\mu_0, {w_0})(\cdot))h_1\|_{{\WInfB}}\leq&\Cr{CCro}\|h_1\|_{\LInfB},\label{estdirYw}\\
 \|(\partial_{\mu}\Y(\mu_0, {w_0})(\cdot)){k_{\mu_0}}\|_{{\WInfB}}\leq& 
  \Cr{CCro}|\mu_0|^{-1}\left(1+\|\nabla w_0\|_{\LInfB}\right)\|{k_{\mu_0}}\|_{\MKB},\label{estdirYmu}\\
\|(\partial_w(\partial_w\Y(\mu_0, w_0)(\cdot))h_1)h_2\|_{{\WInfB}}\leq&\Cr{CCro}\|h_1\|_{\LInfB}\|h_2\|_{\LInfB},\label{estD11x}\\
 \|(\partial_{\mu}(\partial_w\Y(\mu_0, w_0)(\cdot))h_1){k_{\mu_0}}\|_{{\WInfB}}\leq&\Cr{CCro}|\mu_0|^{-1}\left(1+\|\nabla w_0\|_{\LInfB}\right)\|h_1\|_{\WInfB}\|{k_{\mu_0}}\|_{\MKB}.\label{estD12x}
\end{align}
\end{subequations}

\end{Lemma}
\begin{proof} 
Since $w_0\in W^{1,\infty}(\Br;(0,1))\subset C(\oB;(0,1))$, we have 
$$0<\min_{\oB}w_0\leq\max_{\oB}w_0<1,$$ 
hence $\HH_{w_0}=\WInfB$. 
Further, \cref{LemMG}\cref{itemM1,itemM2} yield   
that $\mu_0\in{\mathcal M}_G^+(\oB)$, \cref{estmu} holds, and ${\K_{\mu_0}}\cap{\mathcal M}_G(\oB)={\K_{\mu_0}}$. 
These observations allow us to  apply  \cref{LemDD}, yielding existence of the  Gateaux partial differentials of any order for functional $(\mu, w)\to\Y(\mu, w)(x)$ and bounds  \cref{DerYx} along the directions from ${\K_{\mu_0}}\times\WInfB$ for all $x\in \oB$. Furthermore, the resulting differentials are two times continuously differentiable functions of $x$ due to the chain rule, the smoothness of $\psi$ away from $(0,0)$, and the regularity provided by \cref{LemMG}\cref{item3}.

It remains to establish  estimates \cref{estGatods}. 
First, using \cref{estnorm1,estnorm2} and the assumptions on $w_0,\mu_0$, and $h_i$, we find that 
\begin{subequations}\label{EstSep}
\begin{align}
\|\G^{+}((1-w_0)\mu_0)\|_{{\WInfB}}
\leq&\Cr{Cro}\|1-w_0\|_{\LInfB}|\mu_0|\nonumber\\
\leq&\Cr{Cro}|\mu_0|,\\
\|\G^{-}(w_0\mu_0)\|_{{\WInfB}}\leq&\Cr{Cro}\|w_0\|_{\LInfB}|\mu_0|\nonumber\\
 \leq&\Cr{Cro}|\mu_0|,\\
\|\G^{\pm}(h_i\mu_0)\|_{{\WInfB}}\leq&\Cr{Cro}\|h_i\|_{\LInfB}|\mu_0|,\qquad i\in\{1,2\},\\
\|\G^{+}((1-w_0){k_{\mu_0}})\|_{{\WInfB}}\leq&\Cr{CroG3}\|1-w_0\|_{\WInfB}\|{k_{\mu_0}}\|_{\MKB}\nonumber\\
\leq&\Cr{CroG3}\left(1+\|\nabla w_0\|_{\LInfB}\right)\|{k_{\mu_0}}\|_{\MKB},\\
\|\G^{-}(w_0{k_{\mu_0}})\|_{{\WInfB}}\leq&\Cr{CroG3}\|w_0\|_{\WInfB}\|{k_{\mu_0}}\|_{\MKB},\\
\|\G^{\pm}(h_1{k_{\mu_0}})\|_{{\WInfB}}\leq&\Cr{CroG3}\|h_1\|_{\WInfB}\|{k_{\mu_0}}\|_{\MKB}.
\end{align}
\end{subequations}
Combining \cref{estdirpsi,lowerbnd,EstSep} and using the product and chain rules where  necessary, we can estimate the directional derivatives {from \cref{DerYx}} in $\oB$ as follows: 
\begin{subequations}\label{ests}
\begin{align}
  &|(\partial_w\Y(\mu_0, {w_0})(\cdot))h_1|\nonumber\\
  \leq& 
  |D\psi(\G^{+}((1-{{w_0}})\mu_0),\G^{-}({{w_0}}\mu_0))||(-\G^{+}(h_1\mu_0),\G^{-}(h_1\mu_0))|\nonumber\\
  \leq& \Cr{CCro}\|h_1\|_{\LInfB},
  \end{align}
\begin{align}
  &|\nabla_x((\partial_w\Y(\mu_0, {w_0})(\cdot))h_1)|\nonumber\\
  \leq& 
  |D^2\psi(\G^{+}((1-{{w_0}})\mu_0),\G^{-}({{w_0}}\mu_0))|(|\nabla_x \G^{+}((1-w_0)\mu_0)|+|\nabla_x\G^{-}(w_0\mu_0)|)|(-\G^{+}(h_1\mu_0),\G^{-}(h_1\mu_0))|\nonumber\\
  &+|D\psi(\G^{+}((1-{{w_0}})\mu_0),\G^{-}({{w_0}}\mu_0))||(-\nabla_x\G^{+}(h_1\mu_0),\nabla_x\G^{-}(h_1\mu_0))|\nonumber\\
  \leq& \Cr{CCro}\|h_1\|_{\LInfB},
  \end{align}
\begin{align}
   &|\partial_{\mu}\Y(\mu_0, {w_0})(\cdot)){k_{\mu_0}}|\nonumber\\
   \leq& |D\psi(\G^{+}((1-{{w_0}})\mu_0),\G^{-}({{w_0}}\mu_0))||(\G^{+}((1-{{w_0}}){k_{\mu_0}}),\G^{-}({w_0}{k_{\mu_0}}))|\nonumber\\
   \leq& \Cr{CCro}|\mu_0|^{-1}\left(1+\|\nabla w_0\|_{\LInfB}\right)\|{k_{\mu_0}}\|_{\MKB},
  \end{align}
  \begin{align}
   &|\nabla_x(\partial_{\mu}\Y(\mu_0, {w_0})(\cdot)){k_{\mu_0}})|\nonumber\\
   \leq& |D^2\psi(\G^{+}((1-{{w_0}})\mu_0),\G^{-}({{w_0}}\mu_0))|(|\nabla_x \G^{+}((1-w_0)\mu_0)|+|\nabla_x\G^{-}(w_0\mu_0)|)\nonumber\\
   &\cdot|(\G^{+}((1-{{w_0}}){k_{\mu_0}}),\G^{-}({w_0}{k_{\mu_0}}))|\nonumber\\
   &+|D\psi(\G^{+}((1-{{w_0}})\mu_0),\G^{-}({{w_0}}\mu_0))||(\nabla_x\G^{+}((1-{{w_0}}){k_{\mu_0}}),\nabla_x\G^{-}({w_0}{k_{\mu_0}}))|\nonumber\\
   \leq& \Cr{CCro}|\mu_0|^{-1}\left(1+\|\nabla w_0\|_{\LInfB}\right)\|{k_{\mu_0}}\|_{\MKB},
  \end{align}
\begin{align}
  &|(\partial_w(\partial_w\Y(\mu_0, w_0)(\cdot))h_1)h_2|\nonumber\\
  \leq & |D^2\psi(\G^{+}((1-w_0)\mu_0),\G^{-}(w_0\mu_0))||(-\G^{+}(h_1\mu_0),\G^{-}(h_1\mu_0))||(-\G^{+}(h_2\mu_0),\G^{-}(h_2\mu_0))|\nonumber\\
  \leq&\Cl{CCro}\|h_1\|_{\LInfB}\|h_2\|_{\LInfB},
\end{align}
\begin{align}
  &|\nabla_x((\partial_w(\partial_w\Y(\mu_0, w_0)(\cdot))h_1)h_2)|\nonumber\\
  \leq&
  |D^3\psi(\G^{+}((1-w_0)\mu_0),\G^{-}(w_0\mu_0))|(|\nabla_x \G^{+}((1-w_0)\mu_0)|+|\nabla_x\G^{-}(w_0\mu_0)|)|(-\G^{+}(h_1\mu_0),\G^{-}(h_1\mu_0))|\nonumber\\
  &\cdot|(-\G^{+}(h_2\mu_0),\G^{-}(h_2\mu_0))|+|D^2\psi(\G^{+}((1-w_0)\mu_0),\G^{-}(w_0\mu_0))|\nonumber\\
  &\cdot(|(-\G^{+}(h_1\mu_0),\G^{-}(h_1\mu_0))||(-\nabla_x\G^{+}(h_2\mu_0),\nabla_x\G^{-}(h_2\mu_0))|\nonumber\\
  &\phantom{\cdot(}+|(-\nabla_x\G^{+}(h_1\mu_0),\nabla_x\G^{-}(h_1\mu_0))||(-\G^{+}(h_2\mu_0),\G^{-}(h_2\mu_0))|)\nonumber\\
  \leq&\Cr{CCro}\|h_1\|_{\LInfB}\|h_2\|_{\LInfB},
  \end{align}
  \begin{align}
  &|(\partial_{\mu}(\partial_w\Y(\mu_0, w_0)(\cdot))h_1){k_{\mu_0}}|\nonumber\\
  \leq& |D^2\psi(\G^{+}((1-w_0)\mu_0),\G^{-}(w_0\mu_0))||(-\G^{+}(h_1\mu_0),\G^{-}(h_1\mu_0))||(\G^{+}((1-w_0){k_{\mu_0}}),\G^{-}(w{k_{\mu_0}}))|\nonumber\\
  &+|D\psi(\G^{+}((1-w_0)\mu_0),\G^{-}(w_0\mu_0))||(-\G^{+}(h_1{k_{\mu_0}}),\G^{-}(h_1{k_{\mu_0}}))|\nonumber\\
  \leq&\Cr{CCro}|\mu_0|^{-1}\left(1+\|\nabla w_0\|_{\LInfB}\right)\|h_1\|_{\WInfB}\|{k_{\mu_0}}\|_{\MKB},
 \end{align}
and
\begin{align}
  &|\nabla_x((\partial_{\mu}(\partial_w\Y(\mu_0, w_0)(\cdot))h_1){k_{\mu_0}})|\nonumber\\
  \leq& 
  |D^3\psi(\G^{+}((1-w_0)\mu_0),\G^{-}(w_0\mu_0))|(|\nabla_x \G^{+}((1-w_0)\mu_0)|+|\nabla_x\G^{-}(w_0\mu_0)|)
  |(-\G^{+}(h_1\mu_0),\G^{-}(h_1\mu_0))|\nonumber\\
  &\cdot|\G^{+}((1-w_0){k_{\mu_0}}),\G^{-}(w{k_{\mu_0}})|+  
  |D^2\psi(\G^{+}((1-w_0)\mu_0),\G^{-}(w_0\mu_0))|\nonumber\\
  &\cdot(|(-\G^{+}(h_1\mu_0),\G^{-}(h_1\mu_0))||(\nabla_x\G^{+}((1-w_0){k_{\mu_0}}),\nabla_x\G^{-}(w{k_{\mu_0}}))|\nonumber\\
  &\phantom{\cdot(} +|(-\nabla_x\G^{+}(h_1\mu_0),\nabla_x\G^{-}(h_1\mu_0))||(\G^{+}((1-w_0){k_{\mu_0}}),\G^{-}(w{k_{\mu_0}}))|)\nonumber\\
  &+|D\psi(\G^{+}((1-w_0)\mu_0),\G^{-}(w_0\mu_0))||(-\nabla_x\G^{+}(h_1{k_{\mu_0}}),\nabla_x\G^{-}(h_1{k_{\mu_0}}))|
  \nonumber\\
  &+|D^2\psi(\G^{+}((1-w_0)\mu_0),\G^{-}(w_0\mu_0))|(|\nabla_x \G^{+}((1-w_0)\mu_0)|+|\nabla_x\G^{-}(w_0\mu_0)|)\nonumber\\
  &\cdot|(-\G^{+}(h_1{k_{\mu_0}}),\G^{-}(h_1{k_{\mu_0}}))|\nonumber\\
  \leq&\Cr{CCro}|\mu_0|^{-1}\left(1+\|\nabla w_0\|_{\LInfB}\right)\|h_1\|_{\WInfB}\|{k_{\mu_0}}\|_{\MKB}.
 \end{align}
 \end{subequations}
 Combining \cref{ests}, we arrive at \cref{estGatods}.

\end{proof}
The next Lemma deals with the Fr\'echet partial derivative of $\Y$ with respect to $w$. In particular, we establish an estimate (see \cref{estFre} below) that we use in \cref{SecExEq2} in order to prove our local existence result, \cref{Thm:wEqn-FP}.
\begin{Lemma}\label[Lemma]{Lem:contdwY}
Let  $\rho\in\left(0,{1}/{2}\right)$. {Then, operator $\Y$ is continuously partially Fr\'echet differentiable with respect to $w$ at each $(\mu_0,w_0)\in((\MPB)\backslash\{0\})\times W^{1,\infty}(\Br;(0,1))$. The Fr\'echet partial derivative is a compact operator  given by}
 \begin{align}
 \partial_w\Y(\mu_0,w_0)h=&(\partial_w\Y(\mu_0,w_0)(\cdot))h\qquad{\text{in }\oB}\qquad\text{for } h\in\WInfB\label{pwFre}
\end{align}
and satisfies the following estimate: for all 
\begin{align*}
&w_1,w_2\in W^{1,\infty}(\Br;(0,1)),\\
&\mu_1,\mu_2\in (\MPB)\backslash\{0\}\qquad \text{such that }|\mu_1|=|\mu_2|=|\mu_0|,\\
&h\in\WInfB                                                                                                          \end{align*}
it holds that
\begin{align}
&\left\|(\partial_w\Y(\mu_1,w_1)-\partial_w\Y(\mu_2,w_2))h\right\|_{{\WInfB}} \nonumber\\
\leq &\Cr{CCro}|\mu_0|^{-1} \left(1+\|\nabla w_1\|_{\LInfB}\right)\|h\|_{\WInfB} \|\mu_1-\mu_2\|_{\MKB} +  \Cr{CCro}\|h\|_{\LInfB} \|w_1-w_2\|_{\WInfB}.\label{estFre}
\end{align}
\end{Lemma}
\begin{proof}
\begin{Step}
 Let $h\in\WInfB$ be such that $\|h\|_{\LInfB}< \min\{\min w_0,1-\max w_0\}$. Thanks to \cref{LemDerbnds}, function $s\mapsto \Y(\mu_0,w_0+sh)(x)$ is two times continuously differentiable on $[0,1]$. 
 By Taylor's theorem for real functions, we then have
\begin{align}
&\Y(\mu_0,w_0+h)(x)-\Y(\mu_0,w_0)(x)-(\partial_w\Y(\mu_0,w_0)(x))h\nonumber\\
=& \Y(\mu_0,w_0+h)(x)-\Y(\mu_0,w_0)(x)-\frac{d}{ds}\Y(\mu_0,w_0+sh)(x)|_{s=0}\nonumber\\
=&\int_0^1(1-s)\frac{d^2}{ds^2}\Y(\mu_0,w_0+sh)(x)\,ds\nonumber\\
=&\int_0^1(1-s)(\partial_w(\partial_w\Y(\mu_0, w_0+sh)(x))h)h\,ds.\label{MV}
\end{align}
Using \cref{der3}, we compute the integrand on the right-hand side of \cref{MV}:
\begin{align}
 &(1-s)(\partial_w(\partial_w\Y(\mu_0, w_0+sh)(x))h)h\nonumber\\
  =& (1-s)D^2\psi(\G^{+}((1-{(w_0+sh)})\mu_0),\G^{-}((w_0+sh)\mu_0))(-\G^{+}(h\mu_0),\G^{-}(h\mu_0))^T\nonumber\\
  &\cdot (-\G^{+}(h\mu_0),\G^{-}(h\mu_0))^T(x).\label{integ1}
\end{align}
Exploiting the regularity of $\psi$ and the linearity and regularity of $\G^{\pm}$, we conclude that 
the right-hand side of \cref{integ1} and its spatial gradient are continuous functions of $x$ and $s$. Consequently, the spatial gradient of the right-hand side of \cref{MV} is the integral of the gradient of the integrand. Together with \cref{estD11x,MV} this implies 
\begin{align}
&\left\|\Y(\mu_0,w_0+h)-\Y(\mu_0,w_0)-(\partial_w\Y(\mu_0,w_0)(\cdot))h\right\|_{{\WInfB}}\nonumber\\
 =&\left\|\int_0^1(1-s)(\partial_w(\partial_w\Y(\mu_0, w_0+sh)(\cdot))h)h\,ds\right\|_{{\WInfB}}\nonumber\\\leq&\int_0^1(1-s)\left\|(\partial_w(\partial_w\Y(\mu_0, w_0+sh)(\cdot))h)h\right\|_{{\WInfB}}\,ds\nonumber\\
 \leq&\int_0^1(1-s)\Cr{CCro}\|h\|_{\LInfB}^2\,ds\nonumber\\
 \leq&\frac{1}{2}\Cr{CCro}\|h\|_{\LInfB}^2\nonumber\\
 =&o\left(\|h\|_{\WInfB}\right).\label{Fre1}
\end{align}
From \cref{der1} it is clear that $h\mapsto(\partial_w\Y(\mu_0,w_0)(\cdot))h$ is a linear operator. It is also continuous by \cref{estdirYw}. {By \cref{Fre1},} this operator is then the Fr\'echet derivative of $(\mu_0,w_0)$ with respect to $w$ in $\WInfB$. 

Furthermore, formula \cref{der1} shows that $\partial_w\Y(\mu_0,w_0)$ is a superposition of multiplication by a smooth function and an integral operator with a smooth kernel. As in several instances above, this is due to smoothness of $G^\pm$ and $\psi$ by \cref{RemG2,Lempsi}, respectively. Therefore, it is a superposition of bounded and  compact operators, and as such, it is compact.
\end{Step}
\begin{Step}
{Finally, we derive estimate \cref{estFre}.}
Using smoothness in the Gateaux sense {at every point $x\in \oB$ as} provided by \cref{LemDerbnds} and {once again} Taylor's theorem {for real functions}, we compute for every $x\in \oB$ that
\begin{align}
 &(\partial_w\Y(\mu_1,w_1){(x)}-\partial_w\Y(\mu_2,w_2){(x)})h\nonumber\\
=&(\partial_w\Y(\mu_1,w_1){(x)}-\partial_w\Y(\mu_2,w_1){(x)})h+(\partial_w\Y(\mu_2,w_1){(x)}-\partial_w\Y(\mu_2,w_2){(x)})h\nonumber\\
 =&\int_0^1(\partial_{\mu}(\partial_w\Y(\mu(s), w_1){(x)})h){k_{\mu}}\,ds+ \int_0^1(\partial_w(\partial_w\Y(\mu_2, w(s)){(x)})h)h_{w}\,ds, \label{pwY-split}
\end{align}
where 
\begin{alignat*}{3}
&h_{w} = w_1 - w_2,&& \qquad {k_{\mu}} = \mu_1 - \mu_2, \\
&w(s)= (1-s)w_1+sw_2, && \qquad \mu(s)=(1-s)\mu_1+s\mu_2.
\end{alignat*}
Combining \cref{pwY-split} with \cref{estD11x,estD12x}, we can estimate as follows:
\begin{align*}
&\left\|(\partial_w\Y(\mu_1,w_1)(\cdot)-\partial_w\Y(\mu_2,w_2)(\cdot))h\right\|_{{\WInfB}}\\
\leq & \left\| \int_0^1(\partial_{\mu}(\partial_w\Y(\mu(s), w_1)(\cdot))h){k_{\mu}}\,ds\right\|_{{\WInfB}} +  \left\|\int_0^1(\partial_w(\partial_w\Y(\mu_2, w(s))(\cdot))h)h_{w}\,ds \right\|_{{\WInfB}}\\
\leq &\int_0^1 \Cr{CCro}|\mu(s)|^{-1}\left(1+\|\nabla w_1\|_{\LInfB}\right)\|h\|_{\WInfB}\|{k_{\mu}}\|_{\MKB}{\, ds} + \Cr{CCro}\|h\|_{\LInfB}\|h_{w}\|_{\LInfB} \\
=& \Cr{CCro}|\mu_0|^{-1}\left(1+\|\nabla w_1\|_{\LInfB}\right)\|h\|_{\WInfB} \|{k_{\mu}}\|_{\MKB} + \Cr{CCro}\|h\|_{\LInfB}\|h_{w}\|_{\LInfB},
\end{align*}
 {so \cref{estFre} holds. It remains to observe that this estimate directly implies continuity of the Fr\'echet derivative $\partial_w\Y$.} 
 \end{Step}
\end{proof}
We conclude this Subsection with a result for the case of a spatial domain extension.
\begin{Lemma}%
 Let  $0<\rho_0<\rho<{1}/{2}$. Let $(\mu_0,w_0)$ be such that 
 \begin{align}
  &\mu_0\in ({\mathcal M}^+(\oBo))\backslash\{0\},\nonumber\\
  &w_0\in W^{1,\infty}(B_{\rho_0};(0,1)),\nonumber\\
  &(\mu_0,w_0)\text{ satisfies  \cref{LambdaEqmu}  pointwise in }\oBo.\label{eqLam2}
 \end{align}
Define extensions
\begin{subequations}
\begin{align}
  \mu_0:=&0\qquad\text{in }\oB\backslash\oBo,\label{extmu0}\\
\overline w_0:=&\begin{cases}
            w_0&\text{in }\oBo,\\
            0&\text{in }\oB\backslash \oBo,
           \end{cases}\label{extm01}\\
           w_0:=&\Y_{\rho}(\mu_0,\overline{w}_0)\qquad\text{in }\oB\backslash \oBo\label{w0extB}
 \end{align}
 \end{subequations}
 and operators
 \begin{align*}
  \X_r:=&id-\partial_w\Y_r(\mu_0,w_0)\qquad \text{in }W^{1,\infty}(B_r)\qquad \text{for }r\in\{\rho,\rho_0\},
 \end{align*}
 where $\Y_r$ denotes the realisation of $\Y$ corresponding to the spatial domain  $\overline{B_r}$.
Then, $w_0\in C^2(\oB)$ and equation \cref{LambdaEqmu} is satisfied pointwise in $\oB$. 
Furthermore, $\X_{\rho}$ is invertible if and only if $\X_{\rho_0}$ is invertible.
\begin{proof}
 Since $\rho<1/2$, $\Y(\mu_0,\overline{w}_0)\in C^2(\oB)$ by \cref{LemMG}\cref{item5}.  With \cref{w0extB,eqLam2,extm01} we deduce that \cref{LambdaEqmu} is satisfied in $\oB$. 
 
 It remains to verify the equivalence of invertibility in  $\WInfB$ and $\WInfBo$. 
 By \cref{Lem:contdwY},  $\X_{\rho_0}$ and $\X_{\rho}$ are Fredholm. Hence, it suffices to check 
\begin{align}
 \Ker_{\WInfB}(\X_{\rho})\neq\{0\}\qquad\Leftrightarrow\qquad  \Ker_{\WInfBo}(\X_{\rho_0})\neq\{0\}.\label{eqKer}
\end{align}
Combining \cref{der1,pwFre,extmu0}, we obtain for all $h,h_1,h_2\in\WInfB$ with $h_1= h_2$ in $\oBo$ that 
\begin{subequations}
\begin{alignat}{3}
 \partial_w\Y_{\rho}(\mu_0,w_0)h=&\partial_w\Y_{\rho_0}(\mu_0,w_0) h&&\qquad\text{in }\oBo,\label{Ext0}\\
 \partial_w\Y_{\rho}(\mu_0,w_0)h_1=&\partial_w\Y_{\rho}(\mu_0,w_0) h_2&&\qquad\text{in }\oB.\label{Ext12}
\end{alignat}
\end{subequations}
With \cref{Ext0} it follows directly that if $h\in \Ker_{\WInfB}(\X_{\rho})$, then $h\in \Ker_{\WInfBo}(\X_{\rho_0})$ as well.
Conversely, let
\begin{align}
h_0\in \Ker_{\WInfBo}(\X_{\rho_0}).\label{inKer}
\end{align}
Choose $\overline h_0\in\WInfB$ such that it coincides with $h_0$ on $\oBo$ and set
\begin{align}
\widetilde h_0:=&\partial_w\Y_{\rho}(\mu_0,w_0)\overline h_0\qquad\text{in }\oB.\label{h2}
\end{align}
Together, \cref{Ext0,h2,inKer} imply 
\begin{alignat}{3}
\widetilde h_0=&\partial_w\Y_{\rho}(\mu_0,w_0)\overline h_0&&\nonumber\\
=&\partial_w\Y_{\rho_0}(\mu_0,w_0) h_0&&\nonumber\\
=&h_0&&\nonumber\\
=&\overline h_0
&&\qquad\text{in }\oBo.\label{h3}
\end{alignat}
Finally, \cref{h3,Ext12,h2} yield
\begin{alignat}{3}
 \widetilde h_0=&\partial_w\Y_{\rho}(\mu_0,w_0)\overline h_0&&\nonumber\\
 =&\partial_w\Y_{\rho}(\mu_0,w_0)\widetilde h_0&&\qquad\text{in }\oB,\nonumber
\end{alignat}
meaning that $\widetilde h_0\in \Ker_{\WInfB}(\X_{\rho})$.
Thus, we have proved that \cref{eqKer} holds.
\end{proof}

\end{Lemma}

\subsection{Local well-posedness of  \texorpdfstring{\cref{LambdaEqmu}}{} for \texorpdfstring{$\Om$}{} a small ball (proof of \texorpdfstring{\cref{Thm:wEqn-FP}}{})}\label{SecExEq2}
In this Subsection we use the findings in \cref{SecYball,SecY} to prove \cref{Thm:wEqn-FP} on local well-posedness of \cref{LambdaEqmu}.
Our proof is based on an argument   similar to the one used in the proof of the implicit mapping theorem, see the discussion in \cref{SecCh}\cref{ChI5}.
\begin{proof}[Proof of \cref{Thm:wEqn-FP}]
Define constants
\begin{subequations}\label{IMT:constants}
 \begin{align}
 \Cl{Cw0}:=&\frac{1}{2}\min\left\{\min_{\oB}w_0,1-\max_{\oB}w_0\right\},\label{Cminmax}\\
\Cl{CC6}:=& \left\|\X^{-1}\right\|_{{L(\WInfB)}}\Cr{CCro},\\
\Cr{R1}:=&\min\left\{\left(1+2\Cr{CC6}\right)^{-1},\Cr{Cw0}\right\},\label{R}\\
\Cr{R2}  :=& \frac{1}{2}\Cr{CC6}^{-1}\Cr{R1}|\mu_0|\left(1+\|\nabla w_0\|_{\LInfB}\right)^{-1},\label{Rtil}\\
\Cr{F-contin-mu} :=&  \frac12\Cr{CC6}|\mu_0|^{-1}\left(1+\|\nabla w_0\|_{\LInfB}+\Cr{R1}\right).
\end{align}   
\end{subequations}
Consider operator
\begin{align*}
 \Cl[WW]{OpWu}(\mu,w):=w-\X^{-1}(w-\Y(\mu,w)).
\end{align*}
Obviously, solutions of \cref{LambdaEqmu} are exactly  
the solutions of 
\begin{align}
 w=\Cr{OpWu}(\mu,w).\label{EqW1}
\end{align}
In order to solve \cref{EqW1}, we prove that $\Cr{OpWu}$ is a contraction in its second variable and subsequently use Banach's fixed point theorem.
\begin{Step}
To begin with, we establish the following estimates for all $\mu,\mu_1,\mu_2\in \Cr{SU2}$ and $w,w_1,w_2\in \Cr{W3}$:
\begin{subequations}
\begin{align}
 \|\Cr{OpWu}(\mu,w_2)-\Cr{OpWu}(\mu,w_1)\|_{\WInfB}\leq & \frac12\|w_1-w_2\|_{\WInfB},\label{Fcontr}\\
 \|\Cr{OpWu}(\mu_1,w)-\Cr{OpWu}(\mu_2,w)\|_{\WInfB} \leq & \frac12\Cr{F-contin-mu} \|\mu_1-\mu_2\|_{\MKB}\label{Fconti_mu1},\\
\|\Cr{OpWu}(\mu,w)-w_0\|_{\WInfB}\leq & \Cr{R1}.\label{Fconta}
\end{align}
\end{subequations}
Using the generalised Taylor's theorem \cite[Chapter 4, \S4.6, Theorem 4.A]{ZeidlerFP} and chain rule \cite[Chapter 4, \S4.3, Proposition 4.10]{ZeidlerFP}, \cref{Lem:contdwY}, and the definition of $\Cr{OpWu}$, we compute
\begin{subequations}
\begin{align}
 \Cr{OpWu}(\mu,w_2)-\Cr{OpWu}(\mu,w_1)=&\int_0^1\partial_w\Cr{OpWu}(\mu,w(s_1))h_1\,ds_1\nonumber\\
=&\int_0^1\X^{-1}(\partial_w\Y(\mu,w(s_1))-\partial_w\Y(\mu_0,w_0))h_1\,ds_1\qquad \text{in }\WInfB, \label{F_eqn:1}\\
 \Cr{OpWu}(\mu_1,w)-\Cr{OpWu}(\mu_2,w)=&\int_0^1\partial_{\mu}\Cr{OpWu}(\mu(s_1),w){k_{\mu}}\,ds_1\nonumber\\
 =& \int_0^1 \X^{-1}\partial_{\mu}\Y(\mu(s_1),w){k_{\mu}}\,ds_1 \qquad \text{in }\WInfB, \label{F_eqn:2}
\end{align}    
\end{subequations}
where 
\begin{subequations}\label{hw1}
\begin{alignat}{3}
  &h_1:=w_2-w_1,\qquad  &&{k_{\mu}}:= \mu_1-\mu_2,\\
  &w(s_1):=(1-s_1)w_1+s_1w_2, \qquad &&\mu(s_1):=(1-s_1)\mu_1+\mu_2.
\end{alignat}
\end{subequations}
Combining \cref{estFre,hw1,IMT:constants,setsWM}, we can estimate the right-hand side of \cref{F_eqn:1} using, in particular, convexity of $\Cr{W3}$ and $\Cr{SU2}$, and obtain
\begin{align*}
& \|\Cr{OpWu}(\mu,w_2)-\Cr{OpWu}(\mu,w_1)\|_{\WInfB}\\
\leq & \int_0^1 \left\|\X^{-1}\right\|_{{L(\WInfB)}}\|(\partial_w\Y(\mu,w(s_1))-\partial_w\Y(\mu_0,w_0))h_1\|_{\WInfB}\, ds_1\\
\leq&\Cr{CC6}|\mu_0|^{-1} \left(1+\|\nabla w_0\|_{\LInfB}\right)\|w_1-w_2\|_{\WInfB}\|\mu-\mu_0\|_{\MKB}\\
& +\Cr{CC6} \|w_1-w_2\|_{{\LInfB}}\int_0^1\|w(s_1)-w_0\|_{\WInfB}\, ds_1\\
\leq&\Cr{CC6}\left(|\mu_0|^{-1}\left(1+\|\nabla w_0\|_{\LInfB}\right)\Cr{R2}+\Cr{R1}\right)\|w_1-w_2\|_{\WInfB}\\
\leq&\frac12\|w_1-w_2\|_{\WInfB},
\end{align*}
so \cref{Fcontr} holds. In a similar fashion, this time relying on \cref{estdirYw,pwFre} rather than \cref{estFre}, we  estimate the right-hand side of \cref{F_eqn:2}, concluding that
\begin{align}
 &\|\Cr{OpWu}(\mu_1,w)-\Cr{OpWu}(\mu_2,w)\|_{\WInfB} \nonumber\\
 \leq & \left\|\X^{-1}\right\|_{{L(\WInfB)}} \int_0^1\|\partial_{\mu}\Y(\mu(s_1),w){k_{\mu}}\|_{\WInfB}\, ds_1\nonumber\\
 \leq & \Cr{CC6} \int_0^1|\mu(s_1)|^{-1}\,ds_1\left(1+\|\nabla w\|_{\LInfB}\right)\|{k_{\mu}}\|_{\MKB}\nonumber\\
= & \Cr{CC6} |\mu_0|^{-1} \left(1+\|\nabla w\|_{\LInfB}\right)\|\mu_1-\mu_2\|_{\MKB}\label{estax1}\\
\leq & \Cr{CC6} |\mu_0|^{-1} \left(1+\|\nabla w_0\|_{\LInfB}+\Cr{R1}\right)\|\mu_1-\mu_2\|_{\MKB}\nonumber\\
=&\frac12\Cr{F-contin-mu} \|\mu_1-\mu_2\|_{\MKB},\nonumber
\end{align}
so \cref{Fconti_mu1} holds too.
Finally, we obtain \cref{Fconta} using \cref{Fcontr,estax1,Rtil} and the assumption that $(\mu_0,w_0)$ solves \cref{LambdaEqmu}:
\begin{align*}
&\|\Cr{OpWu}(\mu,w)-w_0\|_{\WInfB}\\
=&\|\Cr{OpWu}(\mu,w)-\Cr{OpWu}(\mu_0,w_0)\|_{\WInfB}\nonumber\\
\leq&\|\Cr{OpWu}(\mu,w)-\Cr{OpWu}(\mu,w_0)\|_{\WInfB}+\|\Cr{OpWu}(\mu,w_0)-\Cr{OpWu}(\mu_0,w_0)\|_{\WInfB}\nonumber\\
\leq&\frac12\|w-w_0\|_{\WInfB}+ \Cr{CC6} |\mu_0|^{-1} \left(1+\|\nabla w_0\|_{\LInfB}\right) \|\mu-\mu_0\|_{\MKB}\\
\leq&\frac12\Cr{R1}+\Cr{CC6} |\mu_0|^{-1} \left(1+\|\nabla w_0\|_{\LInfB}\right)\Cr{R2} \\
 =&\Cr{R1}.
\end{align*}
\end{Step}
\begin{Step}
Let $\mu\in \Cr{SU2}$. By \cref{LemMG}\cref{item5}, $\Y$ and, hence, $\Cr{OpWu}(\mu,\cdot)$, are well-defined in $\Cr{W3}$. Furthermore, due to \cref{Fconta,Cminmax,R}, $\Cr{OpWu}(\mu,\cdot)$ is a self-map and, due to \cref{Fcontr}, a contraction in this closed ball  of $\WInfB$. Applying Banach's fixed point theorem, we infer the existence of a unique fixed point $w$. As mentioned at the beginning of the proof, $(\mu,w)$ is then a solution to \cref{LambdaEqmu}. Invoking \cref{Lemma:reg_w}, we find that $w\in C^2(\oB;(0,1))$. 
Altogether, we conclude that the solution map 
\begin{align}
 \Cr{OpW1}: \Cr{SU2} \to \Cr{W3}\cap C^2(\oB;(0,1)),\qquad \Cr{OpW1}(\mu):=w\nonumber
\end{align}
is well-defined.

It remains to verify \cref{W-stability}. Consider $\mu_1,\mu_2\in \Cr{SU2}$ and $w_1=\Cr{OpW1}(\mu_1)$ and $w_2=\Cr{OpW1}(\mu_2)$. Combining \cref{Fcontr,Fconti_mu1}, we obtain

\begin{align*}
&\|\Cr{OpW1}(\mu_1)-\Cr{OpW1}(\mu_2)\|_{\WInfB}\nonumber\\
=& \|\Cr{OpWu}(\mu_1,w_1)-\Cr{OpWu}(\mu_2,w_2)\|_{\WInfB} \\
\leq &\|\Cr{OpWu}(\mu_1,w_1)-\Cr{OpWu}(\mu_1,w_2)\|_{\WInfB}  + \|\Cr{OpWu}(\mu_1,w_2)-\Cr{OpWu}(\mu_2,w_2)\|_{\WInfB} \\
\leq &\frac12 \|w_1-w_2\|_{\WInfB}  + \frac12\Cr{F-contin-mu}\|\mu_1-\mu_2\|_{\MKB}.
\end{align*}
This gives
\begin{align*}
\|\Cr{OpW1}(\mu_1)-\Cr{OpW1}(\mu_2)\|_{\WInfB}\leq& \Cr{F-contin-mu} \|\mu_1-\mu_2\|_{\MKB},
\end{align*}
so estimate \cref{W-stability} holds.
\end{Step}
\end{proof}
We conclude this Subsection  with a proof of \cref{CompThmLambdaEq}.
\begin{proof}[Proof of \cref{CompThmLambdaEq}]
Since $\nu\leq\mu$, the Radon-Nikodym theorem provides the existence of a Radon-Nykodym derivative $\frac{d\nu}{d\mu}$,
whereas the Lebesgue differentiation theorem implies that $w$ takes values in $[0,1]$, $\mu$-a.e. Since $(\mu,\nu)$ solves \cref{M_LambdaEq}, $(\mu,\frac{d\nu}{d\mu})$ then solves \cref{LambdaEq} by definition.
 
\end{proof}

\subsection{Solvability of \texorpdfstring{\cref{LambdaEqmu}}{} for \texorpdfstring{$\mu_0$}{} a point mass}
Let us consider a special case when $\mu_0$ is a point mass. As seen below, in this situation, equation \cref{LambdaEqmu} is explicitly solvable and, moreover, the solution pair $(\mu_0,w_0)$ satisfies \cref{Assmu0w0}.
\begin{Example}[$\mu_0$ point mass]\label{Exdelta}
 Let $\rho\in(0,1)$, $m>0$, 
 \begin{align}
    \mu_0:=m\delta_0,\nonumber                                                              \end{align} 
 and $w_0:\R^d\to[0,1]$ be Borel. Using 
 \begin{align*}
\supp(\mu_0)=\{0\}                                        \end{align*}
 and the linearity of $\G^{\pm}$, we compute
 \begin{align}
  \Y(\mu_0,w_0)=&\frac{(1-{w_0}{(0)})\G^{+}\delta_0}{(1-{w_0}{(0)})\G^{+}\delta_0+{w_0}{(0)}\G^{-}\delta_0}\qquad\text{in }\oB,\label{eqRHS}
 \end{align}
 where
 \begin{align}
  \G^\pm\delta_0=G^\pm(\cdot,0)>0.\nonumber%
 \end{align}
 due to assumptions \cref{DefG2,AssPhi,Kpm,aabb}  and $\rho<1$.
In particular, 
\begin{align}
 \Y(\mu_0,w_0){(0)}=&\frac{(1-{w_0}{(0)})\gamma^+}{(1-{w_0}{(0)})\gamma^++{w_0}{(0)}\gamma^-},\nonumber
\end{align}
where 
\begin{align}
  \gamma^\pm:=&G^{\pm}(0,0)>0.\nonumber
 \end{align} 
 Thus, $w_0$ is a solution to \cref{LambdaEq} if and only if 
 \begin{align}
  {w_0}{(0)}=&\frac{(1-{w_0}{(0)})\gamma^+}{(1-{w_0}{(0)})\gamma^++{w_0}{(0)}\gamma^-}.\label{eqwst2}
 \end{align}
 Since $0\leq {w_0}{(0)}\leq 1$, the only solution that  \cref{eqwst2} has is  
 \begin{align}
 w_0(0)=\frac{\sqrt{\gamma^+}}{\sqrt{\gamma^+}+\sqrt{\gamma^-}}\in(0,1).\label{solw}
 \end{align}
Combining \cref{solw,eqRHS}, we obtain the unique solution to \cref{LambdaEqmu} in the form 
\begin{align}
 {w_0}=&\frac{\sqrt{\gamma^-}G^+}{\sqrt{\gamma^-}G^++\sqrt{\gamma^+}G^-}(\cdot,0)\qquad\text{in }\oB.\nonumber
\end{align}
 Let $x\in \oB$. Using formula \cref{der1}, we compute
\begin{align}
 \partial_w(\Y(\mu_0,w_0)(x))h_1
 =&\frac{1}{m}\xi(x)h_1{(0)},
\label{der5}
\end{align}
where%
\begin{align}
 \xi:=&D\psi((1-w_0(0))G^+(\cdot,0),w_0(0)G^-(\cdot,0))\cdot(-G^+(\cdot,0),G^-(\cdot,0))\nonumber\\
 \in&W^{1,\infty}({\Br}).\nonumber
\end{align}
In particular,
\begin{align}
 \xi{(0)}=&D\psi((1-w_0(0))\gamma^+,w_0(0)\gamma^-)\cdot(-\gamma^+,\gamma^-)\nonumber\\
 =&D\psi\left(\frac{\sqrt{\gamma^-}}{\sqrt{\gamma^+}+\sqrt{\gamma^-}}\gamma^+,\frac{\sqrt{\gamma^+}}{\sqrt{\gamma^+}+\sqrt{\gamma^-}}\gamma^-\right)\cdot(-\gamma^+,\gamma^-)\nonumber\\
 =& -1.\label{xione}
\end{align}
Combining \cref{xione,der5,pwFre}, we conclude that operator
\begin{align*}
 \X h_1=&h_1-\partial_w\Y(\delta_0,w_0)h_1\\
 =&h_1-\frac{1}{m}\xi h_1{(0)}
\end{align*}
is invertible in $W^{1,\infty}({\Br})$, its inverse being
\begin{align*}
 \X^{(-1)}g=g+\frac{1}{m+1}\xi g{(0)}.
\end{align*}
Thus, $(\mu_0,w_0)$ satisfies all  conditions in  \cref{Assmu0w0}.
\end{Example}

\section{Well-posedness of a PDE with PM diffusion and fixed advection direction}\label{Sec:FixedAdv}
In this Section 
we consider the Cauchy problem for a PDE with the PM diffusion \cref{PM} and a fixed advection direction $V$:
\begin{subequations}\label{PMEdriftCauchy}
\begin{alignat}{3}
 &\partial_tu=\Delta u^2-\nabla\cdot\left(Vu\chi(u)\right)&&\quad\text{in }(0,T]\times\R^d,\label{PMEdrift}\\
 &u(0,\cdot)=u_0&&\quad\text{in }\R^d.\label{u02}
\end{alignat}
\end{subequations}
Later in \cref{ExPLnew} we take
\begin{align}
 V=w\left(\nabla H\star u\right),\nonumber%
\end{align}
which turns \cref{PMEdrift}  into  \cref{PLnew}. 

For $V\equiv0$, \cref{PMEdrift} is a special case of the standard porous media equation that was covered in depth in   \cite{Vazquez}. Many works also addressed \cref{PMEdrift} for either constant $V$ or $\chi\equiv1$, see e.g. references in the survey  \cite{Kalashnikov}. More recent studies have concentrated on the case $\chi\equiv1$ and $V$ that belongs to a possibly  anisotropic $L^{p,q}$-space.  We refer the interested reader to the detailed explorations in   \cite{HwangKangKim1,HwangKangKim2} as well as references therein. 

In our setting, both $\chi$ and $V$ are non-constant. Since we were unable to locate the necessary facts specifically for this particular scenario and moreover, need to account for dependencies on the parameters in the  estimates for its solutions for the purpose of our analysis 
in
\cref{SecMainProof,ExPLnew}, we provide all required results. For the reader's convenience, we include proofs in most cases even though they follow standard patterns.

The rest of the Section is decomposed into two parts. We start with developing well-posedness in \cref{SecWPPME}. In \cref{Subsection:support}, we deal with the situation when the initial value $u_0$ is compactly supported. We verify that the support of the corresponding solution cannot explode for small times. 
\subsection{Well-posedness of \texorpdfstring{\cref{PMEdriftCauchy}}{}}\label{SecWPPME}
In this Subsection we aim at  establishing  well-posedness of \cref{PMEdriftCauchy}. For $V$ and $u_0$ we consider the following set of assumptions: for some $T>0$
\begin{subequations}\label{AssV4}
\begin{align} 
 V&\in \left(L^{\infty}(E_T)\right)^d,\label{AssV1}\\
 V&\in L^{\infty,1}(E_T),\label{AssV11}\\
 \nabla \cdot V&\in L^{\infty}(E_T),\label{AssDV1}\\
  \nabla \cdot V&\in  L^{\infty,1}(E_T);\label{AssDV11}
\end{align}
\end{subequations}
and
\begin{subequations}\label{Assumpu01}
\begin{align}
 0\leq&u_0\in \LInfn\cap\LTwon,\label{Assumpu0}\\
 &u_0\in \LOnen. 
\end{align}
\end{subequations}
Not all of them are required for each result.

We introduce the following notion of {(sub/super-)} solution to \cref{PMEdriftCauchy}.
\begin{Definition}[(Sub/super-)solutions to \cref{PMEdriftCauchy}]\label[Definition]{DefPMEDrift}
Let $T>0$ and \cref{Asschi,AssV1,Assumpu0} be satisfied.
We call a function $u:[0,T]\times\R^d\to[0,\infty)$  a solution to \cref{PMEdriftCauchy} in $[0,T]\times\R^d$ if:
\begin{enumerate}
\item $u\in L^{\infty}\left(0,T;\LInfn\cap\LTwon\right)$;
\item%
  $u^{\frac{3}{2}}\in L^2\left(0,T;\HOnen\right)$;
 \item%
  $\partial_tu\in L^2\left(0,T;\HmOnen\right)$, $\partial_tu \in L^{\infty}\left(0,T;\left(W^{2,p}(\R^d)\right)'\right)$ for all $p\in[1,\infty]$;
 \item $u$ satisfies \cref{PMEdrift} in  a weak, 
 \begin{subequations}
 \begin{align}
  \left<\partial_tu,\varphi\right>=-\int_{\R^d}\left(\nabla u^2-Vu\chi(u)\right)\cdot\nabla\varphi\,dx\qquad\text{a.e. in }(0,T)\text{ for all }\varphi\in\HOnen,\label{weaksol}
 \end{align}
 and a very weak,  
 \begin{align}
  \left<\partial_tu,\varphi\right>=\int_{\R^d} u^2\Delta \varphi+Vu\chi(u)\cdot\nabla\varphi\,dx\qquad\text{a.e. in }(0,T)\text{ for all }\varphi\in \bigcup\limits_{p=1}^{\infty} W^{2,p}(\R^d),\label{veryweaksol}
 \end{align}
 \end{subequations}
 senses;
\item%
$u(0,\cdot)=u_0$ in $\HmOnen$.
\end{enumerate}
We call $u$ a subsolution (supersolution) to \cref{PMEdriftCauchy}  in $[0,T]\times\R^d$ if $=$ is replaced by $\leq$ ($\geq$) in \cref{weaksol,veryweaksol}.
\end{Definition}

Our first Lemma deals with uniqueness of solutions to \cref{PMEdriftCauchy}. 
\begin{Lemma}\label[Lemma]{LemPMEDriftEx-uniqueness}
 Under the assumptions of \cref{DefPMEDrift} solutions to \cref{PMEdriftCauchy} are unique.
\end{Lemma}
\begin{proof}
We use a version of the  duality argument that is standard when dealing with weak solutions of an  equation with PM diffusion. The nonlinearity in the advection term is handled similar to  \cite[Section 2]{EfendievSenba}, \cite[Chapter 9, Section 2]{Efendiev2013}.

  Let $u$ and $\wu$ be two solutions corresponding to $u_0$ and $V$. Subtracting the very weak formulations \cref{veryweaksol} for these solutions, we obtain for their difference 
  \begin{align}
   \left<\partial_t\left(u-\wu\right),\varphi\right>
   =&\iRn \left(u^2-\wu^2\right)\Delta\varphi\,dx+\iRn V\left(u\chi(u)-\wu\chi\left(\wu\right)\right)\cdot\nabla\varphi\,dx\label{Difeq1}
  \end{align}
for $\varphi\in \HTwon$ a.e. in $(0,T)$.  Consider the equation
 \begin{align}
  -\Delta \varphi+\varphi=u-\wu.\label{Eqphi}
 \end{align}
As is well-known, $-\Delta+\id :H^{k+2}\left(\R^d\right)\to H^k\left(\R^d\right)$ is an isometric isomorphism for all $k\in\Z$. 
   Since $u,{\wu}\in {L^{\infty,2}}(E_T)$ and $\partial_t u,{\partial_t\wu}\in L^2\left(0,T;\HmOnen\right)$, we have
\begin{align}
 &{\varphi\in L^{\infty}\left((0,T);\HTwon\right)\cap C\left([0,T];\HOnen\right),\qquad \partial_t\varphi\in L^2\left(0,T;\HOnen\right)},\nonumber\\
 &\|\varphi\|_{\HOnen}=\left\|u-\wu\right\|_{\HmOnen}\qquad\text{in }[0,T].
 \nonumber%
\end{align}
Plugging $\varphi$ into \cref{Difeq1} and integrating over $\left(0,t\right)$ for $t\in(0,T)$, we obtain {using standard calculus of functions with values in a Hilbert space}
\begin{align}
\frac{1}{2}\left[\iRn |\nabla\varphi|^2+\varphi^2\,dx\right]_0^t=&\int_0^t\frac{1}{2}\frac{d}{dt}\iRn |\nabla\varphi|^2+\varphi^2\,dx\,ds\nonumber\\
 =&\int_0^t\left<\partial_t\left(-\Delta \varphi+\varphi\right),\varphi\right>\,ds\nonumber\\
   =&-\int_0^t\iRn \left(u^2-\wu^2\right)\left(u-\wu\right)\,dx\,ds+\iRn\left(u^2-\wu^2\right)\varphi\,dx\,ds\nonumber\\ &+\int_0^t\iRn
   V\left(u\chi(u)-\wu\chi\left(\wu\right)\right)
   \cdot\nabla\varphi\,dx\,ds\qquad \text{in }[0,T].\label{Difeq2}
\end{align}
Due to the assumptions on $\chi$, we have
\begin{align}
 |z_1\chi\left(z_1\right)-z_2\chi\left(z_2\right)|=&\left|\int_{z_1}^{z_2}s\chi'\left(s\right)+\chi\left(s\right)\,ds\right|\nonumber\\
 \leq&2\|\chi'\|_{L^{\infty}((0,\infty))}\left|\int_{z_1}^{z_2}s\,ds\right|\nonumber\\
 =&\|\chi'\|_{L^{\infty}((0,\infty))}\left|z_1^2-z_2^2\right|.\label{estchi1}
\end{align}
Combining \cref{Difeq2,estchi1} with the basic algebraic  inequalities 
\begin{align}
 &\left(z_1^2-z_2^2\right)\left(z_1-z_2\right)\geq \Cl{Ctriv0} \left|z_1^{\frac{3}{2}}-z_2^{\frac{3}{2}}\right|^2,\nonumber\\
 &\left|z_1^2-z_2^2\right|\leq\Cl{Ctriv1}\left(\max\{z_1,z_2\}\right)^{\frac{1}{2}}\left|z_1^{\frac{3}{2}}-z_2^{\frac{3}{2}}\right|\nonumber
\end{align}
and using H\"older's and Young's inequalities, we estimate as follows:
\begin{align}
&\frac{1}{2}\left[\iRn |\nabla\varphi|^2+\varphi^2\,dx\right]_0^t\nonumber\\
 \leq&-\Cr{Ctriv0} \int_0^t\iRn \left|u^{\frac{3}{2}}-\wu^{\frac{3}{2}}\right|^2\,dx\,ds+\Cr{Ctriv1}\|\left(u,\wu\right)\|_{L^{\infty}\left(0,T;\left(\LInfn\right)^2\right)}^{\frac{1}{2}}\int_0^t\iRn\left|u^{\frac{3}{2}}-\wu^{\frac{3}{2}}\right||\varphi|\,dx\,ds\nonumber\\ &+\Cl{Ctr1}\||V|\|_{\left(L^{\infty}(E_T)\right)^d}{\|\left(u,\wu\right)\|_{L^{\infty}\left(0,T;\left(\LInfn\right)^2\right)}^{\frac{1}{2}}}\int_0^t\iRn\left|u^{\frac{3}{2}}-\wu^{\frac{3}{2}}\right||\nabla\varphi|\,dx\,ds\nonumber\\
 \leq&\Cl{Ctr2}\int_0^t\iRn|\nabla\varphi|^2+|\varphi|^2\,dx\,ds\qquad \text{in }[0,T],\label{Difeq3}
\end{align}
where $\Cr{Ctr2}$ depends on norms of $u,\widehat u$, $\chi$, and $V$. 
Applying Gr\"onwall's inequality to \cref{Difeq3} yields
\begin{align}
 \|\varphi\left(t\right)\|_{\HOnen}^2\leq&e^{t2\Cr{Ctr2}}\|\varphi(0)\|_{\HOnen}^2\qquad\text{in }[0,T].\label{estvarphi1}
\end{align}
Combining \cref{Eqphi,estvarphi1} and using $\varphi(0,\cdot)=u_0-u_0=0$ yields uniqueness.

\end{proof}

As is customary in such cases, we seek to obtain solutions to \cref{PMEdriftCauchy} as limits of solutions $\left(u_{\epsilon m}\right)$, \begin{align*}
\eps\in(0,1)\text{ and } m\in\N,                                              \end{align*}
 to initial-boundary value problems (IBVPs) in balls for non-degenerate diffusion-advection equations with regular coefficients,
\begin{subequations}\label{IBVPem}
\begin{alignat}{3}
 &\partial_t\uem=\eps \Delta \uem+\Delta \uem^2-\nabla\cdot\left(V_{\eps}\uem\chi_{\eps}\left(\uem\right)\right)&&\qquad\text{in }(0,T]\times \Bm,\label{PMEdriftem}\\
 &\uem=0&&\qquad\text{in }(0,T]\times S_m,\label{bcem}\\
 &\uem(0,\cdot)=u_{\eps 0}&&\qquad\text{in }\overline{\Bm},\label{u0em}
\end{alignat}
\end{subequations}
extending them to the whole space by
\begin{align}
 \uem=0\qquad\text{in }[0,T]\times \left(\R^d\backslash \overline{\Bm}\right).\nonumber
\end{align}
We make the following assumptions on the coefficient functions $\chi_{\eps}$ and $V_{\eps}$ and the  initial data $u_{\eps 0}$:
\begin{subequations}\label{Assreg_}
\begin{align}
&\chi_{\ve}\in C^2_b\left([0,\infty)\right), \chi(0)=0,\label{Assregchi} \\
&V_{\eps}\in C^{1,2}_b\left([0,T]\times \R^d;\R^d\right),\label{AssregV}\\
 & 0\leq u_{\ve0}\in C^3_0\left(\R^d\right). \label{Assregu0}    \end{align}
\end{subequations}
The main purpose of the subsequent Lemmas is to collect useful estimates for solutions to \cref{IBVPem}. Most of them continue to hold for solutions to \cref{PMEdriftCauchy} and are used in  
the proofs in \cref{ExPLnew,SecMainProof}.
\begin{Lemma}\label[Lemma]{LemSolIBVPem}
 Let \cref{Assreg_} be satisfied and let $m\geq \diam\left(\supp\left(u_{\eps 0}\right)\right)$. Then, the IBVP \cref{IBVPem} possesses a unique classical solution {$\uem\in C^{1,2}\left([0,T]\times\overline{\Bm}\right)
 $.}
 The solution satisfies the following set of estimates:
\begin{subequations}\label{estue}
\begin{align}
&0\leq\uem\leq \Cl{D1}\left(T,\chi_{\eps},V_{\eps},u_{\eps0}\right)\qquad\text{in }[0,T]\times \overline{\Bm},%
\label{estbnduem}   \\
&\|\uem\|_{L^{\infty}\left(0,T;L^2\left(\Bm\right)\right)}^2+\eps\||\nabla \uem|\|_{L^2\left(0,T;L^2\left(\Bm\right)\right)}^2+\left\|\left|\nabla \uem^{\frac{3}{2}}\right|\right\|_{L^2\left(0,T;L^2\left(\Bm\right)\right)}^2 \leq \Cl{D2}\left(T,\chi_{\eps},V_{\eps},u_{\eps0}\right),\label{estnabuem}
\\
&\|\partial_t \uem\|_{L^2\left(0,T;H^{-1}\left(\Bm\right)\right)}^2\leq \Cl{D3}\left(T,\chi_{\eps},V_{\eps},u_{\eps0}\right),\label{ptuem}
\end{align}
\end{subequations}
where
\begin{subequations}\label{Ds}
\begin{align}
\Cr{D1}\left(T,\chi_{\eps},V_{\eps},u_{\eps0}\right)&:=\|u_{\eps0}\|_{\LInfn}e^{T\|\chi_{\eps}\|_{L^{\infty}((0,\infty))}\|\nabla\cdot V_{\eps}\|_{L^{\infty}(E_T)}},\label{DsD1}\\
\Cr{D2}\left(T,\chi_{\eps},V_{\eps},u_{\eps0}\right)&:=\|u_{\eps 0}\|_{\LTwon}^2e^{T\Cl{Ce1}\|\chi_{\eps}'\|_{L^{\infty}((0,\infty))}^2\||V_{\eps}|\|_{{L^{\infty}(E_T)}}^2\Cr{D1}\left(T,\chi_{\eps},V_{\eps},u_{\eps0}\right)},\label{constC19}\\
\Cr{D3}\left(T,\chi_{\eps},V_{\eps},u_{\eps0}\right)&:=\Cr{D2}\left(T,\chi_{\eps},V_{\eps},u_{\eps0}\right)\Cl{C26}\max\left\{1,\Cr{D1}\left(T,\chi_{\eps},V_{\eps},u_{\eps0}\right),T\|\chi_{\eps}\|_{L^{\infty}((0,\infty))}^2\||V_{\eps}|\|_{L^{\infty}(E_T)}^2\right\}.
\end{align}
\end{subequations}
\end{Lemma}
\begin{proof}
Since \cref{PMEdriftem} is a non-degenerate quasilinear parabolic PDE, standard theory (see, e.g. \cite[Chapter V, \S6,  Theorem 6.1]{LSU}) implies that the {IBVP} \cref{IBVPem} possesses a unique classical solution $\uem$ provided it is a priori bounded. That is the case if \cref{estbnduem} holds.
Thus, to complete the proof of the Lemma it remains to prove that  estimates \cref{estue} are a priori valid.

Estimates \cref{estbnduem} are a direct application of the maximal principle for non-degenerate parabolic PDEs{, see, e.g. \cite[Chapter 1 \S 2 Theorem 2.1]{LSU}.} 

Further, multiplying equation \cref{PMEdriftem}  by $\varphi\in H^1_0\left(\Bm\right)$ and  integrating by parts using \cref{bcem}, we obtain the following weak reformulation:
\begin{align}
  \left<\partial_t\uem,\varphi\right>=-\int_{B_m}\eps\nabla \uem\cdot\nabla\varphi+\left(\nabla \uem^2-V_{\eps}\uem\chi_{\eps}\left(\uem\right)\right)\cdot\nabla\varphi\,dx\qquad\text{in }(0,T)\ \text{for }\varphi\in H^1_0\left(\Bm\right).\label{weaksolem}
 \end{align}
{Taking $\varphi=\uem$ in \cref{weaksolem}} and  using assumptions on $\chi_{\eps}$, {the chain rule,} and H\"older's and Young's inequalities, we {deduce}
\begin{align}
 &\frac{1}{2}\frac{d}{dt}\|\uem\|_{L^2\left(\Bm\right)}^2+\eps\||\nabla \uem|\|_{L^2\left(\Bm\right)}^2\nonumber\\
 =&-\frac{8}{9}\left\|\left|\nabla \uem^{\frac{3}{2}}\right|\right\|_{L^2\left(\Bm\right)}^2+\frac{2}{3}\int_{\Bm}V_{\eps}\uem^{\frac{1}{2}}\chi_{\eps}\left(\uem\right)\cdot \nabla \uem^{\frac{3}{2}}\,dx\nonumber\\
 \leq&-\frac{1}{2}\left\|\left|\nabla \uem^{\frac{3}{2}}\right|\right\|_{L^2\left(\Bm\right)}^2+\Cr{Ce1}\|\chi_{\eps}'\|_{L^{\infty}((0,\infty))}^2\||V_{\eps}|\|_{{L^{\infty}(E_T)}}^2\|\uem\|_{L^{\infty}\left(0,T;L^{\infty}\left(\Bm\right)\right)}\|\uem\|_{L^2\left(\Bm\right)}^2.\label{esten}
\end{align}
Applying Gr\"onwall's inequality to \cref{esten} yields
\begin{align}
&\|\uem(t,\cdot)\|_{L^2\left(\Bm\right)}^2+2\eps\||\nabla \uem|\|_{L^2\left(0,T;L^2\left(\Bm\right)\right)}^2+\left\|\left|\nabla \uem^{\frac{3}{2}}\right|\right\|_{L^2\left(0,T;L^2\left(\Bm\right)\right)}^2\nonumber\\
 \leq &\|u_{\eps 0}\|_{\LTwon}^2e^{t2\Cr{Ce1}\|\chi_{\eps}'\|_{L^{\infty}((0,\infty))}^2\||V_{\eps}|\|_{{L^{\infty}(E_T)}}^2\|\uem\|_{L^{\infty}\left(0,T;L^{\infty}\left(\Bm\right)\right)}},\qquad\text{for }t\in[0,T].\label{estGr1}
\end{align}
Together, \cref{estbnduem,estGr1} yield \cref{estnabuem}.

{Next, we estimate the right-hand side of \cref{weaksolem} in $H^{-1}\left(\Bm\right)$:}
\begin{align}
 \|\partial_t \uem\|_{H^{-1}\left(\Bm\right)}\leq &\eps\||\nabla \uem|\|_{L^2\left(\Bm\right)}\nonumber\\
 & +\frac{4}{3}\|\uem\|_{L^{\infty}\left(0,T;L^{\infty}\left(\Bm\right)\right)}^{\frac{1}{2}}\left\|\left|\nabla \uem^{\frac{3}{2}}\right|\right\|_{L^2\left(\Bm\right)}\nonumber\\
 &+\|\chi_{\eps}\|_{L^{\infty}((0,\infty))}\||V_{\eps}|\|_{L^{\infty}(E_T)}\|\uem\|_{L^2\left(\Bm\right)}\qquad\text{in }(0,T).\label{est25}
\end{align}
Taking the $L^2(0,T)$ norm on both sides of \cref{est25}{, we deduce}
\begin{align}
 \|\partial_t \uem\|_{L^2\left(0,T;H^{-1}\left(\Bm\right)\right)}
 \leq &\eps\||\nabla \uem|\|_{L^2\left(0,T;L^2\left(\Bm\right)\right)}\nonumber \\
 & +\frac{4}{3}\|\uem\|_{L^{\infty}\left(0,T;L^{\infty}\left(\Bm\right)\right)}^{\frac{1}{2}}\left\|\left|\nabla \uem^{\frac{3}{2}}\right|\right\|_{L^2\left(0,T;L^2\left(\Bm\right)\right)}\nonumber\\
 &+T^{\frac{1}{2}}\|\chi_{\eps}\|_{L^{\infty}((0,\infty))}\||V_{\eps}|\|_{L^{\infty}(E_T)}\|\uem\|_{L^{\infty}\left(0,T;L^2\left(\Bm\right)\right)}.\label{est26}
\end{align}
Combining \cref{est26} with 
\cref{estbnduem,estnabuem} yields \cref{ptuem}.
\end{proof}
Our next Lemma deals with stability in $L^1$. 
\begin{Lemma}\label[Lemma]{LemStabL1m}
Let $\chi_{\eps}$ satisfy \cref{Assregchi}, $V_{\eps},\widehat V_{\eps}$ satisfy \cref{AssregV}, and $u_{\eps 0},\wu_{\eps 0}$ satisfy {\cref{Assregu0},} as well as 
\begin{align*}
m\geq \max\{\diam\left(\supp\left(u_{\eps 0}\right)\right),\diam\left(\supp\left(\wu_{\eps 0}\right)\right)\}.                                                                                                                                                                                                                                          \end{align*}
 Then, the corresponding pair of classical {solutions $\uem$ and $\wuem$ to the IBVP \cref{IBVPem}} satisfies for {$t\in[0,T]$}
 \begin{align}
  &\left\|\left(\uem-\wuem\right)(t,\cdot)\right\|_{L^1\left({\Bm}\right)}\nonumber\\
\leq& \left\|u_{\eps 0}-\wu_{\eps 0}\right\|_{\LOnen}\nonumber\\
&+\Cl{D4_1}\left(T,\chi_{\eps},V_{\eps},\wu_{\eps0}\right)\left(\left\|\uem\nabla\cdot\left(V_{\eps}-\widehat{V}_{\eps}\right)\right\|_{L^1((0,T);L^1(B_m))}+\left\|\uem^{\frac{1}{2}}\left|V_{\eps}-\widehat{V}_{\eps}\right|\right\|_{L^2(0,T;L^2(B_m))}\right),\label{StabestL1m}
 \end{align}
where 
\begin{align}
 \Cr{D4_1}\left(T,\chi_{\eps},V_{\eps},\wu_{\eps0}\right):=\frac{4}{3}\|\chi_{\eps}\|_{C_b^1\left([0,\infty)\right)}\max\left\{1,\Cr{D2}^{\frac{1}{2}}\left(T,\chi_{\eps},V_{\eps},\wu_{\eps0}\right)\right\}\label{Dcontr}
\end{align}
and $\Cr{D2}(\dots)$ is as defined in \cref{Ds}.
\end{Lemma}

\begin{proof} 
The proof follows a standard scheme and relies on a version of Kato's inequality. 
 Let $\uem$ and $\wuem$ be two classical solutions to the IBVP \cref{IBVPem} for  given $V_{\eps}$ and $\widehat{V}_{\eps}$, respectively.
 Their difference
 \begin{align*}
  U_{\eps m}:=\wuem-\uem
 \end{align*}
 satisfies 
 \begin{align}
 \partial_tU_{\eps m}=&\Delta\left(a_0 U_{\eps m}\right)-\nabla\cdot\left(\widetilde{V_{\eps}}U_{\eps m}\right)-b_0\left(\uem\right)\nabla\cdot\left(\widehat{V}_{\eps}-V_{\eps}\right)-\left(\widehat{V}_{\eps}-V_{\eps}\right)b_1\left(\uem\right)\nabla \uem^{\frac{3}{2}} \ \ \text{in }(0,T)\times B_m,\label{PMEdrifted}
\end{align}
 where
 \begin{subequations}\label{b0b1}
 \begin{align}
  a_0:=&\eps+\uem+\wuem>0,\\
  b_0(z):=&z\chi_{\eps}(z),\\
  b_1\left(z\right):=&\frac{2}{3}z^{-\frac{1}{2}}b_0'\left(z\right)\nonumber\\
  =&\frac{2}{3}z^{\frac{1}{2}}\left(\chi_{\eps}'\left(z\right)+\int_0^1\chi_{\eps}'\left(sz\right)\,ds\right),\\
  \widetilde{V}_{{\eps}}:=&\widehat{V}_{\eps}\left(b_0\left(\wuem\right)-b_0(\uem)\right)\nonumber\\
  =&\widehat{V}_{\eps}\int_0^1b_0'\left(s\wuem+\left(1-s\right)\uem\right)\,ds.
 \end{align}
 \end{subequations}
 Multiplying \cref{PMEdrifted} by $\sign\left(U_{\eps m}\right)$, integrating over $\Bm$, by parts where necessary, and using \cref{b0b1}, the chain rule, and Kato's and H\"older's inequalities,  we can estimate as follows in $[0,T]$:
 \begin{align}
  &\frac{d}{dt}\int_{\Bm}|U_{\eps m}|\,dx\nonumber\\
  =&\int_{\Bm}\sign\left({a_0}U_{\eps m}\right)\Delta\left(a_0 U_{\eps m}\right)\,dx-\int_{\Bm}\nabla\cdot\left(\widetilde{V}|U_{\eps m}|\right)\,dx\label{UseKato}\\
  &-\int_{\Bm}\sign\left(U_{\eps m}\right)b_0\left(\uem\right)\nabla\cdot\left(\widehat{V}_{\eps}-V_{\eps}\right)\,dx-\int_{\Bm}\sign\left(U_{\eps m}\right)\left(\widehat{V}_{\eps}-V_{\eps}\right)\cdot b_1\left(\uem\right)\nabla \uem^{\frac{3}{2}}\,dx,\nonumber\\
  \leq&\int_{\Bm}\left|b_0\left(\uem\right)\nabla\cdot\left(\widehat{V}_{\eps}-V_{\eps}\right)\right|\,dx+\int_{\Bm}\left|\left(\widehat{V}_{\eps}-V_{\eps}\right)b_1\left(\uem\right)\nabla \uem^{\frac{3}{2}}\right|\,dx\nonumber\\
  \leq& \|\chi_{\eps}\|_{L^{\infty}((0,\infty))}\left\|\uem\nabla\cdot\left(\widehat{V}_{\eps}-V_{\eps}\right)\right\|_{L^1(B_m)}+\frac{4}{3}\|\chi_{\eps}'\|_{L^{\infty}((0,\infty))}\left\|\left|\nabla \uem^{\frac{3}{2}}\right|\right\|_{L^2\left(\Bm\right)}\left\|\uem^{\frac{1}{2}}\left|V_{\eps}-\widehat{V}_{\eps}\right|\right\|_{L^2(B_m)}.\label{estUb1}
 \end{align}
Here we have used a version of Kato's inequality \cite{Kato}  provided by a combination of  \cite[Theorems 1.2-1.3]{BrezisPonce}:  for all $f\in W^{1,1}_0(\Bm)$ such that $\Delta u\in L^1(\Bm)$ we have $\partial_{\nu}f\in L^1(S_m(0))$ and 
\begin{align}
 \int_{\Bm}\sign(f)\Delta f\,dx\leq &-\int_{S_m(0)}|\partial_{\nu} f|\,dx,\nonumber\\
 \leq&0\label{KatoB}
\end{align}
where $\partial_{\nu}f$ denotes the outer normal derivative of $f$ in the trace sense. It is inequality \cref{KatoB} that allows to estimate the first term on the right-hand side of \cref{UseKato} from above by zero.

Integrating  \cref{estUb1} over $[0,T]$ and using H\"older's inequality yields
\begin{align}
  &\|U_{\eps m}(t,\cdot)\|_{L^1\left(\Bm\right)}-\|U_{\eps m}(0,\cdot)\|_{L^1\left(\Bm\right)}\nonumber\\
  \leq& \|\chi_{\eps}\|_{L^{\infty}((0,\infty))}\left\|\uem\nabla\cdot\left(\widehat{V}_{\eps}-V_{\eps}\right)\right\|_{L^1((0,T);L^1(B_m))}\nonumber\\
  &+\frac{4}{3}\|\chi_{\eps}'\|_{L^{\infty}((0,\infty))}\left\|\left|\nabla \uem^{\frac{3}{2}}\right|\right\|_{L^2(0,T;L^2(B_m))}\left\|\uem^{\frac{1}{2}}\left|V_{\eps}-\widehat{V}_{\eps}\right|\right\|_{L^2(0,T;L^2(B_m))}\nonumber\\
  \leq&\frac{4}{3}\|\chi_{\eps}\|_{C_b^1\left([0,\infty)\right)}\max\left\{1,\left\|\nabla \uem^{\frac{3}{2}}\right\|_{L^2\left(0,T;L^2\left(\Bm\right)\right)}\right\}\nonumber\\
  &\cdot\left(\left\|\uem\nabla\cdot\left(\widehat{V}_{\eps}-V_{\eps}\right)\right\|_{L^1((0,T);L^1(B_m))}+\left\|\uem^{\frac{1}{2}}\left|V_{\eps}-\widehat{V}_{\eps}\right|\right\|_{L^2(0,T;L^2(B_m))}\right)\qquad \text{in }[0,T].\label{estUb2}
 \end{align}
Finally, combining \cref{estUb2} with \cref{estnabuem}, we arrive at \cref{StabestL1m}.
\end{proof}
A comparison property in $L^1$ {holds as well}. 
\begin{Lemma}\label[Lemma]{LemCompL1m}
 Let $\chi_{\eps}$ satisfy \cref{Assregchi}, $V_{\eps}$ satisfy \cref{AssregV}, and $u_{\eps 0},\wu_{\eps 0}$ satisfy \cref{Assregu0} as well as 
\begin{align*}
m\geq \max\{\diam\left(\supp\left(u_{\eps 0}\right)\right),\diam\left(\supp\left(\wu_{\eps 0}\right)\right)\}.                                                                                                                                                                                                                                          \end{align*}
 Then, {any corresponding pair of classical  subsolution $\uem$ and supersolution $\wuem$ to the IBVP \cref{IBVPem} satisfies} 
 \begin{align}
  \left\|\left(\uem-\wuem\right)_+(t,\cdot)\right\|_{\LOnen}\leq \left\|\left(\uem-\wuem\right)_+(0,\cdot)\right\|_{\LOnen}\qquad\text{for }t\in[0,T].
  \nonumber%
 \end{align}
\end{Lemma}
 We omit the proof of \cref{LemCompL1m} since it can be carried out very similar to that of \cref{LemStabL1}.

\bigskip\bigskip
With the results from \cref{LemSolIBVPem,LemStabL1m} at hand, we can now establish existence and  estimates for  \cref{PMEdriftCauchy}. 
\begin{Lemma}\label[Lemma]{LemPMEDriftEx} 
Let {\cref{Asschi,AssV1,AssDV1,Assumpu0}} be satisfied.
Then, the Cauchy problem \cref{PMEdriftCauchy} possesses a unique solution {$u$} in terms of \cref{DefPMEDrift}. The solution satisfies the following set of estimates:
\begin{subequations}\label{uapri}
\begin{alignat}{2}
&\|u\|_{L^{\infty}(E_T)} \leq \Cr{D1}\left(T,\chi,V,u_0\right),\label{estbndu}\\
&\|u \|_{L^{\infty,2}(E_T)}^2+\left\|\left|\nabla u ^{\frac{3}{2}}\right|\right\|_{L^2(E_T)}^2
 \leq \Cr{D2}\left(T,\chi,V,u_0\right),\label{estnabu}\\
&\|\partial_t u \|_{L^2\left(0,T;H^{-1}\left(\R^d\right)\right)}\leq \Cr{D3}\left(T,\chi,V,u_0\right),\label{ptu}
\end{alignat}
\end{subequations}
where $C_i$s are as in \cref{Ds}.
\end{Lemma}
\begin{proof}[Sketch of the proof]
 \cref{LemPMEDriftEx-uniqueness} provides uniqueness, so we only need to prove existence. Since the arguments are standard, we omit most details.
 To begin with, we choose families of regularised coefficient functions $\chi_{\eps}$ and $V_{\eps}$ and initial data $u_{\eps m 0}$ for $$\eps\in(0,1)\text{ and }m\geq2$$ as follows:
\begin{alignat}{3}
&V_{\eps}:=\zeta_{d+1,\eps}\star V,\qquad&& V:=0\text{ in }(-\infty,0)\times\R^d,\label{apprV}\\
&\chi_{\eps}:=\zeta_{1,\eps}\star \left((1-\eta_{1,2\eps})\chi\right),\qquad&& \chi:=0\text{ in }(-\infty,0),\label{apprChi}\\
& u_{\eps m 0}:=\zeta_{d,\eps}\star \left(\eta_{d,\frac{m}{2}}u_0\right),\qquad&&\label{appru0}
 \end{alignat}
 where 
 \begin{alignat}{3}
 &\zeta_{k,\eps}(x):=\eps^{-k}\zeta_1\left(\eps^{-1}|x|\right)\text{ in }\R^k,\qquad&& \zeta_1(s):=\begin{cases}C^{-1}e^{\frac{1}{s^2-1}}& \text{for }|s|<1,\\
 0&\text{for }|s|\geq1,\end{cases}\qquad C:=\int_{-1}^1e^{\frac{1}{s^2-1}}\,ds,
 \label{molli}\\
 &\eta_{k,R}(x):=\eta_1\left(R^{-1}|x|\right)\text{ in }\R^k,&& \eta_1\in C_0^{\infty}([0,1);{[0,1]}),\ \eta_1= 1 \text{ in }\left[0,\frac{1}{2}\right].\label{etaR}
 \end{alignat} 
 This choice and standard properties of convolution  ensure that \cref{Assreg_} is satisfied and for every $\eps\in(0,1)$ and $m\geq2$ it holds that 
 \begin{subequations}\label{convVe}
 \begin{alignat}{3}
 &V_{\eps}\underset{\eps\rightarrow0}{\overset{*}{\rightharpoonup}}V&&\qquad\text{in }\left(L^{\infty}(E_T)\right)^d,\label{convVi}\\
 &V_{\eps}\underset{\eps\rightarrow0}{\rightarrow}V&&\qquad\text{in }\left(L^p(E_T)\right)^d,\ p\in[1,\infty),\label{convV1}\\
 &\nabla\cdot V_{\eps}\underset{\eps\rightarrow0}{\overset{*}{\rightharpoonup}}\nabla\cdot V&&\qquad\text{in }L^{\infty}(E_T),\label{convDVi}\\
 &\nabla\cdot V_{\eps}\underset{\eps\rightarrow0}{\rightarrow}\nabla\cdot V&&\qquad\text{in }L^1(E_T),\label{convDV1}
\end{alignat}
\end{subequations}
\begin{subequations}\label{normVe}
\begin{align}
 \||V_{\eps}|\|_{L^{\infty}(E_T)}\leq &\||V|\|_{L^{\infty}(E_T)},\\
 \|\nabla\cdot V_{\eps}\|_{L^{\infty}(E_T)}\leq &\|\nabla\cdot V\|_{L^{\infty}(E_T)},
\end{align}
\end{subequations}
\begin{align}
 & \chi_{\eps}\to \chi \qquad\text{ in }C^1_b\left(\R^+_0\right),
 \label{conv-chi}
 \end{align}
\begin{subequations}\label{assumpue0}
\begin{alignat}{3}
&\supp(u_{\eps m 0})\subset B_m,\\
 &\|u_{\eps m 0}\|_{\LInfn}\leq \|u_0\|_{\LInfn}&&,\label{ue0binf}\\
&{\underset{m\to\infty}{\lim}\underset{\eps\to0}{\lim}u_{\eps m 0}=u_0}&&\qquad\text{in }\LTwon.\label{conv-u0-L2}
\end{alignat}
\end{subequations}
Since all estimates from \cref{LemSolIBVPem} hold for any $\eps\in(0,1)$ and $2\leq m\in\N$, a standard argument based on compactness and a limit procedure yields the existence of a solution $u$ (defined in terms of \cref{DefPMEDrift}) for the Cauchy problem \cref{PMEdriftCauchy} such that
\begin{align}
{\underset{l\to\infty}{\lim} \underset{k\to\infty}{\lim} u_{\eps_k m_l} =u} \qquad \text{a.e. in } E_T \label{uaeconv}
\end{align}
for some sequences $(\eps_k)$ and $(m_l)$. The solution satisfies estimates \cref{uapri}, which can be  obtained by taking the limit inferior on both sides of each estimate in \cref{estue} and using the weak-$*$ continuity of norms, as well as \cref{assumpue0,conv-chi,normVe}.
\end{proof}

Out next two Lemmas contain some results on continuity with respect to time.
\begin{Lemma}%
 Let the assumptions of \cref{LemPMEDriftEx} be satisfied. Assume additionally that $u_0\in\LOnen$. Then, {the solution $u$ to the Cauchy problem \cref{PMEdriftCauchy} satisfies}
\begin{align*}
u\in C_w\left([0,T];\LInfn\cap \LOnen\right) \end{align*}
 and 
 \begin{align}
 \|u{(t,\cdot)}\|_{\LOnen}\equiv \|u_0\|_{\LOnen}\qquad\text{in }[0,T].\label{LOnepres}
\end{align}
\end{Lemma}
\begin{proof}
Consider ${\eta_{d,n}}$, $n\in\N$,  as defined by \cref{etaR}. 
Then,
\begin{subequations}
\begin{alignat}{3}
&{\eta_{d,n}}\underset{n\to\infty}{\to} 1&&\qquad\text{in }\R^d,\label{convetap}\\
&\eta_{d,n}\underset{n\to\infty}{\overset{*}{\rightharpoonup}} 1&&\qquad \text{in }W^{2,\infty}(\R^d).\label{convetaW2}
\end{alignat}
\end{subequations}
Using $\partial_tu \in L^{\infty}\left(0,T;\left(W^{2,\infty}(\R^d)\right)'\right)$, \cref{veryweaksol} for $\varphi\equiv1$, \cref{convetaW2}, and integrability of $u_0$, we deduce that
\begin{align}
\underset{n\to\infty}{\lim}\iRn u (t,\cdot){\eta_{d,n}}\,dx
=&\left<u(t,\cdot),1\right>\nonumber\\
=&\|u_0\|_{\LOnen}\qquad\text{ for }t\in[0,T].\label{id1}
\end{align}
With \cref{convetap} and \cref{id1} and Fatou's lemma it follows that
\begin{subequations}\label{ineq22}
\begin{align}
 \iRn u (t,\cdot)\,dx\leq&\underset{n\to\infty}{\lim}\iRn u (t,\cdot){\eta_{d,n}}\,dx\\
 =&\iRn u_{0}\,dx\qquad \text{for }t\in[0,T],
\end{align}
\end{subequations}
so that $u\in L^{\infty,1}(E_T)$. Since $\eta_{d,n}\in[0,1]$, we conclude with \cref{ineq22} that \begin{align}
 \|u(t,\cdot)\|_{\LOnen}=\|u_0\|_{\LOnen}\qquad\text{ for }t\in[0,T].\nonumber%
\end{align}

\bigskip
    Finally, we confirm  weak continuity. Since \begin{align*}
&u\in C\left([0,T];\HmOnen\right)\cap  L^{\infty}\left(0,T;\LInfn\cap \LOnen\right),\\
&\LInfn\cap \LOnen\subset \LTwon\subset \HmOnen,                                                                                                                                                                                                                             \end{align*}
 \cite[Chapter 3, \S1, Lemma 1.4]{Temam} applies and yields $$u\in C_w\left([0,T];\LInfn\cap \LOnen\right),$$ 
 as required.
\end{proof}
The next Lemma provides an estimate for the solution change 
over time.
\begin{Lemma}\label[Lemma]{LemContiL1}
 Let the assumptions of \cref{LemPMEDriftEx} be satisfied. 
 Then, the solution $u$ to the Cauchy problem \cref{PMEdriftCauchy} satisfies for all $0\leq\tau<t\leq T$, $R>0$, and 
 \begin{align}
  \varphi\text{ Lipschitz},\quad\varphi(0)=0,\quad  \|\nabla \varphi\|_{\LInfn}\leq 1\label{varphi1}
 \end{align} 
 the estimate
 \begin{align}
  \left|\iRn (u(t,\cdot)-u(\tau,\cdot))\eta_{d,{R}}\varphi\,dx\right|
 \leq(t-\tau)^{\frac{1}{2}}{R}^{\frac{d}{2}}\Cl{C37}\Cr{D3}\left(T,\chi,V,u_0\right),\label{estL1uu0}
 \end{align}
 where $\eta_{d,{R}}$ is the cut-off function from \cref{etaR} and $\Cr{D3}(\dots)$ is as defined in \cref{Ds}. 
\end{Lemma}
\begin{proof}
With \cref{etaR,varphi1}, the chain rule, and  H\"older's inequality we have
\begin{align}
 \left\|\nabla \left(\eta_{d,{R}}\varphi\right)\right\|_{\LTwon}=&\left\|\eta_{d,{R}}\nabla\varphi+\varphi \nabla \eta_{d,{R}}\right\|_{\LTwon}\nonumber\\
 \leq&\left\|\nabla\varphi\right\|_{\LInfn}\left\|\eta_{d,{R}}\right\|_{\LTwon}+R\left\|\nabla\varphi\right\|_{\LInfn}\left\|\nabla\eta_{d,{R}}\right\|_{\LTwon}\nonumber\\
 \leq&{R}^{\frac{d}{2}}\sqrt{2}\left\|\eta_{d,1}\right\|_{\HOnen}\nonumber\\
 =&{R}^{\frac{d}{2}}\Cr{C37}.\label{estphipsi}
\end{align}
Combining \cref{estphipsi,ptu} and H\"older's inequality, we obtain
\begin{align}
 \left|\iRn (u(t,\cdot)-u(\tau,\cdot))\eta_{d,{R}}\varphi\,dx\right|
 \leq&\left|\int_{\tau}^t\left<\partial_tu,\eta_{d,{R}}\varphi\right>\,ds\right|\nonumber\\
 \leq&(t-\tau)^{\frac{1}{2}}\left\|\nabla \left(\eta_{d,{R}}\varphi\right)\right\|_{\LTwon}\|\partial_tu\|_{L^2\left(0,T;\HmOnen\right)}\nonumber\\
 \leq&(t-\tau)^{\frac{1}{2}}{R}^{\frac{d}{2}}\Cr{C37}\Cr{D3}\left(T,\chi,V,u_0\right),\nonumber
\end{align}
as required.
\end{proof}

Next, we establish $L^1$ stability for \cref{PMEdriftCauchy}.
\begin{Lemma}\label[Lemma]{LemStabL1}
Let $\chi$ satisfy \cref{Asschi}, $V,\widehat V$ satisfy \cref{AssV4}, and $u_0,\wu_0$ satisfy  \cref{Assumpu01}. Then, the corresponding solutions $u$ and $\wu$ {to the Cauchy problem \cref{PMEdriftCauchy}} satisfy
 \begin{align}
  &\underset{[0,T]}{\max}\left\|u-\wu\right\|_{\LOnen}\nonumber\\
\leq&\left\|u_0-\wu_0\right\|_{\LOnen}\nonumber\\ 
&+\Cl{D4}\left(T,\chi,V,u_0\right)\left(T\left\|\one_{\{u>0\}}\nabla\cdot\left(V-\widehat{V}\right)\right\|_{L^{\infty,1}(E_T)}+T^{\frac{1}{2}}\left\|\one_{\{u>0\}}\left|V-\widehat{V}\right|\right\|_{L^{\infty}(E_T)\cap L^{\infty,1}(E_T)}\right),\label{StabestL1}
 \end{align}
where 
\begin{align}
 \Cr{D4}\left(T,\chi,V,u_0\right):=\Cr{D4_1}\max\left\{\Cr{D1},\Cr{D1}^{\frac12}\right\}\left(T,\chi,V,u_0\right)\label{l_D4}
\end{align}
and $\Cr{D1}$ and $\Cr{D4_1}$ are as defined in \cref{DsD1} and \cref{Dcontr}, respectively.
\end{Lemma}
\begin{proof} 
 Since solutions to \cref{PMEdriftCauchy} are unique, they can be obtained as limits of  approximations such as introduced in  \cref{apprV,apprChi,appru0,molli,etaR} in the proof of \cref{LemPMEDriftEx}. 
In particular, 
 \begin{align}
  \left\|u_0-\wu_0\right\|_{\LOnen}=\underset{m\to\infty}{\lim}\underset{\eps\to0}{\lim}\left\|u_{\eps m 0}-\wu_{\eps m 0}\right\|_{\LOnen}.\label{convue0L1_}
 \end{align}
 The corresponding solution families $(u_{\eps m})$ and $(\wu_{\eps m})$ satisfy estimates \cref{estue} from \cref{LemSolIBVPem}. Therefore, a standard application of the Lions-Aubin lemma, a diagonal argument, and the above mentioned uniqueness of the limit solutions together imply that
 \begin{align}
  \left\|u-\wu\right\|_{\LOnen}=&\sup_{R>0}\left\|u-\wu\right\|_{L^1(B_R)}\nonumber\\
  =&\sup_{R>0}\underset{m\to\infty}{\lim}\underset{\eps\to0}{\lim}\left\|u_{\eps m}-\wu_{\eps m}\right\|_{L^1(B_R)}\nonumber\\
  \leq&\underset{m\to\infty}{\lim\inf}\underset{\eps\to0}{\lim\inf}\left\|u_{\eps m}-\wu_{\eps m}\right\|_{L^1(\Bm)}\qquad\text{a.e. in }(0,T).\nonumber%
 \end{align}
 
 Further, since assumptions of \cref{LemStabL1m} are satisfied, $(u_{\eps m})$ and $(\wu_{\eps m})$ satisfy estimate \cref{StabestL1m}.  
Passing to the limit inferior on both sides of \cref{StabestL1m} and using 
\cref{convue0L1_,StabestL1_,convV1,convDV1,conv-chi,ue0binf,estbnduem,uaeconv} and the dominated convergence theorem, we arrive at the inequality
 \begin{align}
  \left\|u-\wu\right\|_{\LOnen}
\leq& \left\|u_0-\wu_0\right\|_{\LOnen}+\Cr{D4_1}\left(T,\chi,V,u_0\right)\left(\left\|u\nabla\cdot\left(V-\widehat{V}\right)\right\|_{L^1(E_T)}+\left\|u^{\frac12}\left|V-\widehat{V}\right|\right\|_{L^2(E_T)}\right)\quad\text{a.e.}\label{StabestL1_}
 \end{align}
 Using H\"older's inequality and   interpolation inequalities in $L^p$ spaces on the right-hand side of \cref{StabestL1_} and the fact that $u,\wu\in C_w([0,T];\LOnen)$  finally yields \cref{StabestL1}.
\end{proof}
A comparison property in $L^1$ also holds. 
\begin{Lemma}\label[Lemma]{LemCompL1}
Let $\chi$ satisfy \cref{Asschi}, $V$ satisfy \cref{AssV4}, and $u_0,\wu_0$ satisfy  \cref{Assumpu01}. Then, any corresponding pair of subsolution $u$ and supersolution $\wu$ to the Cauchy problem \cref{PMEdriftCauchy} satisfies
 \begin{align}
  \left\|\left(u-\wu\right)_+(t,\cdot)\right\|_{\LOnen}\leq \left\|\left(u-\wu\right)_+(0,\cdot)\right\|_{\LOnen}\qquad\text{for }t\in[0,T].\nonumber%
 \end{align}
\end{Lemma}
 We omit the proof of \cref{LemCompL1} since it can be carried out very similar to that of \cref{LemStabL1}. 

\subsection{Non-explosion of  support for small times} \label{Subsection:support}
In this Subsection we show that  the support of a solution to \cref{PMEdriftCauchy}
that has a sufficiently small support at the initial time cannot explode on a short time interval. More precisely, our result is as follows.
\begin{Lemma}\label[Lemma]{Lem:support}
Let $\chi$ satisfy \cref{Asschi}, $V$ satisfy \cref{AssV4}, and $u_0$ satisfy  \cref{Assumpu01}. Define
\begin{subequations}
\begin{align}
 &\Cl{Dsupp1}(T,\chi,V):=\frac{1}{4(d+2)}\|(\nabla\cdot V)_-\|_{L^{\infty}(E_T)}\|\chi\|_{L^{\infty}(\Rpositive)},\label{ParSupp1}\\
 &\Cl[Rho]{rho-u0V}(T,\chi,V):=\frac{d+2}{\||V|\|_{L^{\infty}(E_T)}\|\chi'\|_{L^{\infty}(\Rpositive)}},\label{ParSupp2}\\
&\Cl[Rho]{rho4}\left(T,\chi,V,a\right):=\left(1-\frac{\Cr{Dsupp1}(T,\chi,V)}{a}\right)\Cr{rho-u0V}(T,\chi,V),\label{ParSupp3}\\
&\Cl[T]{T-usuppV}(T,\chi,V,
\rho,\aa,\bb)
:= \frac{1}{8\aa}\ln\left(\frac{1}{b}\min\left\{\rho,\Cr{rho4}\left(T,\chi,V,a\right)\right\}\right),\label{Tcsupp}\\
&\rho_{\aa ,\bb }(t):= \bb e^{t4\aa},
\end{align}
where  constants $a$, $b$, $\delta$, and $\rho$ are such that
 \begin{align}
 &0<\rho\leq\Cr{rho-u0V}(T,\chi,V),\\
 &\delta\in(0,1),\\
 &\bb \in (\delta\rho,\rho),\label{Assb}\\
 &\aa>\frac{\Cr{Dsupp1}(T,\chi,V)}{1-\frac{b}{\Cr{rho-u0V}(T,\chi,V)}},\label{AssA}\\
 &\aa\geq \frac{\|u_0\|_{\LInfn}}{\bb ^2-(\delta \rho)^2}.\label{AssA2}
 \end{align}
\end{subequations}
Then,
 \begin{align}
  \supp(u_0)\subset \overline{B_{\delta\rho}}\nonumber%
 \end{align}
 implies for any corresponding supersolution $u:[0,T]\times \oOm\to [0,\infty)$ (in terms of \cref{DefPMEDrift}) to \cref{PMEdriftCauchy} that 
\begin{align*}
&\supp(u(t,\cdot))\subset \overline{B_{\rho_{a,b}(t)}},\qquad \rho_{\aa ,\bb }(t)\in(0,\rho)
\end{align*}
for all $t\in[0,\min\left\{T,\Cr{T-usuppV}(T,\chi,V,
\rho,\aa,\bb)\right\})$.

\end{Lemma}
\begin{proof}
Following a common approach for equations with PM diffusion (see, e.g. \cite[Chapter 14]{Vazquez}), we construct a supersolution to \cref{PMEdrift} that is supported in $B_{\rho}$  and is no smaller than $u$ for $t\in[0,\Cr{T-usuppV}]$. 

 Direct computation shows that
 \begin{align}
  U_{\aa,\bb}(t,x):=& \aa \left(\bb ^2e^{t8\aa }-|x|^2\right)_+\qquad\text{for }(t,x)\in\R\times\R^d\label{Uab}
 \end{align}
is a strong and, hence, a weak and very weak  solution to 
\begin{align}
  \partial_t U=\Delta U^2 +4(d+2)\aa U\quad\text{in }\R\times\R^d\label{PMEAsource}
 \end{align}
and has the following properties:
\begin{align}
 \supp(U_{\aa,\bb}(t,\cdot))=\overline{B_{\rho_{\aa ,\bb }(t)}},\qquad \rho_{\aa ,\bb }(t):= \bb e^{t4\aa }\qquad\text{for }t\in \R,\label{suppAB}
\end{align}
and 
\begin{align}
&U_{\aa,\bb}(t,\cdot)\in \WInfn,\qquad
 \|\nabla U_{\aa,\bb}(t,\cdot)\|_{\LInfn}
 =2\aa \bb e^{t4\aa }\qquad\text{for }t\in\R.\label{estnUBA}
\end{align}
Due to   \cref{suppAB} we have 
\begin{align}
 \rho_{a,b}(t)<\rho\qquad\text{for }t\in
 \left[0,\min\left\{T,\frac{1}{4a}\ln\left(\frac{\rho}{b}\right)\right\}\right),\label{suppUab}
\end{align}
whereas   \cref{Assb,Uab,AssA2} imply 
\begin{align}
 \underset{B_{\delta \rho}}{\inf}U_{\aa,\bb}(0,\cdot)=&\aa\left(\bb ^2-(\delta \rho)^2\right)\geq \|u_0\|_{\LInfn}.\label{estUAB01}
\end{align}
 We are going to show that $U_{\aa,\bb}$ is a supersolution to \cref{PMEdrift} in $[0,\Cl[T]{T_AB}]\times \R^d$ if $\Cr{T_AB}\in(0,T]$ is sufficiently small. Since $U_{\aa,\bb}$ satisfies equation  \cref{PMEAsource}, being a supersolution {to} \cref{PMEdrift} is equivalent to 
 \begin{align}
4(d+2)\aa  U_{\aa,\bb}+\nabla\cdot\left(VU_{\aa,\bb}\chi\left(U_{\aa,\bb}\right)\right)\geq0\qquad \text{a.e. in }(0,\Cr{T_AB}]\times \R^d.\nonumber%
 \end{align}
Using \cref{estnUBA,ParSupp1,ParSupp2} and the assumptions on $V$ and $\chi$, we can estimate as follows:
\begin{align}
 &4(d+2)\aa U_{\aa,\bb}+\nabla\cdot\left(VU_{\aa,\bb}\chi\left(U_{\aa,\bb}\right)\right)\nonumber\\
 =&4(d+2)\aa U_{\aa,\bb}+(\nabla\cdot V)U_{\aa,\bb}\chi\left(U_{\aa,\bb}\right)+V\left(U_{\aa,\bb}\chi'\left(U_{\aa,\bb}\right)+\chi\left(U_{\aa,\bb}\right)\right)\nabla U_{\aa,\bb}\nonumber\\
 \geq&4(d+2)\aa U_{\aa,\bb}-\|(\nabla\cdot V)_-\|_{L^{\infty}(E_T)}\|\chi\|_{L^{\infty}(\Rpositive)}U_{\aa,\bb}\nonumber\\
 &-2\||V|\|_{L^{\infty}(E_T)}\|\chi'\|_{L^{\infty}(\Rpositive)}\|\nabla U_{\aa,\bb}\|_{\LInfn}U_{\aa,\bb}\nonumber\\
 \geq&4(d+2)\aa U_{\aa,\bb}\left(1-\frac{1}{4(d+2)\aa }\|(\nabla\cdot V)_-\|_{L^{\infty}(E_T)}\|\chi\|_{L^{\infty}(\Rpositive)}\right.\nonumber\\
 &\phantom{4(d+2)\aa U_{\aa,\bb}\left(1\right.}\left.-\frac{1}{d+2}\||V|\|_{L^{\infty}(E_T)}\|\chi'\|_{L^{\infty}(\Rpositive)}\bb e^{t4\aa }\right)\nonumber\\
 \geq&4(d+2)\aa U_{\aa,\bb}\left(1-\frac{\Cr{Dsupp1}(T,\chi,V)}{a}-\frac{\bb }{{\Cr{rho-u0V}(T,\chi,V)}}e^{t4\aa }\right)\nonumber\\
 \geq&0\qquad \text{a.e. in }{(0,\min\{T,\Cr{T_AB}({\chi,V},{\aa ,\bb })\}]}\times\R^d,\nonumber
\end{align}
where
\begin{align}
\Cr{T_AB}({\chi,V},{\aa ,\bb }):=&\frac{1}{4\aa }\ln\left(\frac{\Cr{rho4}\left(T,\chi,V,a\right)}{b}\right)>0\nonumber%
\end{align}
due to \cref{AssA,ParSupp3}. 
Consequently, $U_{\aa,\bb}$ is indeed a supersolution to \cref{PMEdriftCauchy} on $[0,\Cr{T_AB}(\dots)]\times\R^d$.

Altogether, we have shown that $U_{\aa,\bb}:[0,\min\left\{T,\Cr{T-usuppV}(\dots)\right\}]\times\R^d\to[0,\infty)$ is a weak supersolution to \cref{PMEdrift}, $U_{\aa,\bb}(0,\cdot)\geq u_0$ holds a.e. due to \cref{estUAB01}, and $U_{\aa,\bb}(\cdot)$ is supported in $B_{\rho_{a,b}(\cdot)}$ due to \cref{suppUab}.
Now we can apply \cref{LemCompL1} and  conclude that $u\leq U_{\aa,\bb}$ a.e. in $[0,\min\left\{T,\Cr{T-usuppV}(\dots)\right\}]\times\R^d$. Hence, $\supp(u(t,\cdot))\subset \overline{B_{\rho_{a,b}(t)}}$ for all $t\in [0,\min\left\{T,\Cr{T-usuppV}(\dots)\right\}]$, as required.  

\end{proof}

\section{Well-posedness of  {\texorpdfstring{\cref{PLnewIBVP}}{}} for fixed \texorpdfstring{$w$}{w}}\label{ExPLnew}
In this Section we study the Cauchy problem \cref{PLnewIBVP} for a fixed $w$. We refer to \cref{SecCh}\cref{PDE1,ChPM} for the discussion of this equation.  Our proofs here rely on the results collected  for the Cauchy problem \cref{PMEdriftCauchy} in \cref{Sec:FixedAdv}. In order to produce a fixed advection direction as in \cref{PMEdrift}, we take
\begin{align}
 V:=\V(w,u):=w(\nabla H\star u).\label{DefVBil}
\end{align} 
In \cref{Subsec:H_V} we establish some useful properties of such $\V$ needed in order to analyse \cref{PLnewIBVP} in \cref{SubSecuEqWP}.

\subsection{Properties of \texorpdfstring{$H$ and $\V$}{H and V}}\label{Subsec:H_V}
In this Subsection we establish some properties of   kernel $H$ from \cref{H} and the bilinear operator $\V$ from \cref{DefVBil}.
\begin{Lemma}[Properties of $H$]\label[Lemma]{LemH}
 Let $F\in C^1(\Rnonnegative)$, and $H$ as defined in \cref{H}. Then,
 \begin{enumerate}
  \item%
  $H\in \WInfn$, $\nabla H$ is supported in $\overline{B_1}$, 
\begin{align}
&\nabla H= -\one_{[0,1]}F(|\cdot|)\sign\qquad\text{ in }B_1\backslash\{0\}\label{DerH},
\end{align}
and
\begin{subequations}
\begin{align}
 \|\nabla H\|_{(\LInfn)^d}=&\max_{[0,1]}|F|,\label{nHnormi}\\
 \|\nabla H\|_{(\LOnen)^d}=&2d\left|B^{d-1}_1\right|\int_0^1|F(s)|s^{d-1}\,ds;\label{nHnorm1}
\end{align}
\end{subequations}
\item%
for $d\geq2$, it holds that $\nabla\nabla^T H\in (\Mn)^{d\times d}$,\begin{align}  
 \nabla\nabla^T H=&
-\left(F'(|\cdot|)\sign\sign^T+\frac{F(|\cdot|)}{|\cdot|}\left(I_d-\sign\sign^T\right)\right)\LL^d\mres (B_1\backslash\{0\})\nonumber\\
& +\left(F(1)\sign\sign^T\right){\cal H}^{d-1}\mres S_1,
 \label{D2H}
 \end{align}
 \begin{align}
 \underset{k,l\in\{1,\dots,d\}}{\max}{\left\|\partial^2_{x_kx_l}H\right\|_{\Mn}} 
 \leq&{\cal H}^{d-1}(S_1)\left(\int_0^1\left(|F'(s)|+\frac{|F(s)|}{s}\right)s^{d-1}\,ds+F(1)\right),\label{estD2H}
 \end{align}
 and
 \begin{align}
 \Delta H=&
-\left(F'(|\cdot|)+(d-1)\frac{F(|\cdot|)}{|\cdot|}\right)\LL^d\mres (B_1\backslash\{0\}) +F(1){\cal H}^{d-1}\mres S_1,
\label{div_D2H}
\end{align}
\begin{align}
\left\|\Delta H\right\|_{\Mn}\leq&{\cal H}^{d-1}(S_1)\left(\int_0^1\left(|F'(s)|+(d-1)\frac{|F(s)|}{s}\right)s^{d-1}\,ds+F(1)\right)
;\label{div_estD2Hp}
 \end{align}
 for $d=1$, it holds that 
\begin{align}
 H''=-F'(|\cdot|)\LL^1\mres (-1,1)-F(0)\delta_0+F(1)(\delta_1+\delta_{-1})\nonumber
\end{align}
and 
\begin{align}
 \|H''\|_{\Mn}=&2\int_0^1|F'(s)|\,ds+F(0)+2F(1)
 .\nonumber%
\end{align}
 \end{enumerate}
\end{Lemma}
\begin{proof}
\begin{enumerate}
  \item
 Consider
 \begin{align*}
G(\tau):=-\int_0^{\tau}\one_{[0,1]}F(s)\,ds\qquad\text{for }\tau\geq0.     \end{align*}
Since $\one_{[0,1]}F\in L^{\infty}(\Rpositive)\cap C(\Rpositive\backslash\{1\})$, we have $G\in W^{1,\infty}(\R)$  and 
\begin{align*}
 G'=-\one_{[0,1]}F\qquad\text{in }\Rpositive\backslash\{1\}.
\end{align*}
Being a superposition of two $W^{1,\infty}$ functions, $G$ and $|\cdot|$, $H$ is then also a $W^{1,\infty}$ function. Using the chain rule and the fact that \begin{align*}
\nabla |x|=\sign(x) \qquad\text{for }x\neq0,                                                                                                                                                                                                        \end{align*}
 we obtain \cref{DerH}. Equality \cref{nHnormi} is an evident consequence of \cref{DerH}. As to \cref{nHnorm1}, it follows from \cref{DerH} by direct computation and the fact that 
 \begin{align}
  \int_{S^1}|x|_1\,dx=&2d\int_{S^1\cap\{x_1>0\}}x_1\,dx\nonumber\\
  =&2d\left|B^{d-1}_1\right|.\nonumber  
 \end{align}

 \item We only consider the case $d\geq2$. Similar arguments work for $d=1$. Calculations are, in fact, simpler in that case. 
 
 Observe that $\nabla H$ is a product of three functions which are locally of bounded variation (BV). Indeed, since $\one_{[0,1]}(|\cdot|)=\one_{\overline{B_1}}$, i.e. is the characteristic function of a set with a smooth boundary of a finite perimeter, $\one_{[0,1]}(|\cdot|)$ is of BV. Hence, its distributional gradient is a finite measure. The gradient consists only of the jump part:
 \begin{align}
  \nabla\one_{[0,1]}(|\cdot|)=-\sign{\cal H}^{d-1}\mres S_1.\label{DerChar}
 \end{align}
This can be obtained, e.g. by using the divergence theorem: 
\begin{align}
 \left<\nabla\one_{[0,1]}(|\cdot|),\varphi\right>=&-\int_{\overline{B_1}}\nabla\cdot\varphi\,dx\nonumber\\
 =&-\int_{S_1}\varphi\cdot x\,d {\cal H}^{d-1}\qquad\text{for }\varphi\in (C_c^1(\R^d))^d.\nonumber
\end{align}
As to $x\mapsto \frac{x}{|x|}$, this function has a locally integrable derivative given by
\begin{align}
 \nabla^T\sign(x)=\frac{1}{|x|}\left(I_d-\sign\sign^T\right)(x)\qquad\text{for }x\neq0.\label{Dersign}
\end{align} 
Hence, it is also locally of BV. Finally, this is also true for $-F(|\cdot|)$ as it is a composition of $C^1$ and $W^{1,\infty}$ functions and as such is locally $W^{1,\infty}$. 

 Further, since $\supp(\nabla H)\subset \overline{B_1}$,  $\nabla\nabla^TH$ is also supported in $\overline{B_1}$. 
Having checked that $\nabla H$ is a product of functions of BV locally, we can apply the product rule for BV functions (a consequence of the chain rule for a smooth outer function, see, e.g. \cite[Chapter 3, \S3.10, Theorem 3.96]{Ambrosio}) in a bounded open ball which contains $\overline{B_1}$. Combined with \cref{DerChar,Dersign}, this theorem yields $\nabla H\in (BV(\R^d))^{{d}}$ and formula \cref{D2H}. Formula \cref{div_D2H} for the divergence of $H$ follows by computing the trace of the density matrices on the right-hand side of \cref{D2H}.

Finally, \cref{div_estD2Hp,estD2H} are  easily obtained by direct calculation from \cref{D2H,div_D2H}, respectively, using standard properties of finite measures. We omit further details.

 \end{enumerate}
 
\end{proof}

\begin{Lemma}\label[Lemma]{LemV1}
Let assumptions of \cref{LemH} be satisfied. 
Let $\V$ be as defined in \cref{DefVBil} and $O\subset\R^d$ be a domain. Then, 
\begin{align}
\V\in &B(\LInfo\times\LOnen;(\LOneo)^d)\nonumber\\
&\cap B(\LInfo\times\LOnen;(\LInfo)^d)\nonumber\\
\nabla\cdot (V(\cdot,\cdot))\in& B(\WInfo\times \LOnen;\LOneo)\nonumber\\
& \cap B(\WInfo\times \LInfn;\LInfo), \nonumber                                                                                                    \end{align}
and 
\begin{subequations}\label{estV}
\begin{align}
\|\V(w,u)\|_{(\LOneo)^{d}}\leq &\Cr{CF_1}\|w\|_{\LInfo}\|u\|_{\LOnen},\label{estV11}\\
 \|\V(w,u)\|_{(\LInfo)^{d}}\leq &\Cr{CF_2}\|w\|_{\LInfo}\|u\|_{\LOnen},\label{estV1}\\
 \|\nabla\cdot(\V(w,u))\|_{\LOneo}\leq &\Cr{CF_3}\|w\|_{\WInfo}\|u\|_{\LOnen},\label{estV21}\\
 \|\nabla\cdot(\V(w,u))\|_{\LInfo}\leq &\Cr{CF_4}\|w\|_{\WInfo}\|u\|_{\LInfn},\label{estV22}
\end{align}
\end{subequations}
where
 \begin{align}
 \Cl{CF_1}:=&\|\nabla H\|_{(\LOnen)^d},\nonumber\\
 \Cl{CF_2}:=&\|\nabla H\|_{(\LInfn)^d},\label{DefCF_2}\\
 \Cl{CF_3}:=&\|\nabla H\|_{(\LOnen)^d}+\|\Delta H\|_{\Mn},\nonumber\\
  \Cl{CF_4}:=&\|\nabla H\|_{(\LInfn)^d}+\|\Delta H\|_{\Mn}.\nonumber
 \end{align}
\end{Lemma}
\begin{proof}
Bilinearity is obvious, so we only verify that $V$  and $\nabla\cdot(V(\cdot,\cdot))$ are well-defined continuous maps and satisfy the required estimates. 

We first study the term $\nabla H\star u$ and its divergence for $u \in L^1_{loc}(\R^d)$ using the properties of the derivatives of $H$ that we established in \cref{LemH}. We may assume that $u$ is Borel. Indeed, otherwise we can redefine it on a set of Lebesgue measure zero to obtain a Borel representative. This would not change the values of $\nabla H\star u$. 
Since $\nabla H\in \LInfn$ and is compactly supported, $\nabla H\star u$ is well-defined. Further, since $\Delta H\in \Mn$ and is compactly supported, formula (4.2.5) from \cite[Chapter 4]{HoermanderI} can be applied, yielding   
\begin{align}
 \nabla\cdot(\nabla H\star u)=\Delta H\star u\label{D2Hu}
\end{align}
 in the distributional sense. 

In view of $\nabla H\in\LOnen\cap\LInfn$,  Young's inequality for convolutions  implies that $\nabla H\star u\in\LOnen\cap\LInfn$ for $u\in\LOnen$ or $u\in\LInfn$, and
\begin{subequations}\label{estD1HL}
\begin{align}
 \|\nabla H\star u\|_{(\LOnen)^{{d}}}\leq& \|\nabla H\|_{(\LOnen)^d} \|u\|_{\LOnen},\label{estD1HL1}\\
 \|\nabla H\star u\|_{(\LInfn)^{{d}}}\leq& \|\nabla H\|_{(\LInfn)^d} \|u\|_{\LOnen},\label{estD1HL2}\\
 \|\nabla H\star u\|_{(\LInfn)^{{d}}}\leq& \|\nabla H\|_{(\LOnen)^d} \|u\|_{\LInfn}.\label{estD1HL3}
\end{align}
\end{subequations}
Next, since $u$ is Borel and $\Delta H$ is a finite vector measure,   \cite[Proposition 3.9.9]{Bogachev} applies and yields $\Delta H\star u\in \LOnen$ for $u\in \LOnen$ and $\Delta H\star u\in \LInfn$ for $u\in \LInfn$, and
\begin{subequations}\label{estD2HL}
\begin{align}
\|\Delta H\star u\|_{\LOnen}\leq& \|\Delta H\|_{\Mn}\|u\|_{\LOnen}, \label{estDiv1}\\
 \|\Delta H\star u\|_{\LInfn}\leq& \|\Delta H\|_{\Mn}\|u\|_{\LInfn}. \label{estDiv2}
\end{align}
\end{subequations}

Now we return to the original bilinear maps. First, assume  $w\in\LInfo$. Then, the product $w(\nabla H\star u)$ is a well-defined function.  Combining \cref{estD1HL1,estD1HL2} with H\"older's inequality, directly yields estimates \cref{estV11,estV1}, respectively.

Let $w\in \WInfo$. 
Due to the chain rule and \cref{D2Hu}, the divergence of $\nabla H\star u$ is well-defined and satisfies
\begin{align}
 \nabla\cdot(w(\nabla H\star u))=&w(\Delta H\star u)+\nabla w\cdot (\nabla H\star u)\qquad\text{a.e. in }O.\label{chainH2u}
\end{align}
Estimating the right-hand side of \cref{chainH2u} using H\"older's inequality and \cref{estDiv1,estD1HL1} (\cref{estDiv2,estD1HL3}), we finally  arrive at \cref{estV21} (\cref{estV22}).

\end{proof}

\begin{Lemma}\label[Lemma]{Lem-Vop}
Let assumptions of \cref{LemH} be satisfied. 
Let $\V$ be as defined in \cref{DefVBil}, $O\subset\R^d$ be a domain, and $T>0$.
Then, 
\begin{align}
\V\in &B(L^{\infty}(\EO)\times L^{\infty,1}(E_T);(L^{\infty,1}(\EO))^d)\nonumber\\
&\cap B(L^{\infty}(\EO)\times L^{\infty,1}(E_T);(L^{\infty}(\EO))^d)\nonumber\\
\nabla\cdot (V(\cdot,\cdot))\in & B(L^{\infty}(0,T;\WInfo)\times L^{\infty,1}(E_T);(L^{\infty}(0,T;\LOneo))^d)\nonumber\\
&\cap B(L^{\infty}(0,T;\WInfo)\times L^{\infty}(E_T);(L^{\infty}(0,T;\LInfo))^d),   \nonumber                                                                                                  \end{align}
and 
\begin{subequations}\label{estVT}
\begin{align}
\|\V(w,u)\|_{(L^{\infty,1}(\EO))^{{d}}}\leq &\Cr{CF_1}\|w\|_{L^{\infty}(\EO)}\|u\|_{L^{\infty,1}(E_T)},\label{estVTone}\\
 \|\V(w,u)\|_{(L^{\infty}(\EO))^{{d}}}\leq &\Cr{CF_2}\|w\|_{L^{\infty}(\EO)}\|u\|_{L^{\infty,1}(E_T)},\label{estVTinf}\\
 \|\nabla\cdot(\V(w,u))\|_{L^{\infty}(0,T;\LOneo)}\leq &\Cr{CF_3}\|w\|_{L^{\infty}(0,T;\WInfo)}\|u\|_{L^{\infty,1}(E_T)},\label{estVTnone}\\
 \|\nabla\cdot(\V(w,u))\|_{L^{\infty}(0,T;\LInfo)}\leq &\Cr{CF_4}\|w\|_{L^{\infty}(0,T;\WInfo)}\|u\|_{L^{\infty}(E_T)},\label{estVTninf}
\end{align}
\end{subequations}
where $C_i$s are as in \cref{LemV1}.
\end{Lemma}
\begin{proof}
 By \cref{LemV1}, $\V$ is a bounded bilinear map between the corresponding time-independent spaces. Hence it is well-defined, bilinear, and bounded as a mapping between the considered  time-dependent spaces and  preserves strong measurability. {Estimates} \cref{estVT} are a direct consequence of the corresponding bounds \cref{estV}. 
\end{proof}
For later use, we define 
\begin{align}
 \Cl{CF_All}:=&\max\{\Cr{CF_3},\Cr{CF_4}\}.\label{labelCF_All}
\end{align}
\subsection{Well-posedness of  {\texorpdfstring{\cref{PLnewIBVP}}{}} {and non-explosion of support}}\label{SubSecuEqWP}
In this Subsection, we establish several results for  the Cauchy problem   \cref{PLnewIBVP} for a fixed $w$. They include local well-posedness  as well as non-explosion of support. Although  \cref{thmmain} requires  compactly supported $u_0$, here we allow for more general initial data. 
\begin{Notation}
To simplify notation, unlike in \cref{Sec:FixedAdv},  we explicitly mention only the parameters $T$ and $S$ for the constants/mappings $C_i$ and sets $U_i$ and $W_i$  introduced in this Subsection.
\end{Notation}
\begin{Theorem}[Local well-posedness of \cref{PLnewIBVP}]\label{Lem:uEqn_FP}
Let \cref{H,Asschi,AssF}
be satisfied and $m,m_{\infty}>0$ be some numbers. For $T\in(0,1]$ and $S>0$ define 
\begin{subequations}
\begin{align}
 \Cl{C-D1}(S):=&2\Cr{CF_All}\|\chi\|_{L^\infty(\Rpositive)}m_{\infty}S,\label{l_C-D1}\\
 \Cl{C-D1_}(S):=&\frac{8}{3}\Cr{CF_All}\|\chi\|_{C_b^1([0,\infty))}\max\left\{1,m^{\frac{1}{2}}m_{\infty}^{\frac{1}{2}}e^{\frac{1}{2}T\Cr{Ce1}\|\chi'\|_{L^{\infty}((0,\infty))}^2\Cr{CF_All}^2m^2m_{\infty}e^{T\Cr{C-D1}(S)}}\right\}\nonumber\\
 &\cdot\max\left\{m_{\infty}e^{T\Cr{C-D1}(S)},m_{\infty}e^{\frac{1}{2}T\Cr{C-D1}(S)}\right\}
,\label{l_C-D1_}\\
\Cl{DULip0}(T,S):=&\frac{1}{1-T^{\frac{1}{2}}S\Cr{C-D1_}(S)},\label{l_DULip0}\\
\Cl{DULip}(T,S):=&\frac{T^{\frac{1}{2}}\Cr{C-D1_}(S)m}{1-T^{\frac{1}{2}}S\Cr{C-D1_}(S)},\label{l_DULip}\\
\Cl[T]{T-ueqnFP}(S):=&\min\left\{1,\frac{\ln 2}{\Cr{C-D1}(S)},\frac{1}{\left(S\Cr{C-D1_}(S)\right)^2}\right\}\label{l_T-ueqnFP}
\end{align}
\end{subequations}
and  
\begin{alignat}{3}
\Cl[U]{U0}:=&\{0\leq u_0\in \LInfn\cap\LOnen:&&\quad \|u_0\|_{\LOnen}\leq m,\quad \|u_0\|_{\LInfn}\leq m_{\infty}\},\nonumber\\
\Cl[U]{SU1}(T):=& \{0\leq u\in L^{\infty,1}(E_T)\cap L^{\infty}(E_T):&&\quad \|u\|_{\LOnen}\leq m, \quad\|u\|_{\LInfn}\leq 2m_{\infty}\quad\text{ a.e. in }(0,T)
\},\label{domU}\\
 \Cl[W]{W1}(T,S):=&\{0\leq w\in L^{\infty}(0,T;\WInfn):&&\quad \|w\|_{L^{\infty}(0,T;\WInfn)}\leq S\},\nonumber%
\end{alignat}
where $\Cr{CF_All}$ is from \cref{labelCF_All} and $\Cr{Ce1}$ is as in \cref{constC19}. Then, if 
\begin{align}
 u_0\in \Cr{U0}\nonumber%
\end{align}
and 
\begin{subequations}\label{AssW1}
 \begin{align}
  &w\in \Cr{W1}(T,S)
 \end{align}
for some
\begin{align}
 &S>0,\\
 &0<T<\Cr{T-ueqnFP}(S),\label{l_T4}
\end{align}
\end{subequations}
then there exists a unique $u$ such that $(u,w)$ is a solution in $[0,T]\times\R^d$ to \cref{PLnewIBVP} (in terms of \cref{DefIPDE}). 

The solution map 
\begin{align}
  \Cl[UU]{OpU1}:\Cr{U0}\times\Cr{W1}(T,S)\to \Cr{SU1}(T),\qquad (u_0,w)\mapsto u,\nonumber%
\end{align}
is well-defined and Lipschitz continuous in the following sense:
\begin{align}
\|\Cr{OpU1}(u_0,w)-\Cr{OpU1}(\wu_0,\ww)\|_{L^{\infty,1}(E_T)}
 \leq
 &\Cr{DULip0}(T,S)\left\|u_0-\wu_0\right\|_{\LOnen}+\Cr{DULip}(T,S)\left\|w-\ww\right\|_{C([0,T];\WInfo)}\label{U-stability}
\end{align}
for any open $O\subset\R^d$ such that  
\begin{align}
 \left\{\Cr{OpU1}(u_0,w)>0\right\}\subset [0,T]\times \overline O.\label{SetO}
\end{align}
\end{Theorem}
\begin{proof}
Let $T$, $S$, and $w$  satisfy \cref{AssW1}, $u_0\in \Cr{U0}$, and $\bar{u}\in \Cr{SU1}(T)$. Then, \cref{AssV4} holds for  
$$V:=\V(\bar{u},w),$$ 
as demonstrated in \cref{Lem-Vop}. Let $u$ be the corresponding solution to \cref{PMEdriftCauchy}  (in terms of \cref{DefPMEDrift}). Its existence and uniqueness is established in \cref{LemPMEDriftEx}. Further, since $u_0\in \LOnen$ is assumed, \cref{LemContiL1} implies that $u$ has the regularity as required in \cref{DefIPDE}. Set $$\Cl[UU]{OpUw}(u_0,w,\bar{u}):=u.$$
 We prove that $\Cr{OpUw}(u_0,w,\cdot)$ is a self map and a contraction in $\Cr{SU1}$.
 To start, we verify that $\Cr{OpUw}(u_0,w,\cdot)$ maps into a bounded subset of $\LInfn$. From estimates \cref{estbndu} and \cref{estVTninf} we obtain 
\begin{align}
\|u\|_{L^\infty(E_T)}\leq& \Cr{D1}(T,\chi,\V(w,\bar{u}),u_0)\nonumber\\
\leq& m_{\infty} e^{T\|\chi\|_{L^\infty(\Rpositive)}\|\nabla\cdot \V(w,\bar{u})\|_{L^{\infty}(E_T)}} \nonumber\\
\leq& m_{\infty} e^{T\Cr{CF_All}\|\chi\|_{L^\infty(\Rpositive)}\|w\|_{L^{1,\infty}(0,T;\WInfn)}\|\bar{u}\|_{L^{\infty}(E_T)}}\nonumber\\
\leq&  m_{\infty} e^{2T\Cr{CF_All}\|\chi\|_{L^\infty(\Rpositive)}m_{\infty}S}\nonumber\\
=&m_{\infty}e^{T\Cr{C-D1}(S)},\label{uEqnFP:selfmap}
\end{align}
with $\Cr{C-D1}$ as defined in \cref{l_C-D1}. Next, we establish a stability property for $\Cr{OpUw}$. 
For $u_0,\wu_0\in\Cr{U0}$, $w,\ww\in \Cr{W1}(T,S)$, and $u,\wu\in {\Cr{SU1}(T)}$ we have the following estimate due to  \cref{LOnepres,StabestL1,estVT} and bilinearity of $\V$:
\begin{align}%
&\left\|\Cr{OpUw}(u_0,w,u)-\Cr{OpUw}(\wu_0,\ww,\wu)\right\|_{L^{\infty,1}(E_T)}-\left\|u_0-\wu_0\right\|_{\LOnen}\nonumber\\ 
\leq & \Cr{D4}(T,\chi,\V(w,u),u_0)\nonumber\\
&\cdot\left(T\left\|\one_{\{u>0\}}\nabla\cdot(\V(w,u)-\V(\ww,\wu))\right\|_{L^{\infty,1}(E_T)}+T^{\frac{1}{2}}\left\|\one_{\{u>0\}}|\V(w,u)-\V(\ww,\wu)|\right\|_{L^\infty(E_T)\cap L^{\infty,1}(E_T)}\right) \nonumber\\
\leq & T^{\frac{1}{2}}\Cr{D4}(\dots)\nonumber\\&\cdot\left(\|\nabla\cdot(\V(w,u-\wu))\|_{L^{\infty,1}(E_T)}+\left\|\one_{\{u>0\}}\nabla\cdot(\V(w-\ww,\wu)\right\|_{L^{\infty,1}(E_T)}\right.\nonumber\\
&\left.\ \ \ +\||\V(w,u-\wu)|\|_{L^\infty(E_T)\cap L^{\infty,1}(E_T)}+\left\|\one_{\{u>0\}}|\V(w-\ww,\wu)|\right\|_{L^\infty(E_T)\cap L^{\infty,1}(E_T)}\right) \nonumber\\
\leq & 2T^{\frac{1}{2}}\Cr{CF_All}\Cr{D4}(\dots)\nonumber\\
&\cdot\left(\|w\|_{L^{1,\infty}(0,T;\WInfn)}\|u-\wu\|_{L^{\infty,1}(E_T)}+m\left\|w-\ww\right\|_{C([0,T];\WInfo)}\right)\nonumber\\
\leq &2T^{\frac{1}{2}}\Cr{CF_All}\Cr{D4}(\dots)S\|u-\wu\|_{L^{\infty,1}(E_T)}+2T^{\frac{1}{2}}\Cr{CF_All}\Cr{D4}(\dots)m\left\|w-\ww\right\|_{C([0,T];\WInfo)}, \label{U1stab_} 
\end{align} 
where $O\subset\R^d$ is any open set such that
\begin{align}
 \{u>0\}\subset (0,T)\times O.\nonumber
\end{align} 
We can further estimate $2\Cr{CF_All}\Cr{D4}(\dots)$ using \cref{l_D4,uEqnFP:selfmap,estVTinf} and the interpolation inequality for $L^p$ spaces:
\begin{align}
2\Cr{CF_All}\Cr{D4}(T,\chi,\V(w,u),u_0)
=&2\Cr{CF_All}\Cr{D4_1}\max\left\{\Cr{D1},\Cr{D1}^{\frac12}\right\}\left(T,\chi,\V(w,u),u_0\right)\nonumber\\
\leq& \Cr{C-D1_}(S),\label{estwithD6}
\end{align}
with $\Cr{C-D1_}$ as defined in \cref{l_C-D1_}. 
Combining \cref{estwithD6,U1stab_}, we obtain
\begin{align}
 &\|\Cr{OpUw}(u_0,w,u)-\Cr{OpUw}(\wu_0,\ww,\wu)\|_{L^{\infty,1}(E_T)}\nonumber\\
 \leq&T^{\frac{1}{2}}S\Cr{C-D1_}(S)\|u-\wu\|_{L^{\infty,1}(E_T)}+T^{\frac{1}{2}}\Cr{C-D1_}(S)m\left\|w-\ww\right\|_{C([0,T];\WInfo)}+\left\|u_0-\wu_0\right\|_{\LOnen}. \label{U1stab}
\end{align}

Altogether, we see from \cref{U1stab,uEqnFP:selfmap,domU,LOnepres} that $\Cr{OpUw}(u_0,w,\cdot)$ is a self-map and a contraction in $\Cr{SU1}(T)$ for any $u_0\in \Cr{U0}$ and 
\begin{align}
 0<T<\Cr{T-ueqnFP}(S),\nonumber%
\end{align}
with $\Cr{T-ueqnFP}$ as defined in \cref{l_T-ueqnFP}.  Equipped with the metric generated by the $\|\cdot\|_{L^{\infty,1}(E_T)}$, $\Cr{SU1}(T)$ is a complete metric space. Hence, we can apply Banach's fixed point theorem, which produces the unique solution to \cref{PLnewIBVP} in $E_T$ (in the sense of \cref{DefIPDE}). We define
\begin{align}
  \Cr{OpU1}:\Cr{U0}\times\Cr{W1}(T,S)\to \Cr{SU1}(T),\qquad\Cr{OpU1}(u_0,w):=u,\nonumber
\end{align}
where $u$ is the first component of the solution corresponding to $u_0$ and $w$. 

It remains to verify \cref{U-stability}. Plugging  $u=\Cr{OpU1}(u_0,w)$ and $\wu=\Cr{OpU1}(\wu_0,\ww)$,  into \cref{U1stab}, we obtain
\begin{align}
 \|\Cr{OpU1}(u_0,w)-\Cr{OpU1}(\wu_0,\ww)\|_{L^{\infty,1}(E_T)}
 \leq&T^{\frac{1}{2}}S\Cr{C-D1_}(S)\|\Cr{OpU1}(u_0,w)-\Cr{OpU1}(\wu_0,\ww)\|_{L^{\infty,1}(E_T)}\nonumber\\
 &+T^{\frac{1}{2}}\Cr{C-D1_}(S)m\left\|w-\ww\right\|_{C([0,T];\WInfo)}+\left\|u_0-\wu_0\right\|_{\LOnen}, \label{U1stab_2}
\end{align}
for $O$ as in \cref{SetO}. 
Since $T^{\frac{1}{2}}S\Cr{C-D1_}(S)<1$, \cref{U1stab_2} yields
\begin{align}
 \|\Cr{OpU1}(u_0,w)-\Cr{OpU1}(u_0,\ww)\|_{L^{\infty,1}(E_T)}
 \leq
 &\frac{T^{\frac{1}{2}}\Cr{C-D1_}(S)m}{1-T^{\frac{1}{2}}S\Cr{C-D1_}(S)}\left\|w-\ww\right\|_{C([0,T];\WInfo)}\nonumber\\
 &+\frac{1}{1-T^{\frac{1}{2}}S\Cr{C-D1_}(S)}\left\|u_0-\wu_0\right\|_{\LOnen},\nonumber
\end{align}
so \cref{U-stability} holds.
\end{proof}
Our next result ensures that the support of a solution to \cref{PLnewIBVP} does not explode over small times.
\begin{Theorem}[Support control for \cref{PLnewIBVP}]\label{IPDEsupp}
 Let the assumptions of \cref{Lem:uEqn_FP} be satisfied. Define 
 \begin{align*}
 \Cl{D_CF2}(S):=&S\frac{1}{4(d+2)}\Cr{CF_4}\|\chi\|_{L^{\infty}(\Rpositive)}m_{\infty},\\
  \Cl{rhoab}(t,S):=&e^{t4\max\left\{\frac{m_{\infty}}{\bb ^2-(\delta \rho)^2},1+\frac{\Cr{D_CF2}(S)}{1-\frac{b}{\Cr{D_rho-u0V}}}\right\}},
 \end{align*}
 and
 \begin{align}
  \Cl[Rho]{D_rho-u0V}:=\frac{d+2}{m\|\chi'\|_{L^{\infty}(\Rpositive)}\|F\|_{L^{\infty}(0,1)}}\nonumber%
 \end{align}
 and let
 \begin{subequations}\label{Assumpu01Rd1}
 \begin{align}
 &u_0\in\Cr{U0},\\
  &\supp(u_0)\subset \overline{B_{\frac{1}{4}\rho}}\qquad \text{for some }\rho\in(0,\Cr{D_rho-u0V}].\label{u0suppcond}
 \end{align}
 \end{subequations}
 Then, there exists a
function
$$\Cr{Tsuppu}
 :(0,\infty)\to(0,\Cr{T-ueqnFP}(S)],$$
with the following property: if 
\begin{subequations}\label{AssSTRo}
 \begin{align}
  w\in \Cr{W1}(T,S)\qquad\text{and}\qquad 0\leq w\leq1\text{ in }[0,T]\times\R^d%
 \end{align}
for some 
\begin{align}
&S>0,\\
&0<T<\Cr{Tsuppu}(S),
 \end{align}
 \end{subequations}
 then the $u$-component of the corresponding unique solution to \cref{PLnewIBVP} in $[0,T]\times\R^d$ satisfies for all $t\in[0,T]$:
 \begin{subequations}\label{compsupp1} 
\begin{align}
&\supp(u(t,\cdot))\subset \overline{B_{\frac{\rho}{2}\Cr{rhoab}(t,S)}},\\
&0<\Cr{rhoab}(t,S)\leq \Cr{rhoab}(T,S)<2.
 \end{align}
\end{subequations}
Here  $\Cr{CF_4}$ is from \cref{LemV1} and  $\Cr{W1}$ and $\Cr{T-ueqnFP}$ are from  \cref{Lem:uEqn_FP}.

\end{Theorem}
\begin{proof}
 Let 
$T$, $S$, and $w$  satisfy \cref{AssSTRo,l_T4}. Then, \cref{Lem:uEqn_FP} provides a unique solution with the first component $u\in \Cr{SU1}(T)$ corresponding to such $u_0$ and $w$. Set
 $$V:=\V(u,w).$$
 Combining  \cref{estVTninf,estVTinf,LOnepres,nHnormi,DefCF_2}, we obtain for the functions defined in \cref{Lem:support}:
\begin{align}
 \Cr{Dsupp1}(T,\chi,V)\leq &\Cr{D_CF2}(S),\nonumber\\
 \Cr{rho-u0V}(T,\chi,V)\geq&\Cr{D_rho-u0V},\nonumber\\
\Cr{rho4}\left(\chi,V,a\right)\geq&\left(1-\frac{\Cr{D_CF2}(S)}{a}\right)\Cr{D_rho-u0V}\nonumber\\
=:&\Cl[Rho]{D_rho4}(S,a).\nonumber
\end{align}
This leads to the following sufficient conditions on $a$, $b$, $\delta$, and $\rho$:
\begin{subequations}\label{abdr}
 \begin{align}
 &0<\rho\leq\Cr{D_rho-u0V},\label{assrho1}\\
 &\delta\in(0,1),\\
 &\bb \in (\delta\rho,\rho),\label{_Assb}\\
 &\aa>\frac{\Cr{D_CF2}(S)}{1-\frac{b}{\Cr{D_rho-u0V}}},\label{_AssA}\\
 &\aa\geq \frac{m_{\infty}}{\bb ^2-(\delta \rho)^2}.\label{_AssA2}
 \end{align}
\end{subequations}
By assumption \cref{u0suppcond} we have \cref{assrho1} satisfied.  Take
\begin{align*}
 \delta:=&\frac{1}{4},\\
 b:=&\frac{1}{2}\rho,\\
 a:=&
\max\left\{\frac{m_{\infty}}{\bb ^2-(\delta \rho)^2},1+\frac{\Cr{D_CF2}(S)}{1-\frac{b}{\Cr{D_rho-u0V}}}\right\},
\end{align*}
so that all conditions \cref{abdr} are met. 
Finally, set
\begin{align}
 \Cl[T]{Tsuppu}(S)
 :=&
\min\left\{\Cr{T-ueqnFP}(S),\frac{1}{8\aa}\ln\left(\frac{1}{b}\min\left\{\rho,\Cr{D_rho4}(S,a)\right\}\right)\right\}
.\nonumber
 \end{align}
Now \cref{Lem:support} applies and yields \cref{compsupp1}.

\end{proof}
The next Theorem provides a bound that controls the solution change in KR norm over small times.
\begin{Theorem}[Local stability in ${\mathcal M}_{KR}$ for \cref{PLnewIBVP}]\label{IPDEconti0}
 Let the assumptions of \cref{IPDEsupp} be satisfied. 
 Then, there exists a
function
$$\Cl{D7}:(0,\infty)\to(0,\infty)$$ with the following property: 
if $T,S$, and $w$ satisfy \cref{AssSTRo}, 
 then the $u$-component of the corresponding unique solution to \cref{PLnewIBVP} in $[0,T]\times\R^d$ satisfies
 \begin{align}
   \|u(t,\cdot)-u_0\|_{\MKB}\leq t^{\frac{1}{2}}\Cr{D7}(S)\qquad\text{for all }t\in[0,T].\label{estuconti}
 \end{align}.
 \end{Theorem}
 \begin{proof}
 Let 
 $T$, $S$, and $w$  satisfy \cref{AssSTRo,l_T4}.  
We apply estimate \cref{estL1uu0} from \cref{LemContiL1} taking  \begin{align*}
&V:=\V(u,w),\\ 
&\tau:=0,\\
&R:=2\rho.                                                                                                                    \end{align*}
 By \cref{etaR} we have $\eta_{d,2\rho}\equiv 1$ in $B_{\rho}$. Combining this observation with \cref{dualKR,estL1uu0}, we  conclude that 
 \begin{align}
  \|u(t,\cdot)-u_0\|_{\MKB}
 \leq&t^{\frac{1}{2}}{(2\rho)}^{\frac{d}{2}}\Cl{C50}\Cr{D3}\left(T,\chi,\V(u,w),u_0\right)\qquad\text{for all }t\in(0,T].\label{estD7}
 \end{align}
Estimating in the same fashion as, e.g. in the proof of \cref{Lem:uEqn_FP}, one readily finds that
\begin{align}
{(2\rho)}^{\frac{d}{2}}\Cr{C50}\Cr{D3}\left(T,\chi,\V(u,w),u_0\right)\leq\Cr{D7}(S)\label{DefD7}
\end{align}
for a suitable function $\Cr{D7}$.
Combining \cref{DefD7,estD7}, 
we arrive at \cref{estuconti}.
\end{proof}

\section{Local well-posedness of \texorpdfstring{\cref{PLnewIBVP}-\cref{LambdaEq}}{} and \texorpdfstring{\cref{M_PLnewIBVP}-\cref{M_LambdaEq}}{} (proofs of  \texorpdfstring{\cref{thmmain,M_thmmain}}{})}\label{SecMainProof}
\begin{Notation}
 To simplify the notation, we do not explicitly mention the dependence on the parameters listed in \cref{thmmain} for constants, mappings, and sets {that} we introduce in the proof below.
\end{Notation}
\begin{Notation}
 We assume that any operator ${\mathcal P}$  acting on a space $V$ of time-independent functions extends to time-dependent functions $v:[0,T]\to V$, $T>0$ in the natural way, i.e.
 \begin{align}
  {\mathcal P}v(t,\cdot):={\mathcal P}(v(t,\cdot)).\nonumber
 \end{align}
\end{Notation}

Finally we are ready to prove  \cref{thmmain} on local well-posedness of system \cref{PLnewIBVP}-\cref{LambdaEq}. The proof is based on decoupling the two equations, solving them separately by means of the tools established in \cref{Sec:wEqn,ExPLnew}, and then applying Banach's fixed point theorem.

 \begin{proof}[Proof of \cref{thmmain}]
  \begin{Step} To begin with, we extend the result of \cref{Thm:wEqn-FP} in order to solve \cref{LambdaEq} in the time-dependent setting.
   For $T>0$ we define
 \begin{alignat*}{5}
  \Cl[U]{SU3}(T):=&\Big\{u\in\Cr{SU1}(T): &&\quad\supp(u(t,\cdot))\subset \oB&&\quad\text{for all }t\in[0,T], \\
  &&&\quad u-\mu_0\in C([0,T];\MKB),&&\\
   &&&\quad \|u(t,\cdot)-\mu_0\|_{\MKB}\leq \Cr{R2}&&\quad\text{for all }t\in[0,T]\Big\},
 \end{alignat*}
and
\begin{align*}
  \Cl[W]{SW3}(T):=&\Big\{w\in C([0,T];W^{1,\infty}(\Br;[0,1])):\quad \|w(t,\cdot)-w_0\|_{\WInfB}\leq \Cr{R1}\quad\text{for all }t\in[0,T]\Big\},
\end{align*}
where $\Cr{R2}$ and $\Cr{R1}$ are from \cref{Thm:wEqn-FP} and $\Cr{SU1}(T)$ is from \cref{Lem:uEqn_FP}. Then,
\begin{alignat}{5}
&u\in\Cr{SU3}(T)&&\qquad\Rightarrow\qquad u(t,\cdot)\in \Cr{SU2}&&\qquad\text{for all }t\in[0,T],\label{solm1}
\\
&w\in\Cr{SW3}(T)&&\qquad\Rightarrow\qquad w(t,\cdot)\in \Cr{W3}&&\qquad\text{for all }t\in[0,T].\nonumber
\end{alignat}
We are given a pair $(\mu_0,w_0)$ that meets the conditions of 
\cref{Thm:wEqn-FP}.
 With \cref{solm1} and  \cref{Thm:wEqn-FP} it follows that
\begin{align*}
 u\in\Cr{SU3}(T)\qquad\Rightarrow\qquad \Cr{OpW1}(u(t,\cdot))\in \Cr{W3}\qquad\text{for all }t\in[0,T]
\end{align*}
and 
\begin{align*}
 \|\Cr{OpW1}(u(t_1,\cdot))-\Cr{OpW1}(u(t_2,\cdot))\|_{\WInfB}\leq \Cr{F-contin-mu}\|u(t_1,\cdot)-u(t_2,\cdot)\|_{L^1(\Br)}\qquad\text{for all }t_1,t_2\in[0,T],
\end{align*}
yielding
\begin{align}
 \Cr{OpW1}:\Cr{SU3}(T)\to \Cr{SW3}(T).\label{oper2}
\end{align}

 \end{Step}

\begin{Step} 
We turn to equation \cref{PLnewIBVP} that we want to solve for $u_0\in\Cr{U01}$ and $w\in \Cr{SW3}(T)$ and ensure that the $u$-component of the  solution belongs to $\Cr{SU3}(T)$. In order to use the theory prepared in  \cref{SubSecuEqWP} for this equation, we first extend $w$ from $\oB$ to the whole of $\R^d$.  
For this purpose, we   introduce an extension operator $\E$ such that
\begin{align*}
 &\E\in L(\WInfB;\WInfn),\\
 &\E w=w\qquad\text{in }\oB,\\
 &\E (W^{1,\infty}(\Br;[0,1]))\subset W^{1,\infty}(\R^d;[0,1]),\\
 &\|\E\|_{L(\WInfB;\WInfn)}< 2.
\end{align*}
Consequently,  
 \begin{subequations}
 \begin{align}
  &\E \in L(C([0,T];\WInfB);C([0,T];\WInfn)),\label{ContiEmbET}\\
  &\E w=w\qquad\text{in }[0,T]\times\oB,\\
  &\E(C([0,T];W^{1,\infty}(\Br;[0,1])))\subset C([0,T];W^{1,\infty}(\R^d;[0,1])),\label{Eemb_2}\\
  &\|\E\|_{L(C([0,T];\WInfB);C([0,T];\WInfn))}=\|\E \|_{L(\WInfB;\WInfn)}< 2.\label{EnormT_}
 \end{align}
 \end{subequations}
Define 
\begin{subequations}\label{mainconst}
 \begin{align}
  \Cl{DS1}:=&2(\|w_0\|_{\WInfB}+\Cr{R1}),\label{constrad}\\
  \Cl[T]{MTT1}:=&\min\left\{\Cr{Tsuppu}(\Cr{DS1}),
  \left(\frac{\Cr{R2}}{2\Cr{D7}(\Cr{DS1})}\right)^{2}\right\},
 \end{align}
 \end{subequations}
 where $\Cr{Tsuppu}$ and $\Cr{D7}$ originate from \cref{IPDEconti0,IPDEsupp}, respectively. 
Due to \cref{rho0}, each $u_0\in \Cr{U01}\subset \Cr{U0}$, where $\Cr{U0}$ is as defined in \cref{Lem:uEqn_FP},   satisfies  \cref{Assumpu01Rd1}. Hence,   \cref{Lem:uEqn_FP,IPDEsupp,IPDEconti0} apply and together with  \cref{mainconst} yield 
\begin{align}
 \Cr{OpU1}(\Cr{U01}\times\Cr{W1}(T,\Cr{DS1}))\subset& \Cr{SU3}(T)\qquad \text{ for }T\in (0,\Cr{MTT1}),\label{oper1_1}
\end{align}
where $\Cr{OpU1}$ and $\Cr{W1}$ are as defined in \cref{Lem:uEqn_FP}. Furthermore, due to \cref{constrad,ContiEmbET},
\begin{align}
 \E (\Cr{SW3}(T))\subset\Cr{W1}(T,\Cr{DS1})\qquad \text{ for }T>0.\label{oper1_2}
\end{align}
From \cref{oper1_1,oper1_2} we deduce
\begin{align}
 \Cr{OpU1}(\Cr{U01}\times\E (\Cr{SW3}(T)))\subset\Cr{SU3}(T)\qquad \text{ for }T\in (0,\Cr{MTT1}).\label{oper1}
\end{align}
\end{Step}
\begin{Step}
In view of  \cref{oper1,oper2}, for any  $T\in (0,\Cr{MTT1})$ the map
\begin{align}
\Cl[SS]{SM1}: \Cr{U01}\times \Cr{SU3}(T)\to \Cr{SU3}(T),\qquad \Cr{SM1}(u_0,u):=\Cr{OpU1}(u_0,\E \Cr{OpW1}(u))\label{DefS1}
\end{align}
is well-defined. 
Let $u_0,\wu_0\in \Cr{U01}$, and $u,\wu\in \Cr{SU3}(T)$. 
Using \cref{U-stability,W-stability,NormComp}, we estimate 
\begin{align}
 &\left\|\Cr{SM1}(u_0,u)-\Cr{SM1}(\wu_0,\wu)\right\|_{L^{\infty,1}(E_T)}-\Cr{DULip0}(T,\Cr{DS1})\left\|u_0-\wu_0\right\|_{\LOnen}\nonumber\\
 \leq
 &\Cr{DULip}(T,\Cr{DS1})\left\|\Cr{OpW1}(u)-\Cr{OpW1}(\wu)\right\|_{C([0,T];\WInfB)}\nonumber\\
 \leq&\Cr{F-contin-mu}\Cr{DULip}(T,\Cr{DS1})\left\|u-\wu\right\|_{C([0,T];\MKB)}\nonumber\\
 \leq&\rho\Cr{F-contin-mu}\Cr{DULip}(T,\Cr{DS1})\left\|u-\wu\right\|_{L^{\infty,1}(E_T)}.\nonumber%
\end{align}
Recalling \cref{l_DULip}, we see that
\begin{align*}
 \rho\Cr{F-contin-mu}\Cr{DULip}(T,\Cr{DS1})\underset{T\to0}{\to}0.
\end{align*}
Consequently, there exists some 
\begin{align*}
 \Cr{Tstar}\in (0,\Cr{MTT1}),
\end{align*}
such that for any $T\in(0,\Cr{Tstar}]$, $u_0,\wu_0\in \Cr{U01}$, and $u,\wu\in \Cr{SU3}(T)$ it holds that
\begin{align}
 &\left\|\Cr{SM1}(u_0,u)-\Cr{SM1}(\wu_0,\wu)\right\|_{L^{\infty,1}(E_{T})}\leq \frac{1}{2}\left\|u-\wu\right\|_{L^{\infty,1}(E_{T})}+\frac{1}{2}\Cl{SMCini}\left\|u_0-\wu_0\right\|_{\LOnen}.\label{SMLip2}
 \end{align}
 Due to \cref{SMLip2}, $\Cr{SM1}(u_0,\cdot)$ is a self-map and contraction in this space for any $u_0\in \Cr{U01}$. Equipped with the metric generated by the $L^{\infty,1}(E_T)$-norm, $\Cr{SU3}(T)$ is clearly a complete metric space. Hence, Banach's fixed point  provides the existence of a unique fixed point 
\begin{align*}
u\in \Cr{SU3}(T).
\end{align*}
Define the solution map
\begin{align}
 \Cl[SS]{SM}:\Cr{U01}\to \Cr{SU3}(T),\qquad \Cr{SM}(u_0):=u.\nonumber
\end{align}
Plugging $u:=\Cr{SM}(u_0)$ and $\wu:=\Cr{SM}(\wu_0)$ into \cref{SMLip2}, we arrive at 
\begin{align}
 \left\|\Cr{SM}(u_0)-\Cr{SM}(\wu_0)\right\|_{L^{\infty,1}(E_{T})}\leq\Cr{SMCini}\left\|u_0-\wu_0\right\|_{\LOnen}.\label{SMLip4}
\end{align}
Together with \cref{W-stability,NormComp}, estimate \cref{SMLip4} leads to 
\begin{align}
 \left\|\Cr{SM}(u_0)-\Cr{SM}(\wu_0)\right\|_{L^{\infty,1}(E_{T})}+\left\|\Cr{OpW1}(\Cr{SM}(u_0))-\Cr{OpW1}(\Cr{SM}(\wu_0))\right\|_{C([0,T];\WInfB)}\leq\Cr{CLipS}\left\|u_0-\wu_0\right\|_{\LOnen}.\label{SMLip3}
\end{align}
Define 
\begin{align*}
 \Cl[SS]{Suw}:\Cr{U01}\to \Cr{SU3}(T)\times \E (\Cr{SW3}(T)),\qquad \Cr{Suw}(u_0):=(\Cr{SM}(u_0),\E\Cr{OpW1}(\Cr{SM}(u_0))).
\end{align*}
Combining our findings so far with \cref{Thm:wEqn-FP,IPDEsupp}, we conclude that for any $T\in(0,\Cr{Tstar}]$ the constructed map $\Cr{Suw}$ is well-defined and produces solutions for \cref{PLnewIBVP}-\cref{LambdaEq} in $[0,T]\times\R^d$ in terms of \cref{DefSol} that satisfy  conditions \cref{Condu,Condw}. Furthermore, by \cref{SMLip3} these solutions also satisfy \cref{mainLip}.

\end{Step}
\begin{Step}
It remains to verify local  uniqueness. For some $u_0\in\Cr{U01}$ and  $T\in(0,\Cr{Tstar}]$, let  $(\wu,\ww)$ be a solution to \cref{PLnewIBVP}-\cref{LambdaEq} in $[0,T]\times\R^d$ that satisfies \cref{AssUniq1} and, hence,
\begin{align}
 \ww\in \Cr{W1}(T,\Cr{DS1}).\nonumber
\end{align}
Then, due to \cref{oper1_1} and the uniqueness part of \cref{Lem:uEqn_FP}, we have 
\begin{align}
 \wu\in \Cr{SU3}(T),\label{wuU}
\end{align}
so that, in particular, \begin{align}
\supp(\wu(t,\cdot))\subset \oB\quad\text{for all }t\in[0,T].  \label{suppwu}                                                                                         \end{align}
Now consider
\begin{align*}
 \widetilde{w}:=\E \left(\ww|_{[0,T]\times\oB}\right).
\end{align*}
In view of \cref{suppwu}, 
\begin{align}
 \widetilde{w}\wu(t,\cdot)=\ww\wu(t,\cdot)\qquad\text{in }\LOnen\text{ for all }t\in[0,T],\nonumber
\end{align}
so that 
\begin{align*}
(\wu,\widetilde{w})\in \Cr{SU3}(T)\times \E (\Cr{SW3}(T))
\end{align*}
is a solution to \cref{PLnewIBVP}-\cref{LambdaEq} in $[0,T]\times\R^d$. Moreover, we can make use of  the uniqueness parts of \cref{Thm:wEqn-FP,Lem:uEqn_FP} to conclude that 
\begin{alignat}{3}
   &\widetilde{w}|_{[0,T]\times\oB}=\Cr{OpW1}\left(\wu\right)&&\qquad\text{in }[0,T]\times\oB,\label{wu}\\
   &\wu=\Cr{OpU1}\left(u_0,\E \left(\widetilde{w}|_{[0,T]\times\oB}\right)\right)&&\qquad\text{a.e. in }[0,T]\times\R^d,\label{uw}
\end{alignat}
respectively. Due to \cref{wu,uw,DefS1,wuU},
\begin{align*}
 \wu=\Cr{SM1}(u_0,\wu)\qquad\text{a.e. in }[0,T]\times\R^d,
\end{align*}
i.e. $\wu$ is the fixed point of the contraction $\Cr{SM1}(u_0,\cdot)$ in $\Cr{SU3}(T)$, hence must coincide with $u$. Finally, yet another application of the uniqueness part of  \cref{Thm:wEqn-FP} yields that $w=\ww$ in $[0,T]\times\oB$, as required.
\end{Step}

\end{proof}
The proof of \cref{M_thmmain} is straight-forward.
\begin{proof}[Proof of \cref{M_thmmain}]
 This is a direct consequence of \cref{CompThmLambdaEq} and the fact that every Lebesgue-measurable function a.e. coincides with a Borel one.
\end{proof}
Finally, let us consider an application of \cref{thmmain}.
\begin{Example}[Existence for $\mu_0$ a point mass]\label{ExSolvPointM}
 For the pair $(\mu_0,w_0)$ from \cref{Exdelta} we already know that it satisfies \cref{Assmu0w0} for any $m>0$ and $\rho\in(0,1)$. Choose some $\rho_0>0$ and let
 \begin{align}
  &u_0\in\LInfn,\qquad \|u_0\|_{\LOnen}=m,\qquad\supp(u_0)\subset \oBo.\nonumber
 \end{align}
Set
\begin{align*}
 u_{0\lambda}(x):=\lambda^du_0(\lambda x)\qquad\text{for }x\in\R^d,\ \lambda>0.
\end{align*}
Then, for $\lambda>0$ it holds that
\begin{subequations}\label{Propu0La}
\begin{align}
  &u_{0\lambda}\in\LOnen,\qquad \|u_{0\lambda}\|_{\LOnen}=m,\qquad \supp(u_{0\lambda})\subset \overline{B_{\lambda^{-1}\rho_0}}
 \end{align}
 and 
 \begin{align}
  \iRn u_{0\lambda}\varphi\,dx-\left<\delta_0,\varphi\right>=&\iRn u(x)(\varphi(\lambda^{-1}x)-\varphi(0))\,dx\nonumber\\
  \leq &m\lambda^{-1}\rho_0\qquad\text{for all }\varphi\in\WInfn\text{ such that }\|\nabla\varphi\|_{\LInfn}\leq1.
 \end{align}
 \end{subequations}
 Now assume that $m$ and $\rho$ satisfy \cref{rho0}. 
In consequence of \cref{Propu0La}, $u_{0\lambda}\in \Cr{U01}$ for $m_{\infty}=\lambda^d\|u_0\|_{\LInfn}$ provided that
\begin{align}
 &\lambda\geq\rho_0\max\left\{\frac{4}{\rho},\frac{2m}{\Cr{R2}}\right\}.\nonumber
\end{align}
Thus, by \cref{thmmain} any essentially bounded compactly supported and suitably rescaled function can serve as an initial value giving rise to a locally unique solution to system \cref{PLnewIBVP}-\cref{LambdaEq}.

\bigskip
The choice of $u_0$ close to $\mu_0=m\delta_0$ has practical relevance. It can, for instance, describe an initial invasion stage of a tumour.
\end{Example}

\section*{Acknowledgement}
\addcontentsline{toc}{section}{Acknowledgement}
\begin{itemize}
\item The authors express their gratitude to the anonymous reviewers for their helpful comments that
contributed to the improvement of the paper.
\item 
The authors were supported by the Engineering and Physical Sciences Research Council [grant number
EP/T03131X/1].
\item For the purpose of open access, the authors have applied a Creative Commons Attribution (CC BY) licence to any Author Accepted Manuscript version arising.         
\item No new data were generated or analysed during this study.
\end{itemize}

\phantomsection
\addcontentsline{toc}{section}{References}
\printbibliography
\end{document}